\documentclass[10pt]{article}
\usepackage[T1]{fontenc}
\usepackage[utf8]{inputenc}
\usepackage[english]{babel}
\usepackage[autostyle, english = american]{csquotes}
\MakeOuterQuote{"}

\usepackage{amsthm}
\theoremstyle{definition}
\newtheorem{Theor}{Theorem}[subsection]
\newtheorem*{thm*}{Theorem}
\newtheorem{Lemma}[Theor]{Lemma}

\newtheorem*{prop*}{Proposition}

\newtheorem{construction}[Theor]{Construction}

\newtheorem{notation}[Theor]{Notation}
\newtheorem{observation}[Theor]{Observation}

\newtheorem{defn}[Theor]{Definition}

\newtheorem{cor}[Theor]{Corollary}
\newtheorem{prop}[Theor]{Proposition}
\newtheorem{ex}[Theor]{Example}
\newtheorem{rem}[Theor]{Remark}

\usepackage{color}

\definecolor{note_color}{rgb}{0.0,0.7,0.0}

\usepackage{mathrsfs}
\usepackage{latexsym}
\usepackage{graphicx}
\usepackage{amscd,amssymb,amsmath,amsbsy,amsfonts}
\usepackage{amsmath}
\usepackage{mathabx}
\usepackage[mathscr]{eucal}
\usepackage[all]{xy}
\usepackage{color}
\definecolor{reference}{rgb}{0.20,0.36,0.74}
\definecolor{citation}{rgb}{0,.40,.80}
\usepackage[colorlinks,linktocpage,urlcolor=blue,linkcolor=reference,citecolor=citation]{hyperref}
\usepackage{verbatim}
\usepackage{enumerate}
\usepackage{cleveref}
\crefformat{equation}{(#2#1#3)}
\usepackage{tikz}
\usepackage{graphicx}
\usepackage[left=2cm, top=2cm, right=2cm, bottom=2cm]{geometry}

\DeclareFontFamily{OT1}{pzc}{}
\DeclareFontShape{OT1}{pzc}{m}{it}{<-> s * [1.200] pzcmi7t}{}
\DeclareMathAlphabet{\mathpzc}{OT1}{pzc}{m}{it}

\usepackage{enumitem}

\newlength{\bibitemsep}
\newlength{\bibparskip}
\let\oldthebibliography\thebibliography
\renewcommand\thebibliography[1]{%
  \oldthebibliography{#1}%
  \setlength{\parskip}{\bibitemsep}%
  \setlength{\itemsep}{\bibparskip}%
}

\makeatletter
\DeclareRobustCommand{\emdef}{%
  \@nomath\em \if b\expandafter\@car\f@series\@nil
  \normalfont \else \bfseries \fi}
\makeatother

\newcommand{\fcat}{\mathscr}

\renewcommand{\phi}{\varphi}
\renewcommand{\epsilon}{\varepsilon}

\DeclareMathOperator{\Sym}{Sym}
\newcommand{\Mod}{\mathrm{Mod}}

\DeclareMathOperator{\gr}{gr}

\DeclareMathOperator{\St}{St}

\DeclareMathOperator{\Der}{Der}

\DeclareMathOperator{\Rep}{Rep}

\DeclareMathOperator{\Spec}{Spec}

\DeclareMathOperator{\Proj}{Proj}

\DeclareMathOperator{\QCoh}{QCoh}

\newcommand{\dR}{{\mathrm{dR}}}
\newcommand{\crys}{{\mathrm{crys}}}

\DeclareMathOperator{\Nm}{Nm}

\DeclareMathOperator{\Hom}{Hom}

\DeclareMathOperator{\End}{End}

\DeclareMathOperator{\Map}{Map}
\DeclareMathOperator{\Fun}{Fun}

\DeclareMathOperator{\Ind}{Ind}

\DeclareMathOperator*{\colim}{colim}
\DeclareMathOperator*{\fib}{fib}

\DeclareMathOperator*{\cofib}{cofib}

\DeclareMathOperator{\Alg}{Alg}
\DeclareMathOperator{\CAlg}{CAlg}

\DeclareMathOperator{\ev}{ev}

\DeclareMathOperator{\Tot}{Tot}

\DeclareMathOperator{\triv}{triv}

\DeclareMathOperator{\Env}{Env}
\DeclareMathOperator{\Fil}{Fil}

\mathchardef\mdef="2D

\DeclareMathOperator{\const}{const}

\usepackage{tikz-cd}
\usepackage[alphabetic]{amsrefs}
\usepackage{hyperref}
\usepackage{bbm}
\usepackage{authblk}
\numberwithin{equation}{section}

\crefalias{prop}{proposition}
\newcommand{\fg}{\mathrm{fg}}
\newcommand{\Op}{\mathrm{Op}}

\newcommand{\sur}{\mathrm{surj}}

\newcommand{\oplax}{\mathrm{oplax}} 

\newcommand{\Spt}{\mathrm{Spt}}
\newcommand{\Spc}{\mathrm{Spc}}
\newcommand{\SSeq}{\mathrm{SSeq}}

\newcommand{\free}{\mathrm{free}}

\newcommand{\LL}{\mathbb L}

\newcommand{\ins}{\mathrm{ins}}

\newcommand{\pd}{\mathrm{pd}}
\newcommand{\LOmega}{\mathrm{L\Omega}}

\newcommand{\DG}{\mathrm{DG}}

\newcommand{\spl}{\mathrm{split}}
\newcommand{\lax}{\mathrm{lax}}
\newcommand{\LGamma}{\mathrm{L\Gamma}}

\newcommand{\op}{\mathrm{op}}
\newcommand{\LLambda}{\mathrm{L\Lambda}}

\newcommand{\B}{\mathrm{B}}
\newcommand{\ds}{\mathrm{ds}}

\newcommand{\forget}{\mathrm{forget}}

\newcommand{\Gr}{\mathrm{Gr}}
\newcommand{\LSym}{\mathrm{LSym}}

\newcommand{\A}{\mathrm{A}}

\newcommand{\D}{\mathrm{D}}
\newcommand{\apoly}{\mathrm{addpoly}}
\newcommand{\epoly}{\mathrm{excpoly}}
\newcommand{\poly}{\mathrm{poly}}
\newcommand{\adic}{\mathrm{adic}}

\newcommand{\sqzero}{\mathrm{sqzero}}
\newcommand{\nonu}{\mathrm{nu}}
\newcommand{\R}{\mathrm{R}}

\newcommand{\Rees}{\mathrm{Rees}}
\newcommand{\coEnv}{\mathrm{coEnv}}

\newcommand{\DAlg}{\mathrm{DAlg}}

\newcommand{\injec}{\mathrm{inj}}

\newcommand{\LT}{\mathrm{LT}}
\newcommand{\aug}{\mathrm{aug}}

\newcommand{\C}{\mathrm{C}}
\newcommand{\add}{\mathrm{add}}
\newcommand{\T}{\mathrm{T}}

\newcommand{\gen}{\mathrm{gen}}

\newcommand{\LEq}{\mathrm{LEq}}

\newcommand{\cat}{\mathrm{Cat}}
\newcommand{\pdred}{\mathrm{pdred}}
\newcommand{\Fin}{\mathrm{Fin}}
\newcommand{\Mnd}{\mathrm{Mnd}}

\newcommand{\Adj}{\mathrm{Adj}}

\DeclareMathOperator{\Id}{Id}

\mathchardef\mdef="2D 

\begin{document}

	\title{\textbf{Divided powers and derived De Rham cohomology}}
	\date{}
	\author{Kirill Magidson}
	\affil{\small{Northwestern University}}
	\affil{\emph{\small{kirill.magidson@northwestern.edu}}}

	\maketitle
	
	\begin{abstract}
		We develop the formalism of derived divided power algebras, and revisit the theory of derived De Rham and derived crystalline cohomology in this framework. We characterize derived De Rham cohomology of a derived commutative algebra $A$ over a base $R$, together with the Hodge filtration on it, in terms of the universal property as the largest filtered divided power thickening of $A$. We show that our approach recovers the classical De Rham cohomology in the case of a smooth map $R\rightarrow A$, and therefore in general, recovers the derived De Rham cohomology in the sense of Illusie. Along the way, we develop some generalities on square-zero extensions and derivations in derived algebraic geometry and apply them to give the universal property of the first Hodge truncation of the derived De Rham cohomology. Finally, we define derived crystalline cohomology relative to a general divided power base, show that it satisfies the main properties of the crystalline cohomology and coincides with the classical crystalline cohomology in the smooth case.
	\end{abstract}

	\tableofcontents

	\section{Introduction and motivation.}

\subsection{Crystalline cohomology and De Rham cohomology.}

Assume that $A$ is a commutative algebra, smooth over some field $k$. The \textbf{De Rham cohomology} of $A$ is a sequence of $k$-vector spaces $H^{n}_{\dR}(A/k)$ defined as the cohomology of the \textbf{De Rham complex} of $A$

$$
\xymatrix{  A \ar[r]^-{d_{\dR}}& \Omega^{1}_{A/k} \ar[r]^-{d_{\dR}} &  \Omega^{2}_{A/k} \ar[r]& \;\;...   \;\;  ,}
$$
where $\Omega^{n}_{A/k}: = \Lambda^{n}_{A} (\Omega^{1}_{A/k})$ is the $n$-th exterior power of the $A$-module of Kahler differentials of $A$ relative to $k$, and $d_{dR}$ is the De Rham differential, i.e. the unique differential extending the universal $R$-linear derivation of $A$, $d_{\dR}: A \rightarrow \Omega^{1}_{A/k}$ and satisfying the graded Leibnitz rule. The importance of De Rham cohomology in algebraic geometry comes from the fact when $k=\mathbb{C}$, De Rham cohomology computes the singular cohomology of the topological space of complex points of $\Spec(A)$. However, over a field of positive characteristic, De Rham cohomology is a less satisfactory invariant, as it consists purely of torsion, and one can not recover characteristic $0$ information from it. Instead, the theory of \textbf{crystalline cohomology} comes into play.

The genesis of crystalline cohomology theory started with a deeper understanding of the De Rham cohomology. In the seminal paper, \cite{Gr} A.Grothendieck found a remarkable geometric interpretation of De Rham cohomology as the cohomology of a certain site. Assume first that $A$ is smooth over a field of characteristic $0$. The \textbf{infinitesimal site} of $\Spec(A)$ is the category of pairs $(U,Z)$ where $U$ is an open affine subscheme of $\Spec(A)$ and $Z$ is a nilpotent thickening of $U$. The infinitesimal site has a structure sheaf $\mathcal{O}$ whose value on the pair $(U,Z)$ is the algebra of functions $\Gamma(Z,\mathcal{O}_{Z})$. Grothendieck showed that over a field of characteristic zero, the De Rham cohomology is equivalent to the cohomology of the structure sheaf of the infinitesimal site.

Now assume $A$ is defined over $\mathbb{F}_{p}$. Based on the observation that in characteristic $0$ De Rham cohomology computes the cohomology of the infinitesimal site, Grothendieck defined a new theory of \textbf{crystalline cohomology} as the cohomology of the so-called \textbf{crystalline site}. The definition of the crystalline site is similar to the definition of the infinitesimal site, in the sense that both sites parametrize certain thickenings of open affine subscheme of $\Spec(A)$. However, a crucial change is implemented in the crystalline case: instead of merely nilpotent thickenings, one now considers \textbf{divided power thickenings}. Let us recall what it means. Assume $I \subset A$ is an ideal in a ring $A$. A \textbf{divided power structure} on $I$ is a sequence of maps $\gamma_{n}: I \rightarrow A$ for $n\geq 1$ satisfying a number of relations ensuring that $\gamma_{n}(x)$ behaves like $\frac{x^{n}}{n!}$ for any $x\in I$ and $n\geq 1$. Over a field of characteristic $0$, any ideal has unique divided powers. But over other base rings, there can be many different divided power structures on an ideal, or none at all. The \textbf{crystalline site} of $\Spec(A)$ is the category whose objects are triples $(U,Z,\gamma)$ where $U$ is an open affine subscheme of $\Spec(A)$, $Z$ is a nilpotent thickening of $U$ over $\mathbb{Z}_{p}$ with a divided power structure $\gamma$ on the ideal defining the closed embedding $U \subset Z$, where divided powers are required to be compatible with the divided powers on the ideal $(p) \subset \mathbb{Z}_{p}$. The scheme $Z$ is referred to as a "divided power thickening" of $U$. The crystalline site has a crystalline structure sheaf which takes any triple $(U,Z,\gamma)$ to the algebra of global functions $\Gamma(Z,\mathcal{O}_{Z})$ on $Z$. The \textbf{crystalline cohomology} $H^{n}_{\crys}(A/\mathbb{Z}_{p})$ is the cohomology of the structure sheaf of the crystalline site.

It turns out that the crystalline cohomology defined as the cohomology of the crystalline site, is closely related to the De Rham cohomology defined in terms of differential forms. A comparison theorem of Grothendieck says that after base change to $\mathbb{F}_{p}$, the crystalline cohomology groups are canonically isomorphic to the \textbf{De Rham cohomology} groups of $A$ over $\mathbb{F}_{p}$:

\begin{equation}\label{crys_mod_p}
H^{n}_{\crys}(A/\mathbb{Z}_{p}) \otimes \mathbb{F}_{p} \simeq H^{n}_{\dR}(A/\mathbb{F}_{p}).
\end{equation}

Another statement concerns the presence of a smooth lift of $\Spec(A)$ to a formally smooth $p$-adic scheme $\Spec \widetilde{A}$. In this situation, the crystalline cohomology of $A$ can be computed as $p$-adically completed De Rham cohomology of the lift:

\begin{equation}\label{crys_lift}
H^{n}_{\crys}(A/\mathbb{Z}_{p}) \simeq H^{n}_{\dR}(\widetilde{A}/\mathbb{Z})_{p}^{\widehat{ \;\;}}.
\end{equation}

These statements can seem surprising at first glance. The right side of Equations \ref{crys_mod_p} and \ref{crys_lift} are defined in very different terms than the left side. The De Rham cohomology is defined as the cohomology of a concrete commutative differential graded algebra. On the other hand, the crystalline cohomology is defined in terms of divided power algebras rather than commutative differential graded algebras. The equivalences \ref{crys_mod_p} and \ref{crys_lift} suggest that there should be a close relationship between commutative differential graded algebras and divided power algebras. This is further affirmed by a classical computation due to L.Illusie \cite{I2}. The formation of De Rham complex $$(A/R) \longmapsto (\Omega^{\bullet}_{A/R}, d_{\dR})$$ on the category of smooth $R$-algebras can be derived to all $R$-algebras to obtain the theory of \textbf{the derived De Rham complex} developed by L.Illusie in the foundational texts \cite{I1,I2}. The derived De Rham complex assigns to any commutative $R$-algebra $A$ a commutative algebra object $|\LL\Omega^{\bullet}_{A/R}|$ in the derived category $\D(R)$ of $R$-modules. Assume $R\rightarrow A$ is a surjective map and the ideal $I = \ker(R\rightarrow A)$ is generated by a regular sequence. Then it is shown in \cite{I2} that there is an equivalence

\begin{equation}\label{dR_pd}
|\LL\Omega^{\bullet}_{A/R}| \simeq \Env^{\pd}_{R}(I),
\end{equation}
where $\Env^{\pd}_{R}(I)$ is the \textbf{divided power envelope} of $R$ at the ideal $I$, i.e. the algebra obtained by freely adjoining to $R$ all divided powers of generators of the ideal $I$. The equivalence \ref{dR_pd} respect an additional piece of structure present on both sides, namely, a filtration. For a smooth map $R \rightarrow A$, the De Rham complex $\Omega_{A/R}^{\bullet}$ has a multiplicative \textbf{Hodge filtration} whose $n$-th stage is the subcomplex

$$
\Omega^{\geq n}_{A/R} : = 0 \rightarrow 0 \rightarrow ... \rightarrow 0 \rightarrow \Omega^{n}_{A/R} \rightarrow \Omega_{A/R}^{n+1} \rightarrow ...
$$
In other words, the subcomplex $\Omega^{\geq n}_{A/R} \subset \Omega^{\bullet}_{A/R}$ is spanned by differential forms $a_{0} da_{1} \wedge da_{2} \wedge ... \wedge da_{i} \in \Omega_{A/R}^{i}$ of degree $ i \geq n$. Deriving this construction, we obtain a "derived" Hodge filtration $\LL\Omega^{\geq \star}_{A/R}$ and thereby provides us a with functor $\LL\Omega_{-/R}: \CAlg_{R,\heartsuit} \rightarrow \Fil^{\geq 0}\CAlg_{R}$, where the target is the $\infty$-category of $\mathbb{E}_{\infty}$-algebra objects in the symmetric monoidal $\infty$-category $\Fil^{\geq 0}\Mod_{R}$ of non-negatively filtered $R$-modules.

The derived Hodge filtration has a counterpart on the divided power side; for a surjective map $R \rightarrow A$ with regular ideal $I$, the divided power envelope $\Env^{\pd}_{R}(I)$ carries a filtration $\Env^{\pd,\geq \star}_{R}(I)$ whose $n$-th stage is spanned by divided power monomials $\gamma_{i_{1}}(x_{1}) .... \gamma_{i_{k}}(x_{k})$ with $i_{1}+...+i_{k} \geq n$ for $x_{i} \in I$.

A classical computation of L.Illusie \cite[Corollary 1.4.4.1]{I2} states that the equivalence \ref{dR_pd} does in fact refine to a filtered equivalence

\begin{equation}\label{dR_pd_fil}
\LL\Omega^{\geq \star}_{A/R} \simeq \Env_{R}^{\pd,\geq \star}(I).
\end{equation}

This paper is an attempt to develop a formalism in which the equivalence \ref{dR_pd_fil} becomes obvious by giving a universal property of the derived De Rham cohomology in the world of divided power algebras. Such a theory ought to make various crystalline comparisons conceptually simpler. Another expectation from such a theory, is that it should make more geometrically clear the relation between derived De Rham cohomology and the stacky approach via \textbf{the De Rham stack} developed recently by Drinfeld and Bhatt-Lurie in the texts \cite{Dr20}, \cite{BL22(2)}, \cite{Bh23}.

Let us now outline our approach to derived De Rham cohomology derived via divided power algebras. Fix a base commutative ring $R$, and let $A$ be a commutative $R$-algebra. The \textbf{De Rham complex of $A$ relative to $R$} is a differential graded $R$-algebra $(\Omega^{\bullet}_{A/R},d_{\dR})$ which has the following universal property. Assume $(B^{\bullet},d)$ is another differential graded $R$-algebra endowed with a commutative $R$-algebra map $f^{0}:A \rightarrow B^{0}$. Then $f^{0}$ uniquely extends to a differential graded $R$-algebra map $f^{\bullet} : (\Omega^{\bullet}_{A/R}, d_{\dR}) \rightarrow (B^{\bullet}, d)$. In other words, the De Rham complex of $A$ relative to $R$ is the universal commutative differential graded $R$-algebra which has $A$ as its degree $0$ component. 

First, let us start by clarifying what we mean by a derived divided power algebra. Let $\DAlg_{\mathbb{Z}}$ be the $\infty$-category of all \textbf{derived rings} as defined by A.Raksit \cite{R}. The $\infty$-category $\DAlg_{\mathbb{Z}}$ is defined as the $\infty$-category of algebras over a certain monad $$\LSym_{\mathbb{Z}}: \Mod_{\mathbb{Z}} \rightarrow \Mod_{\mathbb{Z}}.$$ The monad $\LSym_{\mathbb{Z}}$ is the unique monad which has the properties that its value on a finitely generated free abelian group on $n$ generators is the polynomial ring $\mathbb{Z}[x_{1},...,x_{n}]$, and that it preserves sifted colimits and finite totalizations. Recall that the free polynomial ring on an abelian group $M$ can be defined more conceptually by the formula

$$
\Sym_{\mathbb{Z}}(M):=\bigoplus_{n\geq 0} (M^{\otimes^{n}})_{\Sigma_{n}},
$$
where $(-)_{\Sigma_{n}}$ stands for coinvariants with respect to the symmetric group, and the symmetric group acts on the tensor power $M^{\otimes^{n}}$ by permutation. 

The notion of a \textbf{divided power algebra} is obtained by replacing coinvariants with respect to symmetric groups by \emph{invariants}, i.e. the free divided power algebra on $M$ is given by the formula

$$
\Gamma_{\mathbb{Z}}(M):=\bigoplus_{n\geq 0} (M^{\otimes^{n}})^{\Sigma_{n}}.
$$
A more concrete description of the free divided power ring is as follows. Let us choose a basis $M=\mathbb{Z}^{n}$. Then $\mathbb{Z}\langle x_{1},...,x_{n}\rangle:=\Gamma_{\mathbb{Z}}(M)$ is the smallest subring of $\mathbb{Q}[x_{1},...,x_{n}]$ which contains $\mathbb{Z}[x_{1},...,x_{n}]$ and all \textbf{divided powers} of the generators $x_{i}$, i.e. elements of the form $$\frac{x_{i}^{d}}{d!}\in \mathbb{Q}[x_{1},...,x_{n}].$$

To promote the construction $M \longmapsto \Gamma_{\mathbb{Z}}(M)$ to a monad, one has to work with non-unital algebras\footnote{this has to do with the fact that outside of characteristic $0$, the unit in a unital ring never supports divided powers.}. The ring $\mathbb{Z}\langle x_{1},...,x_{n}\rangle$ is augmented, and the augmentation ideal $\mathbb{Z}\langle x_{1},...,x_{n}\rangle^{+}$ is a non-unital ring, which we call the \textbf{non-unital polynomial divided power ring} on $n$ generators. We will denote $\Gamma_{\mathbb{Z}}^{+}(\mathbb{Z}^{n}):=\mathbb{Z}\langle x_{1},...,x_{n}\rangle$. There exists the \textbf{non-unital derived divided power algebra} monad on $\Mod_{\mathbb{Z}}$, denoted $$\LGamma_{\mathbb{Z}}^{+}:\Mod_{\mathbb{Z}} \rightarrow \Mod_{\mathbb{Z}},$$ which similarly to $\LSym_{\mathbb{Z}}$, is defined as the unique sifted colimit and finite totalization preserving monad whose value on a finitely generated free abelian group on $n$ generators is $\Gamma_{\mathbb{Z}}^{+}(\mathbb{Z}^{n})$. We denote $$\DAlg_{\mathbb{Z}}^{\nonu,\pd}:=\Alg_{\LGamma^{+}_{\mathbb{Z}}}(\Mod_{\mathbb{Z}})$$ the $\infty$-category of non-unital derived divided power rings. These definitions can be extended to any ground derived ring $R$, thus giving us the $\infty$-category $\DAlg_{R}$ of derived $R$-algebras, and the $\infty$-category $\DAlg_{R}^{\nonu,\pd}$ of non-unital derived divided power $R$-algebras.

A global version of the notion of a non-unital derived divided power ring is the notion of a \textbf{derived divided power map}, or a \textbf{derived divided power Smith ideal}. Roughly, a derived divided power map $A\rightarrow B$ is the data of a map of derived rings $A\rightarrow B$ together with the data of a non-unital derived divided power $A$-algebra on $I=\fib(A \rightarrow B)$. To explain it more precisely, let us digress for a moment and delve into the notion of ideal in the derived world. Let $\DAlg_{\mathbb{Z}}^{\Delta^{1}}$ be the $\infty$-category of arrows $A \rightarrow B$ in the $\infty$-category of derived rings. $\DAlg_{\mathbb{Z}}^{\Delta^{1}}$ is monadic over the arrow category $\Mod^{\Delta^{1}}_{\mathbb{Z}}$. Consider the functor $\fib: \Mod_{\mathbb{Z}}^{\Delta^{1}} \rightarrow \Mod_{\mathbb{Z}}^{\Delta^{1}}$ sending an arrow $(X\rightarrow Y)$ to the fiber $(\fib(X \rightarrow Y) \rightarrow X$. This functor is evidently an equivalence. Therefore, the data of a derived ring map $(A\rightarrow B)$ can be equivalently characterized as the data of a derived ring $A$, together with a certain data on the map $I=\fib(A \rightarrow B) \rightarrow A$. We refer to this data as the data of a \textbf{derived Smith ideal}, and explain in more details what it amounts to in the text. A part of the data of a Smith ideal $(I \rightarrow A)$ is the structure of a non-unital derived $A$-algebra on $I$. We define a \textbf{derived divided power map} $(A\rightarrow B)$  as a map of derived rings $(A\rightarrow B)$, together with a lift of $I$ to a non-unital derived divided power $A$-algebra. We call the corresponding derived Smith ideal $(I \rightarrow A)$ a \textbf{derived divided power Smith ideal}, and sometimes simply \textbf{derived divided power algebra} for short. The notion can be extended to any ground derived ring $R$.  We denote $\DAlg^{\Delta^{1},\pd}_{R}$ the $\infty$-category of derived divided power maps over $R$, and reserve a separate notation $\DAlg^{\Delta^{1}_{\vee},\pd}_{R}$ for the $\infty$-category of derived divided power Smith ideals over $R$. We freely interchange between the two notions, as they are equivalent, but use separate notations depending on whether we want to emphasize the data of a derived Smith ideal or not.

Furthermore, the theory of derived divided power Smith ideals admits a filtered analogue. Let $R$ be a ground derived ring, and let $\Fil^{\geq 0}\Mod_{R}$ be the $\infty$-category of non-negatively filtered objects in $\Mod_{R}$, and $\Fil^{\geq 1}\Mod_{R} \subset \Fil^{\geq 0}\Mod_{R}$ the full subcategory of positive filtrations. We define an $\infty$-category of \textbf{non-negatively filtered derived $R$-algebras} $\Fil^{\geq 0}\DAlg_{R}$ as the $\infty$-category of algebras over a certain filtered version $\LSym^{\geq \star}_{R}: \Mod_{R} \rightarrow \Mod_{R}$ of the $\LSym$-monad. For any object $A^{\geq \star}\in \Fil^{\geq 0}\DAlg_{R}$, the $0$-th stage of the filtration $A^{\geq 0}$ is a derived $R$-algebra, and the "augmentation fiber" $A^{\geq 1}$ is a non-unital filtered derived $A^{\geq 0}$-algebra. Now the monad $\LGamma_{R}^{+}:\Mod_{R} \rightarrow \Mod_{R}$ admits a filtered enhancement $\LGamma_{R}^{+,\geq \star}: \Fil^{\geq 1}\Mod_{R} \rightarrow \Fil^{\geq 1}\Mod_{R}$, thereby giving as the notion of a \textbf{non-unital filtered derived divided power $R$-algebra}. We define an $\infty$-category $\Fil^{\geq 0} \DAlg^{\pd}_{R}$ of \textbf{(non-negatively) filtered derived divided power $R$-algebras} as an $\infty$-category monadic over $\Fil^{\geq 0}\Mod_{R}$, whose object consist of objects $A^{\geq \star}$ in $\Fil^{\geq 0}\DAlg_{R}$ together with a $A^{\geq 0}$-linear non-unital filtered derived divided power structure on the augmentation fiber $A^{\geq 1}$. This monad is trivial in characteristic $0$, i.e. when $R$ has characteristic $0$, the forgetful functor induces an equivalence $\Fil^{\geq 0}\DAlg_{R}^{\pd} \simeq \Fil^{\geq 0}\DAlg_{R}$. In non-zero characteristic, the data of divided powers on a discrete object $A^{\geq \star} \in \Fil^{\geq 0}\CAlg^{\pd}_{R,\heartsuit}$ roughly amounts to specifying for any element $x\in A^{\geq i}$ in filtered stage $i\geq 1$ a sequence of elements $\gamma_{n}(x) \in A^{\geq ni}$ for all $n\geq 1$, which satisfy the usual divided power relations.

There is a functor of taking $0$-th graded piece $\gr^{0}: \Fil^{\geq 0}\DAlg^{\pd}_{R} \rightarrow \DAlg_{R}$. It is not easy to see that it preserves all limits and hence admits a left adjoint.

\begin{defn}\label{main_definition_0}
The left adjoint of the functor $\gr^{0}$ $$\LOmega_{-/R}^{\geq \star}: \DAlg_{R} \xymatrix{\ar[r]&} \Fil^{\geq 0}\DAlg^{\pd}_{R} $$ sends a derived $R$-algebra $A$ to the \textbf{Hodge-filtered derived De Rham cohomology} $\LOmega^{\geq \star}_{A/R}$ of $A$ relative to $R$. 

\end{defn}

This definition requires justification. Recall that there already exists a functor $\LL\Omega_{-/R}: \CAlg_{R,\heartsuit} \rightarrow \Fil^{\geq 0}\CAlg_{R}$ with values in non-negatively filtered $\mathbb{E}_{\infty}$-algebras obtained by resolving any $R$-algebra by polynomial algebras, and applying the functor of Hodge-filtered De Rham complex termwise. Extending this to all connective derived algebras by commuting with arbitrary geometric realizations, we obtain a functor $\LL\Omega^{\geq \star}_{-/R}: \DAlg_{R,\geq 0} \rightarrow \Fil^{\geq 0}\CAlg_{R}$. To justify the Definition \ref{main_definition_0}, we need to show that the following triangle commutes

\begin{equation}\label{tr}
\xymatrix{  \DAlg_{R,\geq 0}  \ar[rrd]_-{ \LL\Omega^{\geq \star}_{-/R}  } \ar[rr]^-{\LOmega^{\geq \star}_{-/R}}&& \Fil^{\geq 0}\DAlg^{\pd}_{R} \ar[d]\\
&& \Fil^{\geq0} \CAlg_{R},   }
\end{equation}
where the right vertical functor is the forgetful functor. To prove this statement, we need a little more preliminary formalism. But already at this point, we notice that there is at least one immediate corollary of the definition.

\begin{prop}\label{first_comparison}
Let $R$ be a discrete commutative ring, and $R \rightarrow A$ be a surjective map such that the ideal $I=\ker(R \rightarrow A)$ is generated by a regular sequence, and let $\Env^{\pd,\geq \star}_{R}(I)$ be the divided power envelope of $R$ at $I$ endowed with the divided power filtration. Then there is an equivalence of filtered $R$-algebras

$$
\LOmega^{\geq \star}_{A/R} \simeq \Env^{\pd,\geq \star}_{R}(I).
$$

\end{prop}

To prove that the diagram \ref{tr} commutes, by commutation with sifted colimits it suffices to show that there is a natural equivalence $\LOmega_{A/R}^{\geq \star} \simeq \Omega^{\geq \star}_{A/R}$ for any polynomial $R$-algebra $A$. To formulate this comparison statement, we need to connect our definition of derived De Rham cohomology in terms of filtered divided power algebras with the more familiar definition of the De Rham cohomology as a cochain complex. Passing to the language of cochain complexes can be done as in \cite{R} using the equivalence between the complete filtered derived category and the $\infty$-category of \textbf{dg modules}, or \textbf{coherent cochain complexes}. Namely, let $\widehat{\Fil\Mod}_{R} \subset \Fil\Mod_{R}$ be the full subcategory of complete objects. There is a functor of taking \textbf{associated graded object} $$\gr: \widehat{\Fil\Mod}_{R} \xymatrix{\ar[r]&} \Gr\Mod_{R},$$ and it turns out that it is monadic. We let $\DG_{-}\Mod_{R} $ be the resulting monadic $\infty$-category over $\Gr\Mod_{R}$, so that the functor $\gr$ lifts to an equivalence $\gr: \widehat{\Fil\Mod}_{R} \simeq \DG_{-}\Mod_{R}.$ We define an $\infty$-category $\DG_{-}\DAlg^{\pd}_{R}$ of \textbf{derived differential graded divided power algebras} by requiring that the functor $\gr$ lifts to an equivalence $\gr: \widehat{\Fil\DAlg}^{\pd}_{R} \simeq \DG_{-}\DAlg^{\pd}_{R}.$

\begin{defn}\label{main_definition_01}
The functor of \textbf{derived De Rham complex} \begin{equation}\label{main_definition_1}\LOmega^{\bullet}_{A/R}: \DAlg_{R} \xymatrix{\ar[r]&} \DG_{-}\DAlg^{\pd}_{R}\end{equation} is the left adjoint of the limit preserving functor $\ev^{0}: \DG_{-}\DAlg^{\pd}_{R} \rightarrow \DAlg_{R}$ extracting the $0$-th degree part of a differential graded derived divided power algebra. Given a derived $R$-algebra $A$, the derived De Rham complex $\LOmega^{\bullet}_{A/R}$ can be computed as the composition $\LOmega^{\bullet}_{A/R} :=\gr \; \LOmega^{\geq \star}_{A/R}$. 
\end{defn}

The $\infty$-category $\DG_{-}\Mod_{R}$ has a so-called \textbf{negative $t$-structure}, whose heart is the abelian category of cochain complexes. What we precisely show is that for a smooth map $R\rightarrow A$, the derived De Rham complex $\LOmega^{\bullet}_{A/R}$ is discrete in this $t$-structure, and the corresponding cochain complex is the classical De Rham complex (with the structure of a commutative differential graded algebra). A precise formulation of this statement uses the décalage isomorphism of Illusie, \cite[Proposition 4.3.2.1]{I2}.

\begin{prop}\label{first_comparison}
Let $R$ be a discrete commutative ring, and $A$ be a smooth $R$-algebra, and $\Omega^{\bullet}_{A/R}$ be the De Rham complex of $A$ over $R$. Then the derived De Rham complex $\LOmega^{\bullet}_{A/R}$ is discrete in the negative $t$-structure, and there is an equivalence of commutative differential graded $R$-algebras

$$
\LOmega^{\bullet}_{A/R} \simeq \Omega^{\bullet}_{A/R}.
$$

\end{prop}

\begin{rem}
At this point, let us clarify how the universal property given above is different from the Raksit's. The dichotomy lies in what kind of algebras one considers as a target for the derived De Rham complex functor, and there are essentially two main points of difference.

- We have introduced an $\infty$-category $\DG_{-}\Mod_{R}$ such that the functor $\gr: \Fil\Mod_{R} \rightarrow \Gr\Mod_{R}$ lifts to an equivalence $\gr: \widehat{\Fil\Mod}_{R} \xymatrix{ \ar[r]^-{\sim}&} \DG_{-}\Mod_{R}$. In Raksit's terminology, objects of $\DG_{-}\Mod_{R}$ are called \textbf{$h_{-}$-differential graded objects}. There also exists a notion of \textbf{$h_{+}$-differential graded objects}. Namely, following \cite{R}, one defines a graded algebra $\mathbb{D}_{+}:=R\oplus R[1](1)$ and the $\infty$-category of $h_{+}$-differential graded objects as $\DG_{+}\Mod_{R}:= \Mod_{\mathbb{D}_{+}}(\Gr\Mod_{R})$. Raksit gives a simple universal property of the derived De Rham complex which is analogous to the one given in Definition \ref{main_definition_01}, but in the setting of $h_{+}$-differential graded algebras, rather than in the setting of filtered derived algebras or $h_{-}$-differential graded algebras.  

- The theory of the derived De Rham complex in the setting of $h_{+}$-differential graded algebras turns out to be very different from the setting of filtered derived algebras or $h_{-}$-differential graded algebras. In the $h_{+}$-setting, there are no divided powers. In contrast, in the filtered/$h_{-}$-setting, divided powers become crucial. The main reason why one needs divided powers, is that the left adjoint of the functor $\ev^{0}: \DG_{-}\DAlg_{R} \rightarrow \DAlg_{R}$ is \emph{not} the derived De Rham complex, but is rather related to the \emph{infinitesimal cohomology}. 

\end{rem}

One important property of (derived) De Rham cohomology is its relation to the notions of the universal derivation and the universal square-zero extension. This can be explained by means of the following two examples. Let $R$ be a discrete commutative ring, and $A$ a smooth $R$-algebra. Then the naive truncation $\Omega_{A/R}^{\{0,1\}}$ of the De Rham complex $\Omega^{\bullet}_{A/R}$ in degrees $\leq 2$ coincides with the two-step complex $d:A \rightarrow \Omega_{A/R}^{1}$, where the map $d$ is the \textbf{universal (De Rham) derivation} of the map $R \rightarrow A$. its universality is that for any other $R$-linear derivation $\delta: A \rightarrow M$ with values in an $A$-module $M$, there is a unique map $f: \Omega^{1}_{A/R} \rightarrow M$ such that the derivation $\delta$ factors as $\delta=f\circ d$. Another example is the case of a surjective map $R\rightarrow A$ whose ideal $I$ is generated by a regular sequence. As we mentioned above, there is an equivalence $\LOmega^{\geq \star}_{A/R} \simeq \Env_{R}^{\geq \star}(I)$. Truncating the latter filtered object in degrees $\leq 2$, we obtain a square-zero extension 

$$
I/I^{2} \rightarrow R/I^{2} \rightarrow R/I.
$$
This square-zero extension also has a universal property. Suppose we have a map $(I \rightarrow R) \rightarrow (J \rightarrow S)$ of algebras with an ideal, and $J^{2}=0$. Then this map factors uniquely through a map $(I/I^{2} \rightarrow R/I^{2}) \rightarrow (J\rightarrow S) $. For this reason, we might call the object $I/I^{2} \rightarrow R/I^{2} $ the \textbf{universal square-zero extension} associated to the ideal $I \subset R$. The two examples we just explained suggest two ideas: 

1) truncation of the derived De Rham cohomology (or derived De Rham complex) in degrees $\leq 2$ must have a universal property as a universal square-zero extension (respectively, the universal derivation), 

2) there should exist a canonical way of transforming the universal square-zero extension into the universal derivation, and vice versa.

The second idea is in fact an instance of one of the foundational ideas of derived algebraic geometry: square-zero extensions are controlled by derivations. The passage from square-zero extensions to derivations is done by taking associated gradeds for two-step filtration existing on any square-zero extension, and interpreting the extension data as a derivation. A reasonable expectation then is that the equivalence of square-zero extensions and derivations is a specialization of the equivalence $\widehat{\Fil\Mod}_{R} \simeq \DG_{-}\Mod_{R}$ to objects concentrated in degrees $0,1$ only. Indeed, consider subcategories $\Fil^{[0,1]}\Mod_{R} \subset \Fil^{\geq 0}\Mod_{R}$ of \textbf{two-step filtrations} and $\DG_{-}^{\{0,1\}}\Mod_{R} \subset \DG^{\geq 0}\Mod_{R}$ of \textbf{differential two-step graded modules}. We show that $\Fil^{[0,1]}\Mod_{R}$ and $\DG_{-}^{\{0,1\}}\Mod_{R}$ are symmetric monoidal $\infty$-categories and there is a symmetric monoidal equivalence $$\Fil^{[0,1]}\Mod_{R} \simeq \DG_{-}^{\{0,1\}}\Mod_{R}.$$ 

Moreover, we interpret the $\infty$-category $\Fil^{[0,1]}\DAlg_{R}:=\DAlg(\Fil^{[0,1]}\Mod_{R})$ as the $\infty$-category of \textbf{derived square-zero Smith ideals} and $\DG_{-}^{\{0,1\}}\DAlg_{R}:=\DAlg(\DG_{-}^{\{0,1\}}\Mod_{R})$ as the $\infty$-category of \textbf{derived derivations}. For instance, objects of the latter can be thought as triples $(A,M,\delta)$ where $A\in \DAlg_{R}$, $M \in \Mod_{A}$, and $\delta: A \rightarrow A\oplus M[1]$ is a derivation. There exists a symmetric monoidal functor $$\xymatrix{   \Fil^{\geq 0}\Mod_{R} \ar[r]^-{\gr}& \DG_{-}^{\geq 0}\Mod_{R} \ar[r]^-{\{0,1\}}& \DG_{-}^{\{0,1\}}\Mod_{R}   } $$ which induces a functor $\Fil^{\geq 0}\DAlg_{R} \rightarrow \DG^{\{0,1\}}_{-}\DAlg_{R}$, which produces a universal derivation for any non-negatively filtered derived algebra. We show the following:

\begin{Theor}\label{another_comparison}
Let $R\rightarrow A$ be a map of derived rings, and $\LOmega^{\geq \star}_{A/R}$ be the Hodge filtered derived De Rham cohomology of $A$ over $R$. Then the differential two-step graded object $ \LOmega^{\{0,1\}}_{A/R}$ is canonically equivalent to the universal derivation $d: A \rightarrow A \oplus \LL_{A/R}$.
\end{Theor}

This statement is to be expected, and its proof rather tautologically follows by checking universal properties. The essential content of the discussion above is Theorem \ref{sqzero_as_der_1} which interprets the equivalence $\Fil^{[0,1]}\DAlg_{R} \simeq \DG_{-}^{\{0,1\}}\DAlg_{R}$ as an equivalence between square-zero Smith ideals and derivations. 

In the end of the text, we explain how to set-up the theory of derived crystalline cohomology relative to a general (derived) divided power base. We consider the $\infty$-category $\mathbb{Z}^{\crys}$ consisting of pairs $(A\rightarrow \overline{A},R)$ where $(A\rightarrow \overline{A})$ is a derived divided power map (to which we will refer to as  \textbf{divided power base}), and $\overline{R}$ a derived $\overline{A}$-algebra, and define a functor $\R\Gamma_{\crys}: \mathbb{Z}^{\crys} \rightarrow \DAlg_{\mathbb{Z}}^{\Delta^{1},\pd} $ with values in the $\infty$-category of derived divided power rings, whose value on a pair $(A\rightarrow \overline{A},R)$ is the \textbf{derived crystalline cohomology of $R$ relative to $(A\rightarrow \overline{A})$}, denoted $\R\Gamma_{\crys}(R/(A\rightarrow \overline{R})$. We show that it follows from our definition that derived crystalline cohomology satisfies a base change property with respect to maps $(A \rightarrow \overline{A}) \rightarrow (B \rightarrow \overline{B})$ of divided power bases. This property implies the following theorem, which summarizes the main properties of derived crystalline cohomology. To formulate it, we notice that for a given divided power base $(A \rightarrow \overline{A})$, the derived divided power ring $\R\Gamma_{\crys}(R/(A \rightarrow \overline{A}))$ is a derived divided power algebra under $(A\rightarrow \overline{A})$, and hence, n particular, the underlying derived ring of $\R\Gamma_{\crys}(R/(A\rightarrow \overline{A}))$ is a derived $A$-algebra.

\begin{Theor}

Let $(A\rightarrow \overline{A})$ be a derived divided power base. 

\begin{enumerate}

\item Let $R$ be a derived $\overline{A}$-algebra. There is a natural equivalence $$\R\Gamma_{\crys}(R/(A\rightarrow \overline{A})) \underset{A}{\otimes} \overline{A} \simeq \LOmega_{R/\overline{A}}. $$

\item Let $\overline{R}$ be a derive $\overline{A}$-algebra, and $R$ a derived $A$-algebra endowed with an equivalence $R\otimes_{A} \overline{A} \simeq \overline{R}$. Then there is a natural equivalence

$$
\LOmega_{R/A} \simeq \R\Gamma_{\crys}(\overline{R}/(A \rightarrow \overline{A})).
$$

\item Let $(\mathcal{A} \rightarrow \overline{A})$ be the derived divided power envelope of $(A\rightarrow \overline{A})$. Recall that $\mathcal{A} \simeq \LOmega_{\overline{A}/A}$. The divided power structure on $(A\rightarrow \overline{A})$ gives a natural map $(\mathcal{A} \rightarrow \overline{A}) \rightarrow (A\rightarrow \overline{A})$. Let $R$ be a derived $\overline{A}$-algebra, and consider its derived De Rham cohomology $\LOmega_{R/A}$ with the natural $\mathcal{A}$-linear structure. There is a natural equivalence

$$
\LOmega_{R/A} \underset{\mathcal{A}}{\otimes} A \simeq \R\Gamma_{\crys}(R/(A\rightarrow \overline{A})).
$$

\end{enumerate}

\end{Theor}

The equivalences listed above are formulated as equivalences of derived algebras. In the relevant section, we explain how to promote it to an equivalence of derived divided power algebras.

The first two properties are the main properties expected of crystalline cohomology. The first merely states, that derived crystalline cohomogy is a natural lift of derived De Rham cohomology. The second says that we can compute derived crystalline cohomology as derived De Rham cohomology of a lift. The third property is a generalization of the idea that one can compute the crystalline cohomology of an $\mathbb{F}_{p}$-algebra $R$ as the derived De Rham cohomology of $R$ relative to $\mathbb{Z}$, based changed to $\mathbb{Z}_{p}$ along the natural map $\LOmega_{\mathbb{F}_{p}/\mathbb{Z}} \rightarrow \mathbb{Z}_{p}$. This observation is to the best of our knowledge, due to B.Bhatt, and appears in the text \cite{Bh12(2)}.

Let us finish this introduction by mentioning the comparison with the stacky and site-theoretic approach, which will be proven in a subsequent paper \cite{Magidson2}. In the work of \cite{Dr20} of Drinfeld, and the works \cite{BL22(2)} and \cite{Bh23} of B.Bhargav and J.Lurie there was developed a stacky approach to derived De Rham cohomology. Let $X$ be a derived stack over a base ring $R$. We can define a new prestack $X_{\dR}$, called \textbf{De Rham stack} of $X$, understood merely as a fuctor $\DAlg_{R,\geq 0} \rightarrow \Spc$ which is defined by the formula

$$
X_{\dR}(R) := X(R/R^{\sharp}),
$$
where $R/R^{\sharp}$ is the quotient of $R$ by a certain generalized ideal $R^{\sharp} \rightarrow R$ playing the role of nil-radical in the divided power setting. One can show (see \cite{Bh23}, Theorem 2.5.6) that for a smooth affine scheme $X$, there is an equivalence of $\mathbb{E}_{\infty}$ $R$-algebras:

$$
\Gamma(X_{\dR},\mathcal{O}_{X_{\dR}}) \simeq |\Omega^{\bullet}_{A/R}|. 
$$
More generally, it follows from \cite[Theorem 7.20]{BL22(2)} and the crystalline comparison for prismatic cohomology that the statement holds true for the derived crystalline cohomology either when the algebra satisfies some finiteness assumption, or is quasi-lci. It is expected that the result should be true for derived De Rham cohomology for any $A$, if we replace the right-hand side by Hodge-completed derived De Rham cohomology. At present we are not aware of a reference for this statement proven in this generality. In the paper \cite{Magidson2}, we characterize the construction $X \longmapsto \Gamma(X_{\dR},\mathcal{O}_{X_{\dR}})$ in terms of a universal property which makes manifest the fact that for any affine scheme $X$, the $\mathbb{E}_{\infty}$-algebra $\Gamma(X_{\dR},\mathcal{O}_{X_{\dR}}) $ lifts to a filtered derived divided power algebra in $\Mod_{R}$. Let $\Fil^{\geq \star} \Gamma(X_{\dR}, \mathcal{O}_{X_{\dR}})$ be the object of $\Fil^{\geq 0}\DAlg^{\pd}_{R}$ arising from this construction. In \cite{Magidson2} we will prove the following:

\begin{Theor}

Let $A$ be a connective derived commutative $R$-algebra, and $X=\Spec A$. There is a natural filtered map 
$$\xymatrix{    \LOmega_{A/R}^{\geq \star} \ar[rr]&&  \Fil^{\geq \star} \Gamma(X_{\dR}, \mathcal{O}_{X_{\dR}})   }$$
which induces an equivalence on Hodge completions.
\end{Theor}

As we mentioned above, the isomorphism above holds for derived De Rham cohomology without Hodge completion for any map $R \rightarrow A$ satisfying either a simple finiteness condition, or a certain quasi-lci condition. We do not know whether the statement is true without Hodge completion for any map $R\rightarrow A$, but if a counter-example exists, it has to be some infinitely-generated and not quasi-lci algebra.

\subsection{Notations and terminology.}

We use the language of $\infty$-categories and higher algebra developed by J.Lurie in \cite{HTT} and \cite{HA}. We freely use the notions of an \emph{$\infty$-category}, a \emph{functor between $\infty$-categories}, an \emph{$\infty$-categorical adjunction} etc. The amount of higher categorical machinery used in the text, is limited to a very basic level. Essentially the only technical statements we use are the \emph{adjoint functor theorem} \cite[Corollary 5.5.2.9]{HTT}, and the \emph{Barr-Beck-Lurie monadicity theorem} formulated in \cite[Theorem 4.7.3.5]{HA}. In general, most of our constructions are formulated in a homotopy invariant way. In particular, all limits and colimits are in general assumed to be in the homotopy sense. For better exposition, we often restrict our attention to some discrete computations, in which case we specify that we are working with discrete cochain complexes, algebras etc. 

We let $\Spc$ be the $\infty$-category of \emph{spaces}, or \emph{$\infty$-groupoids}, and $\Spt$ the $\infty$-category of \emph{spectra}. For a ring spectrum $A$, we denote $\Mod_{A}$ the $\infty$-category of \emph{$A$-module spectra}. For example, our notation for the derived $\infty$-category of abelian groups is $\Mod_{\mathbb{Z}}$. In a stable $\infty$-category $\C$, we denote $[1]$ and $[-1]$ suspension and loop functors respectively. 

We often encounter pull-back diagrams of $\infty$-categories. Unless specified otherwise, these are taken in the $\infty$-category of large $\infty$-categories. 

For a pair of $\infty$-categories $\C$ and $\D$, we let $\Fun(\C, \D)$ be the $\infty$-category of functors from $\C$ to $\D$. The $\infty$-category $\Fun(\C, \C)$ has a monoidal structure given by composition of functors. A \textbf{monad on $\C$} is an associative algebra object in this monoidal $\infty$-category.

Most $\infty$-categories we will encounter in the text will be presentable. We let $\Pr^{L}$ to be the $\infty$-category of presentable $\infty$-categories and continuous functors. This $\infty$-category has a symmetric monoidal structure constructed by J.Lurie.

A \textbf{presentable symmetric monoidal $\infty$-category} is a commutative algebra object in $\Pr^{L}$.

We use the $\bullet$-symbol for the grading on a graded object $X^{\bullet}$, and the $\star$-symbol to indicate the filtration on a filtered object $X^{\geq \star}$.

We use homological notation in the context of $t$-structures on a stable $\infty$-category. For instance, for a stable $\infty$-category $\C$ with a $t$-structure, $\C_{\geq 0}$ is the full subcategory of connective objects. For the full subcategory of discrete objects, we use the symbol $\C_{\heartsuit}$.

Given a presentable $\infty$-category $\C$, we denote $\C^{\omega} \subset \C$ the full $\infty$-subcategory consisting of compact objects in $\C$.

Given a symmetric monoidal $\infty$-category $\C$, one can talk about $\mathbb{E}_{\infty}$-algebra objects in $\C$. For us, $\C$ will typically be a stable symmetric monoidal $\infty$-category of the form $\Mod_{R}$ for $R$ a derived commutative ring. In this setting, we can also talk about \textbf{derived commutative $R$-algebras}.

In the context of derived algebraic geometry, we follow some basic terminology of the texts \cite{GaitsRozI}, \cite{GR2}. For instance, we use the notions of \textbf{prestack} and \textbf{quasicoherent sheaves on prestacks}. A prestack over a derived ring $R$ is a functor $X: \DAlg_{R,\geq 0} \rightarrow \Spc$, where $\DAlg_{R,\geq 0}$ is the $\infty$-category of connective commutative $R$-algebras over the base ring $R$. For an $A\in \DAlg_{R,\geq 0}$, there is a representable prestack $\Spec(A)$, i.e. \textbf{a derived affine scheme}. Any prestack $X$ can be written as a colimit of derived affine schemes

$$
X=\underset{\Spec(A) \rightarrow X}{\colim}\;\; \Spec(A)
$$

The $\infty$-category of quasi-coherent sheaves on $X$ is the limit

$$\QCoh(X) \simeq \underset{\Spec(A) \rightarrow X}{\lim}\;\; \Mod_{A}$$
taken along the diagram of all derived affine schemes mapping to $X$ using base change functors $-\underset{A} {\otimes} B: \Mod_{A} \rightarrow \Mod_{B}$ for any map $\Spec(B) \rightarrow \Spec(A)$ over $X$.

\section*{Acknowledgements.} I am grateful to Ben Antieau, Lukas Brantner, Chris Brav, Rune Haugseng, Adam Holeman, Grigory Kondyrev, Zhouhang Mao, Akhil Mathew, Joost Nuiten, Artem Prikhodko, Noah Riggenbach and Nick Rozenblyum for discussions on the topics of this text. Many ideas of this text were developed by A.Raksit in \cite{R} and Zhouhang Mao in \cite{Mao21} in a slightly different language, and we acknowledge their important contribution to the field. The author was supported by by NSF grants DMS-2152235 and DMS-21002010, and the grant 21-7-2-30-1 of the "BASIS" foundation.

\section{Preliminaries.}

\subsection{Higher category theory.}

Here we will review some higher categorical facts used in the text. We will often work with various $\infty$-categories built as pull-backs and/or lax equalizers of other symmetric monoidal $\infty$-categories. We recall the following construction from Nikolaus-Scholze \cite[Definition II.1.4]{NS}.

\begin{defn}
Assume given two $\infty$-categoties $\C$ and $\D$ and two functors $F,G: \C \rightarrow \D$. The \textbf{Lax equalizer} $\infty$-category $\LEq(F,G) $ is defined as the pull-back 

\begin{equation}\label{Lax_equalizer}
\xymatrix{ \LEq(F,G) \ar[d] \ar[r]& \D^{\Delta^{1}}\ar[d]^-{(\partial^{0}, \partial^{1})} \\
\C \ar[r]_-{(F,G)} & \D \times \D }
\end{equation}
\end{defn}

We will often study the case $\D=\C$, $F=\Id$, and $G$ is some endofunctor of $\C$. Moreover, we will use the fact that if $\C$ is a symmetric monoidal $\infty$-category and $G$ is a lax symmetric monoidal functor, then $\LEq(\Id,G)$ gets a symmetric monoidal structure. Moreover, commutative algebra objects in $\LEq(\Id,G)$ are precisely pairs $(A,f)$ where $A$ is a commutative algebra object in $\C$, and $f: A \rightarrow G(A)$ is a map of commutative algebra objects. To explain this, let us recall some basic terminology regarding symmetric monoidal $\infty$-categories. Let $\Fin_{*}$ be the category of finite pointed sets. In the approach of J.Lurie, an \textbf{$\infty$-operad} is a functor $p: \mathcal{O}_{\infty} \rightarrow \Fin_{*}$ satisfying conditions listed in \cite[Definition 2.1.1.10]{HA}. A \textbf{symmetric monoidal $\infty$-category} is an $\infty$-category $\C^{\otimes}$ equipped with a coCartesian fibtration $\C^{\otimes} \rightarrow \Fin_{*}$ of $\infty$-operads. A \textbf{lax symmetric monoidal functor} $F: \C^{\otimes} \rightarrow \D^{\otimes}$ is simply an $\infty$-operad map. Let $\CAlg(\cat)_{\lax}$ be the $\infty$-category of symmetric monoidal $\infty$-categories and lax symmetric monoidal functors. There is a fully faithful embedding $\CAlg(\cat)_{\lax} \subset \Op_{\infty}$, where $\Op_{\infty}$ is the $\infty$-category of operads defined as in \cite[Definition 2.1.4.1]{HA}. Given a symmetric monoidal $\infty$-category $p: \C^{\otimes} \rightarrow \Fin_{*}$, we let $\C=\C_{\langle 1 \rangle}:=p^{-1}(\langle 1 \rangle )$. The functor $\CAlg(\cat)_{\lax} \rightarrow \cat $,  $(p: \C^{\otimes} \rightarrow \Fin_{*}) \longmapsto  \C$ is conservative and preserves all limits\footnote{Caution: the inclusion $\CAlg(\cat)_{\lax} \hookrightarrow \Op_{\infty}$ is, however, not closed under all limits.}.

In the future, we will often write the data of a symmetric monoidal $\infty$-category $p: \C^{\otimes} \rightarrow \Fin_{*}$ simply as $\C$ keeping symmetric monoidal structure in mind. The tautological coCartesian fibration $\Fin_{*} \rightarrow \Fin_{*}$ is the \textbf{commutative operad}. For a given symmetric monoidal $\infty$-category $\C:=(p: \C^{\otimes} \rightarrow \Fin_{*})$, the collection of all operad maps $\Fin_{*} \rightarrow \C^{\otimes} $ organize into an $\infty$-category of \textbf{commutative algebra objects} in $\C$, denoted $\CAlg(\C)$. The following proposition follows immediately from definitions.

\begin{prop}\label{limits_in_Cat_lax}

Assume we are given

\begin{equation}\label{pullback}
\xymatrix{ \C \ar[r]^-{\sim}& \underset{i}{\lim}\; \C_{i}  }
\end{equation}
a limit diagram in $\CAlg(\cat)_{\lax}$. Then the natural map in $\cat$

$$
\xymatrix{ \CAlg(\C) \ar[r]^-{\sim}& \underset{i}{\lim} \; \C_{i}. }
$$
is an equivalence.

\end{prop}

\begin{prop}\label{Lax_equalizer_equivalence}
Assume $\C$ and $\D$ are symmetric monoidal $\infty$-categories, $F$ is a symmetric monoidal functor, and $G$ is lax symmetric monoidal. The $\infty$-category $\LEq(F,G)$ has a symmetric monoidal structure, and there is an equivalence

$$
\CAlg(\LEq(F,G)) \simeq \LEq(F,G: \CAlg(\C) \rightarrow \CAlg(\D)).
$$
\end{prop}

\begin{proof}
Using \cite[Construction IV.2.1 (ii)]{NS}, there is a symmetric monoidal $\infty$-category $\LEq(F,G)^{\otimes}$ defined as a pull-back diagram of symmetric monoidal $\infty$-categories and lax symmetric monoidal functors, whose underlying $\infty$-category is $\LEq(F,G)$. Taking the $\infty$-category of commutative algebra objects and using \ref{limits_in_Cat_lax}, we get a pull-back square of $\infty$-categories

$$
\xymatrix{\CAlg( \LEq(F,G)) \ar[d] \ar[r]& \CAlg(\D^{\Delta^{1}}) \simeq \CAlg(\D)^{\Delta^{1}} \ar[d]^-{(ev^{0}, \ev^{1})} \\
\CAlg(\C) \ar[r]_-{(F,G)} & \CAlg(\D \times \D) \simeq \CAlg(\D) \times \CAlg(\D), }
$$
and the second statement follows.
\end{proof}

\begin{rem}
Note that if $F$ is only lax symmetric monoidal, but not strictly symmetric monoidal, there is no natural symmetric monoidal structure on $\LEq(F,G)$. This is an example demonstrating that in general the $\infty$-category $\CAlg(\cat)_{\lax}$ is not closed under all limits inside $\Op_{\infty}$. It is crucial that when $F$ is strictly symmetric monoidal, one can construct a symmetric monoidal structure on $\LEq(F,G)$ making the square \ref{Lax_equalizer} a square of symmetric monoidal $\infty$-categories and lax symmetric monoidal functors. In other words, the commutative square \ref{pullback} needs to be \emph{given}, and in general, we \emph{can not define} $\C$ as a pull-back of symmetric monoidal $\infty$-categories and lax symmetric monoidal functors.
\end{rem}

Over a ring of non-zero characteristic, in addition to $\mathbb{E}_{\infty}$-algebras, there are also theories of derived commutative algebras as well as derived divided power algebras. We will review some basics of derived commutative algebras following \cite{R} in the next subsection, and the theory of derived divided power algebras in Section 2.3 via the approach of \cite{BCN21}. We will also formulate and prove a version of Proposition \ref{limits_in_Cat_lax} in the setting of derived algebras in Section 2.3. To this end, we will introduce the abstract framework in which the theories of derived commutative algebras and derived divided power algebras fit.

Recall that given an $\infty$-category $\C$, a \textbf{monad} on $\C$ is an algebra object in the monoidal $\infty$-category $\Fun(\C, \C)$ with respect to the composition product. To formulate and prove a version of Proposition \ref{limits_in_Cat_lax} in the setting of algebras over a monad, we ought to have, given a pair of $\infty$-categories $(\C,\T_{\C})$ and $(\D,\T_{\D})$ endowed with monads, the notion of a \textbf{lax equivariant} functor $F: (\C,\T_{\C}) \rightarrow (\D, \T_{\D})$ with respect to these monads action. In order to speak about such lax equivariant functors, we use some basic $(\infty,2)$-categorical formalism. We generally follow the approach of \cite{GaitsRozI} to $(\infty,2)$-categories, and use some elements of the formalism developed by R.Haugseng in the text \cite{H}. We will give references for all definitions and statements, but will not present a detailed overview of the formalism and restrict to simply explaining what various constructions mean on the naive level of 2-categories.

\begin{notation}
We denote $\cat_{(\infty,2)}$ the $(\infty,1)$-category of $(\infty,2)$-categories. There exist several approaches to $(\infty,2)$-categories in the literature. For instance, one can think of objects of $\cat_{(\infty,2)}$ as the $\infty$-categories of $\infty$-categories enriched in the symmetric monoidal $\infty$-category $\cat$. This means that an $(\infty,2)$-category $\textbf{C}$ is characterized by the data of a collection of objects $X \in \textbf{C}$ and for any pair $X,Y \in \textbf{C}$, an $(\infty,1)$-category $\Hom(X,Y)$ together with composition functors $$\Hom(X,Y) \times \Hom(Y,Z) \rightarrow \Hom(X,Z) $$ with a hierarchy of higher coherences\footnote{See \cite[Remark 2.14]{H} for discussion and references to this approach. We do not have to choose a specific model, because most our statements and constructions already exist in the literature, and we only need to reference them and apply to concrete problems.}. The elements of $\Hom(X,Y)$ are usually called \textbf{1-morphisms} in $\textbf{C}$, and morphisms of $\Hom(X,Y)$ are usually called \textbf{2-morphisms} in $\textbf{C}$. We use bold letters like $\textbf{C},\textbf{D},...$ for objects of $\cat_{(\infty,2)}$. Given an $(\infty,2)$-category $\textbf{C}$, we denote $\imath_{1}\textbf{C}$ the underlying $(\infty,1)$-category obtained by discarding all non-invertible $2$-morphisms. We denote $\textbf{Cat}$ the $(\infty,2)$-category of $(\infty,1)$-categories\footnote{We usually call $(\infty,1)$-categories simply "$\infty$-categories", unless there is any possibility of confusion, or unless we want to be extra pedantic.}. The construction of $\textbf{Cat}$ uses the natural enhancement of the usual $(\infty,1)$-category of $(\infty,1)$-categories by letting the $\Hom$-category between two objects $\C,\D \in \cat$ to be the $\infty$-category $\Fun(\C,\D)$ of all functors and not necessarily invertible natural transformations. Given two $(\infty,1)$-categories $\C,\D$, there is an $(\infty,1)$-category $\Fun(\C,\D)$ whose objects are functors $F: \C \rightarrow \D$, and morphisms are not necessarily invertible natural transformations $F \Rightarrow G$. 
\end{notation}

\begin{construction}
Given a pair of $(\infty,2)$-categoies $\textbf{C},\textbf{D}$, there is an $\infty$-category $\Fun(\textbf{C}, \textbf{D})_{\lax}$ of 2-functors $\textbf{C} \rightarrow \textbf{D}$ and \textbf{lax natural transformations}. A lax natural transformation $\eta: F \Rightarrow G$ between 2-functors $F,G: \textbf{C} \rightarrow \textbf{D}$ is the data assigning to any 1-morphism $f: X \rightarrow Y$ in $\textbf{C}$ a lax square

$$
\xymatrix{  F(X)\ar[d]_-{F(f)} \ar[r]^-{\eta_{X}} &  G(X) \ar@{=>}[ld] \ar[d]^-{G(f)}\\
F(Y) \ar[r]_-{\eta_{Y}} & G(Y)  ,   }
$$
i.e. a not necessarily invertible map $\eta_{f}: G(f) \circ \eta_{X} \rightarrow \eta_{Y} \circ F(f)$. Dually, there is also a notion of an \textbf{oplax natural transformation} for two functors $F,G: \textbf{C} \rightarrow \textbf{D}$. 

See \cite[Definition 2.9]{H} for a precise definition of $ \Fun(\textbf{C}, \textbf{D})_{(\op)\lax}$ using Gray tensor product. We will give a very short overview here. There exists a bifunctor $\otimes^{(\op)\lax}: \cat_{(\infty,2)} \times \cat_{(\infty,2)} \rightarrow \cat_{(\infty,2)} $ called \textbf{(op)lax Gray tensor product} which preserves colimits in each variable, and such that for any triple of $(\infty,2)$-categories $\textbf{C},\textbf{D}, \textbf{E}$, we have

$$
\Map_{\cat_{(\infty,2)}}(\textbf{C}, \Fun(\textbf{D}, \textbf{E})_{\op\lax}) \simeq \Map_{\cat_{(\infty,2)}}(\textbf{D}\otimes^{\op\lax} \textbf{C}, \textbf{E})$$

$$ \simeq \Map_{\cat_{(\infty,2)}}(\textbf{C}\otimes^{\lax} \textbf{D}, \textbf{E} ) \simeq \Map_{\cat_{(\infty,2)}}(\textbf{D}, \Fun(\textbf{C}, \textbf{E})_{\lax}).
$$
Unwinding the definition, one obtains that an (op)lax natural transformation between functors $F,G: \textbf{C} \rightarrow \textbf{D}$ is a functor of $(\infty,2)$-categories $\textbf{C} \otimes^{(\op)\lax} \Delta^{1} \rightarrow \textbf{D} $ whose restriction to $0,1 \in \Delta^{1}$ is $F$ and $G$ respectively.

\end{construction}

\begin{ex}
Let $\mathbb{N}=\{0,1,2,...\}$ be the monoid of natural numbers by addition. We will apply the previous construction to the case $\textbf{C}:=\B\mathbb{N}$, which is an ordinary category with one object and a unique non-identity morphism, and $\textbf{D}:=\textbf{Cat}$. The data of a functor $\B\mathbb{N} \rightarrow \textbf{Cat}$ is the same as the data of a pair $(\C,\T_{C})$ where $\C$ is an $\infty$-category, and $\T: \C \rightarrow \C$ an endofunctor. Giving a lax natural transformation $F: (\C,\T_{\C}) \rightarrow (\D,\T_{\D})$ is equivalent to giving a functor $F: \C \rightarrow \D$ and for any $X\in \C$ a lax equivariant structure map $\eta_{X}: \T_{\D} \circ F(X) \rightarrow F \circ \T_{\C}(X)$.
\end{ex}

\begin{construction}\label{monads_adjunctions}

This construction follows \cite[Definition 4.2, Definition 5.2]{H} and the references therein.

\begin{enumerate}

\item There exist a universal $2$-category $\textbf{mnd}$ containing a monad. The universal property of $\textbf{mnd}$ is that giving a functor $\textbf{mnd} \rightarrow \textbf{Cat}$ is the same as giving an $\infty$-category $\C$ endowed with a monad $\T: \C \rightarrow \C$. The $2$-category \textbf{mnd} contains a single object with $\hom$-category being the augmented simplex category, i.e. $\textbf{mnd}\simeq \B\Delta_{a}$. We define the \textbf{$(\infty,2)$-category of monads and (op)lax-equivariant 1-morphisms} in $\textbf{C}$ as

$$
\Mnd(\textbf{cat})_{(\op)\lax}:=\Fun(\textbf{mnd},\textbf{Cat})_{(\op)\lax}.
$$
To unravel the definition above, objects of $\Mnd(\cat)_{(\op)\lax}$ are pairs $(\C,\T_{\C})$ consisting of an $\infty$-category $\C$ endowed with a monad $\T_{\C}$ and maps are functors $F: \C \rightarrow \D$ which are (op)lax-equivariant with respect to the monads $\T_{\C}$ and $\T_{\D}$. We also denote $\Mnd(\cat)_{(\op)\lax}:=\imath_{1}\Fun(\textbf{mnd},\textbf{Cat})_{(\op)\lax}$ for simplicity. 

More generally, the notion of a monad makes sense internally to any $(\infty,2)$-category. Given $\textbf{C}\in \cat_{(\infty,2)}$, we define the $(\infty,2)$-category of monads and (op)lax equivariant $1$-morphisms in $\textbf{C}$ as the $(\infty,2)$-category of functors $\Mnd(\textbf{C})_{(\op)\lax}:= \Fun(\textbf{mnd}, \textbf{C})_{(\op)\lax}.$

\item  There exits a universal $2$-category $\textbf{adj}$ containing an adjunction with the universal property that the data of a functor $\textbf{adj} \rightarrow \textbf{Cat}$ is equivalent to the data of a pair of $\infty$-categories $\C,\D$ endowed with an adjunction $\xymatrix{ \D   \ar@<-0.5ex>[r]_-{G} & \ar@<-0.5ex>[l]_-{F} \C  }$.  We define the \textbf{$(\infty,2)$-category of adjunctions and (op)lax-equivariant functors} as $$ \Adj(\textbf{Cat})_{(\op)\lax}:=\Fun(\textbf{adj},\textbf{Cat})_{(\op)\lax},$$

and use the simplified notation $\Adj(\cat)_{(\op)\lax}:=\imath_{1}\Fun(\textbf{adj},\textbf{Cat})_{(\op)\lax} $ for the underlying $(\infty,1)$-category. There exists a forgetful functor $\Adj(\textbf{Cat})_{\lax} \rightarrow \Mnd(\textbf{Cat})_{\lax} $ sending any adjunction $\xymatrix{ \D   \ar@<-0.5ex>[r]_-{G} & \ar@<-0.5ex>[l]_-{F} \C  }$ to the monad $(\C, G\circ F) $. 

The notion of an adjunction is well-defined internally to any $(\infty,2)$-category. Given $\textbf{C} \in \cat_{(\infty,2)}$, the $\infty$-category of adjunctions and (op)lax equivariant $1$-morphisms in $\textbf{C}$ is by definition, the $(\infty,2)$-category of functors $\Adj(\textbf{C}):= \Fun(\textbf{adj}, \textbf{C})_{(\op)\lax}.$

\item

According to \cite[Corollary 5.8]{H}, the forgetful functor $\Adj(\textbf{Cat})_{\lax} \rightarrow \Mnd(\textbf{Cat})_{\lax}$ admits a fully faithful right adjoint $\Mnd(\textbf{Cat})_{\lax} \rightarrow \Adj(\textbf{Cat})_{\lax} $ sending an object $(\C,\T) \in \Mnd(\textbf{Cat})_{\lax} $ to the adjunction $$\xymatrix{ \Alg_{\T}(\C)   \ar@<-0.5ex>[r]_-{U_{\T}} & \ar@<-0.5ex>[l]_-{F_{\T}} \C,   }$$

where $\Alg_{\T}(\C)$ is the $\infty$-category of $\T$-algebras in $\C$, $F_{\T}$ is the free algebra functor, and $U_{\T}$ is the forgetful functor. The claim that $\Alg_{\T}(\C)$ is the conventional $\infty$-category of $\T$-algebras in $\C$ as defined in \cite{HA} is established in \cite[Proposition 8.10]{H}.

\end{enumerate}
\end{construction}

\begin{prop}\label{limits_in_Mod_A_lax}

Assume $(\C,\T_{\C}) \xymatrix{ \ar[r]^-{\sim}&} \underset{i}{\lim} (\C_{i}, \T_{\C_{i}})$ is a limit diagram in $\Mnd(\cat)_{\lax}$. In this situation, the natural functor induces an equivalence:

$$\xymatrix{   \Alg_{\T_{\C}}(\C) \ar[r]^-{\sim}& \underset{i}{\lim} \Alg_{\T_{\C_{i}}}(\C_{i})   } $$

\end{prop}

\begin{proof}
This follows immediately from part 3 of Construction \ref{monads_adjunctions} as right adjoints commute with limits, and the functor $\Adj(\cat)_{\lax} \rightarrow \cat$ sending an adjunction $\xymatrix{ \D   \ar@<-0.5ex>[r]_-{G} & \ar@<-0.5ex>[l]_-{F} \C  }$ to the $\infty$-category $\D$ (the source of the right adjoint) commutes with limits as well.
\end{proof}

\begin{construction}\label{construction_of_monads}
Let $\A$ be a monoidal $\infty$-category, $\textbf{BA}$ the corresponding $(\infty,2)$-category which has a single object with $\hom$-category being $\A$. We let $$\Mod_{\A}(\textbf{Cat})_{(\op)\lax}:=\Fun(\textbf{BA},\textbf{Cat})_{(\op)\lax}$$ be the $(\infty,2)$-category of \textbf{$(\infty,1)$-categories with an $\A$-action and (op)lax $\A$-equivariant functors between them}, and following our usual convention let $\Mod_{\A}(\cat)_{(\op)\lax}:=\imath_{1} \Mod_{\A}(\textbf{Cat})_{(\op)\lax}$. Giving an algebra object $\T \in \Alg(\A)$ is equivalent to giving a functor of $(\infty,2)$-categories $\textbf{mnd} \rightarrow \textbf{BA}$. Consequently, given a monad $\T$ in $\Alg(\A)$, we get a limit-preserving restriction functor

\begin{equation}\label{getting_monads}
\Mod_{\A} (\cat)_{\lax} \xymatrix{\ar[r]&} \Mnd(\cat)_{\lax}
\end{equation}
which takes any $\infty$-category $\C$ with $\A$-action to the monad $(\C,\T_{\C})$, where $\T_{\C}$ is the image of $\T \in \Alg(\A)$ under the monoidal functor $\A \rightarrow \Fun(\C,\C)$.

We shall now recall the construction of passing to adjoints in the $(\infty,2)$-categorical setting, and how it interchanges oplax and lax natural transformations. Given a pair of $(\infty,2)$-categories $\textbf{C},\textbf{D}$, let $\Fun(\textbf{C},\textbf{D})_{\oplax,L}$ be  the subcategory of $\Fun(\textbf{C},\textbf{D})_{\oplax}$ with all maps being left adjoint functors, and similarly $\Fun(\textbf{C},\textbf{D})_{\lax,R}$ the subcategory of $\Fun(\textbf{C},\textbf{D})_{\lax}$ with all maps being right adjoint functors. \cite[Ch.12, Corollary 3.1.9]{GR2} states that taking adjoints gives a natural equivalence

$$
\Fun(\textbf{C},\textbf{D})_{\oplax,L}\simeq \Fun(\textbf{C}^{(1,2)-\op} ,\textbf{D})_{\lax,R},
$$
where the superscript ${(1,2)-\op}  $ stands for inverting all $1$ and $2$-morphisms.

Applying this general fact to the construction \ref{getting_monads}, we get a commutative diagram of $\infty$-categories

$$
\xymatrix{  (\Mod_{\A}(\cat)_{L})^{\op} \ar[d] \ar[r]& (\Mnd(\cat)_{L})^{\op} \ar[d]\\
\Mod_{\A}(\cat)_{\lax,R} \ar[r]& \Mnd(\cat)_{\lax,R} ,  }
$$
where the vertical functors are taking right adjoints of $1$-morphisms.

\end{construction}

We will use the following Proposition in several crucial places.

\begin{prop}\label{main_proposition_on_monads}
Assume $\A$ is a monoidal $\infty$-category with a fixed algebra object $\T \in \Alg(\C)$. Suppose that we are given a diagram $\{\C_{i}\}_{i\in I}$ in $\Mod_{\A}(\cat)_{L}$ such that the corresponding diagram in $\Mod_{\A}(\cat)_{\lax}$ obtained by passing to right adjoints, has a limit $\C \simeq \underset{i}{\lim}\; \C_{i}$ in $\Mod_{\A}(\cat)_{\lax}$. Then Construction \ref{construction_of_monads} produces an equivalence

$$
\xymatrix{    \Alg_{\T_{\C}}(\C) \ar[r]^-{\sim}& \underset{i}{\lim} \Alg_{\T_{\C_{i}}}(\C_{i}).   }
$$
\end{prop}

\subsection{Non-abelian derived functors.}

In this section we will review some basic ideas from the theory of derived commutative rings. Many ideas of this subsection are due to L.Brantner and A.Mathew and were first developed in the paper \cite{BM19} in the case $\C=\Mod_{k}$, $k$ a field. We will closely follow the terminology of the paper of A.Raksit \cite{R}.

\begin{defn}
Let $\CAlg^{\poly}_{\mathbb{Z}}$ be the category of all finitely generated polynomial rings. The $\infty$-category of \textbf{connective derived commutative rings} $\DAlg_{\mathbb{Z},\geq 0}$ is obtained by freely adjoining all sifted colimits to $\CAlg^{\poly}_{\mathbb{Z}}$, $$\DAlg_{\mathbb{Z},\geq 0} :=\mathcal{P}_{\Sigma}(\CAlg^{\poly}_{\mathbb{Z}}).$$

Given an object $R$, we let $\DAlg_{R,\geq 0}:= (\DAlg_{\mathbb{Z},\geq 0})_{R/}$ be the under category of \textbf{connective derived commutative $R$-algebras}.
\end{defn}

\begin{rem}
The strategy of defining $R$-algebras used above will be used very frequently throughout the text. In many cases, we will define some $\infty$-category of derived rings over $\mathbb{Z}$ first, and then define $R$-linear objects as an under-category. Consequently, many constructions with derived rings of different types can be done in the case $R$ is a discrete ring, or even $R=\mathbb{Z}$, and then generalized to under-objects.
\end{rem}

\begin{rem}
Let $R$ be a discrete commutative ring. The $\infty$-category $\Mod_{R}$ of left $R$-modules is a stable presentable symmetric monoidal $\infty$-category. There is a $t$-structure $(\Mod_{R,\geq 0},\Mod_{R,\leq 0})$ characterized by the property that $M\in \Mod_{R}$ is connective if the underlying spectrum is connective. This $t$-structure is compatible with the symmetric monoidal structure and is right and left complete. In addition, there is a full subcategory $\Mod_{R,0} \subset \Mod_{R,\geq 0}$ of finitely generated free modules which is closed under direct sums and tensor powers. This $\infty$-category compactly and projectively generates $\Mod_{R,\geq 0}$, i.e. $\Mod_{R,\geq 0} \simeq \mathcal{P}_{\Sigma}(\Mod_{R,0})$.
\end{rem}

We will assume that $R$ is discrete throughout this subsection, unless explicitly stated.

The $\infty$-category $\DAlg_{R,\geq 0}$ is monadic over $\Mod_{R,\geq 0}$. It turns out that it is possible to extend this monad to a monad on $\Mod_{R}$, so that one gets the free \textbf{derived commutative algebra} monad $\LSym_{\R} : \Mod_{R} \rightarrow \Mod_{R}$. The paper \cite{R} performs this construction by using the natural filtration on the free symmetric algebra. We will briefly review the Raksit's construction below.

\begin{ex}\label{Sym_excisive}
Let $X\in \Mod_{R}$. The symmetric power of $X$ is defined by the formula $$\Sym_{\C}(X) :=(\otimes^{n} X)_{h \Sigma_{n}},$$ where the subscript $(-)_{h \Sigma_{n}}$ denotes homotopy orbits with respect to the symmetric group action. The functor $\Sym^{\leq n}=\oplus_{i=1}^{n} \Sym^{i}_{\R}: \Mod_{R} \rightarrow \Mod_{R}$ is excisively polynomial of degree $n$. The (discrete) symmetric power of $X\in \Mod_{R,\heartsuit}$ is defined as $$\Sym_{R,\heartsuit}(X) :=(\otimes^{n} X)_{\Sigma_{n}},$$ where the coinvariants are taken in the naive sense. The functors $\Sym^{\leq n}_{R,\heartsuit} :=\oplus_{i=1}^{n} \Sym^{i}_{R,\heartsuit}: \Mod_{R,\heartsuit} \rightarrow \Mod_{R,\heartsuit}$ are additively polynomial of degree $n$.
\end{ex}

\begin{construction}\label{derived_functors_Cartier}
Let $\T: \Mod_{R,0} \rightarrow \Mod_{R} $ a functor. We define \textbf{the right-left Kan extension }of $\T$, $\T^{RL}: \Mod_{R} \rightarrow \Mod_{R}$ as the right Kan extension along the inclusion $\Mod_{R,0} \subset \Mod^{\omega}_{R,\leq 0}$, followed by the left Kan extension to all of $\Mod_{R}$. 
\end{construction}

\begin{defn}

Let $\D$ be an $\infty$-category containing all totalizations and geometric realizations. We say that a functor $G: \Mod_{R,\leq 0}^{\omega} \rightarrow \D$ preserves \emph{finite coconnective geometric realizations} if for any simplicial object $A_{\ast} $ in $ \Mod_{R,\leq 0}^{\omega}$ which is $n$-skeletal for some $n$ such that the geometric realization $|A_{\ast}|$ belongs to $\Mod_{R,\leq 0}^{\omega}$, the colimit  of $G(A_{\ast})$ exists in $\D$ and the natural map $|G(A_{\ast})| \rightarrow G(|A_{\ast}|)$ is an equivalence. A functor $G: \Mod_{R}^{\free,\fg} \rightarrow \D$ is \textbf{right-left extendable} if the right Kan extension $R G: \Mod_{R,\leq 0}^{\omega} \rightarrow \D$ preserves finite coconnective geometric realizations.

\end{defn}

\begin{defn} This definition concerns the two notions of polynomiality for functors of $\infty$-categories.

\begin{itemize}

\item Let $\C$ be an additive $\infty$-category, and $\D$ an $\infty$-category containing small colimits. An additively polynomial functor $F: \C \rightarrow \D$ of degree $0$ is a constant functor. Assume a polynomial functor of degree $n-1$ has been defined. A functor $F: \C \rightarrow \D$ is \textbf{additively polynomial of degree $n$} if the derivative $DF_{X}(Y) = \fib(F(X\oplus Y) \rightarrow F(X))$ is of degree $n-1$. A functor $F$ is \textbf{additively polynomial} if it is additively polynomial of some degree.

\item Let $\C$ be an $\infty$-category containing finite colimits, and $\D$ an $\infty$-category containing all limits. A functor $F: \C \rightarrow \D$ is \textbf{excisively polynomial} if it is $n$-excisive of some degree $n$ in the sense of Goodwillie calculus (see \cite[Definition 4.2.7]{R}).

\end{itemize}

We let $\Fun_{\apoly}(\Mod_{R,0} ,\Mod_{R})$ be the $\infty$-category of additively polynomial functors $\Mod_{R,0} \rightarrow \Mod_{R}$ and $\Fun_{\epoly}^{\Sigma}(\Mod_{R},\Mod_{R})$ be the $\infty$-category of excisively polynomial sifted colimit preserving functors $\Mod_{R} \rightarrow \Mod_{R}$. 

\end{defn}

For future reference, we record the following result proven in \cite[Proposition 3.34, Theorem 3.35]{BM19} in the case $R$ is a field, and generalized to an arbitrary derived algebraic context in \cite[Proposition 4.2.14, Proposition 4.2.15]{R}.

\begin{prop}\label{addpoly_and_excpoly}
Let $R$ be a connective derived commutative ring. The restriction functor $$\xymatrix{\Fun^{\Sigma}_{\epoly}(\Mod_{R},\Mod_{R}) \ar[r]& \Fun_{\apoly}(\Mod_{R,0},\Mod_{R})}$$ is an equivalence of $\infty$-categories with inverse given by right-left extension.
\end{prop}

\begin{rem}
For any right-left extendable functor $\T: \Mod_{R,0} \rightarrow \D$, the restriction $\T^{RL}|_{\Mod_{R,\geq 0}}$ of its right-left extension to the full subcategory of connective objects coincides with the left Kan extension of $F$ along the inclusion $\Mod_{R,0} \subset \Mod_{R,\geq 0}$. In this case, we will use the notation $\LT: \Mod_{R} \rightarrow \D$ for the derived functor defined on all $\Mod_{R}$.
\end{rem}

\begin{defn}\label{filtered_monad}
Let $\mathbb{Z}_{\geq 0}^{\times}$ be the category of non-negative integers considered with monoidal structure given by multiplication. A \textbf{filtered (sifted colimit preserving) monad} is a lax monoidal functor

 \begin{equation}\label{filtered_monad}
\xymatrix{\T^{\leq \star}: \mathbb{Z}^{\times}_{\geq 0} \ar[r]&      \End_{\Sigma}(\Mod_{R}) . }
\end{equation}

\end{defn}

A.Raksit shows in \cite{R} that for a filtered monad $\T^{\leq \star}$, the colimit $\T:=\underset{\mathbb{Z}_{\geq 0}}{\colim} \;\T^{\leq \star}$ is a monad on $\Mod_{R}$. The advantage of using filtered monads is that many monads are not excisively polynomial on the nose, but have filtrations whose stages are excisively polynomial functors.

\begin{ex}\label{Sym_rl}

The functor

 \begin{equation}\label{filtered_Sym}
\xymatrix{\Sym_{R}^{\leq \star}: \mathbb{Z}^{\times}_{\geq 0} \ar[r]&      \End_{\Sigma}(\Mod_{R})  }
\end{equation}
$$
\xymatrix{ n \ar@{|->}[r] & (X \longmapsto \oplus_{i=1}^{n} \Sym^{i}_{R}(X)) }
$$
is a filtered monad such that $\underset{\mathbb{Z}_{\geq 0}}{\colim} \Sym_{R}^{\leq \star} \simeq \Sym_{R} $. Moreover, the stages of the filtration are sifted colimit preserving excisively polynomial functors. It follows that the functor $\Sym_{R}^{\leq \star}$ lands in the full subcategory $\End_{\Sigma}^{\epoly}(\Mod_{R}) \subset \End(\Mod_{R}), $ which is equivalent to $\Fun^{\apoly}(\Mod_{R,0}, \Mod_{R}) $ by Proposition \ref{addpoly_and_excpoly}. In other words, the filtered monad $\Sym_{\R}^{\leq \star}$ is right-left extended from $\Mod_{R,0} \subset \Mod_{R}$. As it preserves the connective subcategory $\Mod_{R,\geq 0}\subset \Mod_{R}$, the monad $\Sym_{R}^{\leq \star}$ is in particular, uniquely determined by the filtered monad $\Sym^{\leq \star}_{\Mod_{R,\geq 0}}$ on $\Mod_{R,\geq 0}$.
\end{ex}

\begin{construction}\label{LSym_monad}
Let $R$ be a \emph{discrete} commutative ring. Let $\Sym: \Mod_{R,0} \rightarrow \Mod_{R}$ be the free commutative algebra functor, $\Sym_{R}(M):=\bigoplus_{n\geq 0} (M^{\otimes^{n}})_{\Sigma_{n}}$. Construction \ref{derived_functors_Cartier} supplies a derived functor $\LSym_{R} : \Mod_{R} \rightarrow \Mod_{R}$ which we will refer to as the \textbf{free derived commutative algebra} in $\Mod_{R}$. Then \cite[Construction 4.2.19]{R} shows that the functor $\LSym_{R} \in \End^{\Sigma}(\Mod_{R})$ has a monad structure. This is achieved via a version of Example \ref{Sym_rl} in the derived setting. Let $\End^{\Sigma}_{0}(\Mod_{R,\geq 0})$ be the full subcategory of $\End^{\Sigma}(\Mod_{R})$ consisting of functors which preserve the full subcategory $\Mod_{R,0} \subset \Mod_{R,\geq 0}$, and $\End^{\Sigma}_{1}(\Mod_{R,\geq 0})$ the full subcategory of functors satisfying $\pi_{0}F(X) \in \Mod_{R,0}$ for any $X\in \Mod_{R,0}$. Raksit shows in \cite[Remark 4.2.18]{R} that the inclusion $\End_{0}^{\Sigma}(\Mod_{R,\geq 0}) \subset \End_{1}^{\Sigma}(\Mod_{R,\geq 0})$ admits a \emph{monoidal} left adjoint $\tau: \End^{\Sigma}_{1}(\Mod_{R,\geq 0}) \rightarrow \End^{\Sigma}_{0}(\Mod_{R,\geq 0})$. As we observed in Example \ref{Sym_rl}, the filtered monad $\Sym^{\leq \star}_{R}$ is uniquely determined by its restriction to the connective subcategory $\Mod_{R,\geq 0} \subset \Mod_{R}$. Moreover, for any $n$, the functor $\Sym^{\leq n}_{R}$ lies in $\End_{1}^{\Sigma}(\Mod_{R})$. We can define a new filtered monad $\LSym^{\leq \star}_{R}$ by the formula $\LSym^{\leq \star}_{R} := \tau(\Sym_{R}^{\leq \star})$. Moreover, it is an excisively polynomial filtered monad, therefore it extends to a filtered monad on all of $\Mod_{R}$. Defining $$\LSym_{R}:=\underset{\mathbb{Z}_{\geq 0} }{\colim}\;\LSym_{R}^{\leq \star},$$ we get a monad on $\Mod_{R}$. 
\end{construction}

\begin{defn}
Let us apply the Construction \ref{LSym_monad} to the initial ring $R=\mathbb{Z}$. \textbf{A derived commutative ring} is an algebra over the monad $\LSym_{\mathbb{Z}}: \Mod_{\mathbb{Z}} \rightarrow \Mod_{\mathbb{Z}}$ from Construction \ref{LSym_monad}. We let $\DAlg_{\mathbb{Z}}:=\Alg_{\LSym_{\mathbb{Z}}}(\Mod_{\mathbb{Z}})$ to be the $\infty$-category of derived commutative rings. The forgetful functor $\DAlg_{\mathbb{Z}} \rightarrow \CAlg_{\mathbb{Z}}$ commutes with all limits and colimits. Given any other (at this point, not necessarily discrete) object $R\in \DAlg_{\mathbb{Z}}$, we define the $\infty$-category of \textbf{derived commutative $R$-algebras} as the under category $$\DAlg_{R}:=(\DAlg_{\mathbb{Z}})_{R/}.$$

The $\infty$-category $\DAlg_{R}$ is monadic over $\Mod_{R}$ with the monad $\LSym_{R}: \Mod_{R} \rightarrow \Mod_{R}$ which satisfies the formula $$\LSym_{R}(V \otimes R) \simeq \LSym_{\mathbb{Z}}(V) \otimes R$$
for any induced $R$-module $V\otimes R$.
\end{defn}

\subsection{Symmetric sequences and derived pd operads.}

Construction \ref{LSym_monad} produced the monad $\LSym_{R}$ from the monad of free $\mathbb{E}_{\infty}$-algebra on $\Mod_{R}$ for $R$ a discrete commutative ring. There exists another monad on the abelian symmetric monoidal category $\Mod_{R,\heartsuit}$ called \textbf{the free non-unital divided power algebra} monad $\Gamma^{+}_{R,\heartsuit}$ whose underlying functor is

\begin{equation}\label{free_pd}
\Gamma^{+}_{R,\heartsuit}(X) = \bigoplus_{n=1}^{\infty} (X^{\otimes^{n}})^{\Sigma_{n}}.
\end{equation}

This monad can be defined on the $\infty$-category $\Mod_{R}$ for any connective derived commutative ring by means of the theory of \textbf{derived symmetric sequences}. Let us recall the following classical definition.

\begin{defn}
Let $\Spt^{\B\Sigma_{n}}$ be the $\infty$-category of representations of the symmetric group $\Sigma_{n}$ in spectra. We let the \textbf{$\infty$-category of symmetric sequences} to be the product:  $$\SSeq := \underset{n\geq 0}{\prod} \Spt^{\B\Sigma_{n}} .$$

\end{defn}

Below we will recall the universal property of the $\infty$-category of symmetric sequences which we learnt from \cite[Section 1.1]{GaitsRozI}.

\begin{rem}

The $\infty$-category $\SSeq$ is the free stable presentable symmetric monoidal $\infty$-category on the unit object $\Spt \in \Pr^{L}$. The tensor product of two symmetric sequences $\{P(n) \}$ and $\{Q(n)\}$ is given by the formula 

$$
(P \otimes Q)(n) = \bigoplus_{i+j=n} \Ind_{\Sigma_{i} \times \Sigma_{j}}^{\Sigma_{n}} P(i) \otimes Q(j).
$$

The universal property of being free on one generator implies that there is an equivalence 

\begin{equation}\label{SSeq}
\SSeq \simeq \Fun_{\cat} (*, \SSeq) \simeq \Fun_{\Pr^{L}}(\Spt, \SSeq) \simeq \End_{\CAlg(\Pr^{L})}(\SSeq).
\end{equation}

The equivalence \ref{SSeq} gives a monoidal structure on the $\infty$-category $\SSeq$, called the \textbf{composition product}. For two objects $\{P(n)\}, \{Q(n)\} \in \SSeq$, there is a formula $$P\circ Q\simeq \bigoplus_{n=0}^{\infty} (P(n) \otimes Q^{\otimes^{n}})_{h \Sigma_{n}},$$
where the tensor product is taken with respect to the symmetric monoidal structure in $\SSeq$. An \textbf{operad (in spectra)} is an algebra object in $\SSeq$ with respect to the composition product symmetric monoidal structure.

\end{rem}

 \begin{construction}
 
 The forgetful functor $\CAlg(\Pr^{L}) \rightarrow \cat$ is represented by $\SSeq$. Consequently, the monoidal $\infty$-category $\SSeq$ acts on the underlying $\infty$-category of any symmetric monoidal stable presentable $\infty$-category $\C$. The action functor $-\circ : \SSeq \rightarrow \End(\C)$ is given by the formula $$P\circ X = \bigoplus_{n=0}^{\infty} (P(n) \otimes X^{\otimes^{n}})_{h \Sigma_{n}}. $$
 
 Notice that this action preserves sifted colimits. Given an operad $\mathcal{O} \in \Alg(\SSeq)$, we obtain a sifted colimit preserving monad $\T_{\mathcal{O}}:=\mathcal{O}\circ -: \C \rightarrow \C$ acting on $\C$.
 
 \end{construction}

It was observed by B.Fresse \cite{Fr} and later in the $\infty$-categorical context by Francis and Gaitsgory \cite{FrG} that there is another monoidal structure on $\SSeq$ which acts naturally on any stable presentable symmetric monoidal $\infty$-category. The product in this monoidal structure is given by the formula $$P \overline{\circ} Q= \bigoplus_{n=0}^{\infty} (P(n) \otimes Q^{\otimes^{n}})^{h \Sigma_{n}} ,$$ and the action on any symmetric monoidal stable presentable $\infty$-category $\C $ is given by $$P\overline{\circ} X=(\bigoplus_{n=0}^{\infty} P \otimes X^{\otimes^{n}})^{h \Sigma_{n}} .$$ The rigorous construction of this monoidal structure and the action on any symmetric monoidal stable presentable $\infty$-category are not essential to the current text. We would only like to remark that in the particular case of the object $(0,\mathbb{S},\mathbb{S},...)$, the monad $\bigoplus_{n\geq 1} (X^{\otimes^{n}})^{h \Sigma_{n}}$ can be called the monad of "free non-unital $\mathbb{E}_{\infty}$-divided power algebra". One crucial difference with the free $\mathbb{E}_{\infty}$-algebra monad is that the free $\mathbb{E}_{\infty}$-divided power algebra monad does not preserve connective objects, and therefore can not be treated in the same way as in Construction \ref{LSym_monad}. Instead, we will construct a natural action of the monoidal category of discrete symmetric sequences on $\Mod_{P,\heartsuit}$ for any polynomial ring $P$, and produce an action of \textbf{derived symmetric sequences} on $\Mod_{R}$ for any connective derived commutative ring $R$ using the formalism of derived functors. We should note that such construction is essentially contained in \cite{BCN21}, and our goal here is to merely review it and fix the relevant notation.

\begin{construction}\label{SSeq_gen}
Here we recall the construction of derived symmetric sequences given in \cite[Notation 3.52]{BCN21}. Let $\mathbb{Z}[\mathcal{O}_{\Sigma_{n}}] \subset \Mod_{\mathbb{Z}[\Sigma_{n}], \heartsuit}$ be the full additive subcategory in the abelian category of $\Sigma_{n}$-representations spanned by $\mathbb{Z}$-linearizations of finite $\Sigma_{n}$-sets. Let $\mathbb{Z}[\mathcal{O}_{\Sigma}] = \bigoplus_{n\geq 0} \mathbb{Z}[\mathcal{O}_{\Sigma_{n}}]$, understood as direct sum of additive categories. An object of $\mathbb{Z}[\mathcal{O}_{\Sigma}]$ consists of a \emph{finite} collection $\{P(i)\}$ of $\Sigma_{i}$-representations in $\mathbb{Z}[\mathcal{O}_{\Sigma}]$, i.e. $P(n)=0$ for any large enough $n$. We define  the $\infty$-category of \textbf{derived symmetric sequences} as the $\infty$-category of left modules over the additive category $\mathbb{Z}[\mathcal{O}_{\Sigma}]$, i.e. the $\infty$-category of additive functors $$\SSeq^{\gen}:=\Mod_{\mathbb{Z}[\mathcal{O}_{\Sigma}]}=\Fun^{\oplus}(\mathbb{Z}[\mathcal{O}]^{\op}, \Spt).$$ 

The additive category $\mathbb{Z}[\mathcal{O}_{\Sigma}] $, understood as a subcategory $\SSeq_{\heartsuit} \subset \SSeq$ of the $\infty$-category of symmetric sequences in spectra, is closed under tensor product and composition defined in Construction \ref{SSeq} (for explanation, see \cite[Construction 3.58]{BCN21}). Moreover, the duality equivalence $(-)^{\vee}: \mathbb{Z}[\mathcal{O}_{\Sigma}]^{\op} \simeq \mathbb{Z}[\mathcal{O}_{\Sigma}]$ allows us to define a new monoidal structure on $\mathbb{Z}[\mathcal{O}_{\Sigma}]$ by the formula: $$P\overline{\circ} Q := (P^{\vee} \circ Q^{\vee})^{\vee}=\bigoplus_{n=0}^{\infty} (P(n) \otimes Q^{\otimes^{n}})^{\Sigma_{n}}.$$

For any pair of objects $P, Q \in \mathbb{Z}[\Sigma]$, there is a Norm map $$\Nm:  (P(n) \otimes Q^{\otimes^{n}})_{\Sigma_{n}} \rightarrow  (P(n) \otimes Q^{\otimes^{n}})^{\Sigma_{n}}, $$ and this induces a lax monoidal functor $(\mathbb{Z}[\mathcal{O}_{\Sigma}], \circ ) \rightarrow (\mathbb{Z}[\mathcal{O}_{\Sigma}], \overline{\circ})$.
 
 \cite[Proposition 3.72]{BCN21} shows that the composition product $\overline{\circ}$ right-left extends to a monoidal structure on $\SSeq^{\gen}$, and restriction by the Norm $(\mathbb{Z}[\mathcal{O}_{\Sigma}], \circ ) \rightarrow (\mathbb{Z}[\mathcal{O}_{\Sigma}], \overline{\circ}) $ extends to a right lax monoidal functor $\Nm^{*}: (\SSeq^{\gen},\circ ) \rightarrow (\SSeq^{\gen}, \overline{\circ})  $ which is identical on objects. 
 
 \end{construction}
 
 \begin{notation}
 Let us fix some notation related to Construction \ref{SSeq_gen}.
 
 \begin{enumerate}
 \item $\SSeq^{\gen}_{0}$ is the monoidal category $(\mathbb{Z}[\mathcal{O}_{\Sigma}],\circ)$.
 \item $\SSeq^{\gen,\pd}_{0}$ is the monoidal category $(\mathbb{Z}[\mathcal{O}_{\Sigma}],\overline{\circ})$.
 \item $\SSeq^{\gen}$ is the presentable monoidal $\infty$-category $(\SSeq^{\gen},\circ)$. 
 \item $\SSeq^{\gen,\pd}$ is the presentable monoidal $\infty$-category $(\SSeq^{\gen},\overline{\circ})$. 
 \end{enumerate}
 \end{notation}
 
 \begin{rem}
To construct a compatible action of $\SSeq^{\gen}$ and $\SSeq^{\gen,\pd}$ on some stable $\infty$-category, it is sufficient to construct an action of the latter $\SSeq^{\gen,\pd}$ as the action of the former $\SSeq^{\gen}$ can then be obtained by restricting along the lax monoidal functor $\Nm^{*}: \SSeq^{\gen} \rightarrow \SSeq^{\gen,\pd}$.
 \end{rem}
 
Following observation of \cite[Proposition 3.72 (4)]{BCN21}, we record the following proposition. We say that a derived symmetric sequence $P \in \SSeq^{\gen}$ is \textbf{concentrated in arities $\geq 1$} if $P(0)\simeq 0$. 
 
 \begin{prop}\label{Norm_arity_1}
The functor $\Nm^{*}$ is a \emph{monoidal} equivalence on the full subcategories of derived symmetric sequences concentrated in arities $\geq 1$.

\end{prop}

\begin{proof}

Take two objects $P, Q \in R[\Sigma]$ which are trivial in arity $0$\footnote{In fact, in the subsequent argument it suffices to assume that only $Q$ is trivial in arity $0$.}, and consider the symmetric sequence $P(n) \otimes Q^{\otimes^{n}}$. We wish to prove that the map $\Nm_{\Sigma_{n}}: (P(n) \otimes Q^{\otimes^{n}})_{\Sigma_{n}} \rightarrow (P(n) \otimes Q^{\otimes^{n}})^{\Sigma_{n}}$ is an equivalence. To see this, we look at the formula for the arity $r$ module of the symmetric sequence $ Q^{\otimes^{n}}$:
 
 \begin{equation}\label{unwinding_symmetric_sequence}
  Q^{\otimes^{n}} (r) = \bigoplus_{i_{1}+...+i_{n}=r}   \Ind_{\Sigma_{i_{1}}\times ... \times \Sigma_{i_{n}}}^{\Sigma_{r}} Q(i_{1}) \otimes ... \otimes Q(i_{n})  .
 \end{equation}
 
By assumption, $Q(i) = 0$ unless $i\geq 1$. It implies that the right hand side in Equation \ref{unwinding_symmetric_sequence} is $0$ unless $n\leq r$, because if $n>r$, then in order for the sum $i_{1}+...+i_{n} $ to be equal $r$, at least one of the indices must be $0$, which forces the tensor product to vanish. In the case $n\geq r$, we claim that the right hand side of Equation \ref{unwinding_symmetric_sequence} is an induced $\Sigma_{n}$-representation. Indeed, consider a set of indices $i_{1},...,i_{n} \geq 1$ such that $i_{1}+...+i_{n}=r$. The sum in the formula \ref{unwinding_symmetric_sequence} contains as a summand the $\Sigma_{n}$-submodule

 \begin{equation}\label{induces_summand}
  \bigoplus_{\sigma \in \Sigma_{n}} \Ind_{\Sigma_{i_{\sigma(1)}} \times ... \times \Sigma_{i_{\sigma(n)}} }^{\Sigma_{r} } (  Q(i_{\sigma(1)})\otimes ... \otimes Q(i_{\sigma(n)})   ),   \end{equation}
 where the sum is taken with respect to all permuations of the indices. The $\Sigma_{n}$-submodule in Equation \ref{induces_summand} is $\Sigma_{n}$-induced. As the sum in Equation \ref{unwinding_symmetric_sequence} splits into the sum of $\Sigma_{n}$- submodules of the form \ref{induces_summand} for various collections $i_{1},...,i_{n}$, it follows that the right hand side of \ref{unwinding_symmetric_sequence} is an induced $\Sigma_{n}$-module, and the Norm map induces an isomorphism between its $\Sigma_{n}$-orbits and $\Sigma_{n}$-fixed points, or equivalently, its Tate fixed points module vanishes. But the full subcategory of $\Sigma_{n}$-representations whose Tate fixed points vanishes, is a $\otimes$-ideal in the category of all $\Sigma_{n}$-representations. Therefore, the Tate fixed points module of $X(n) \otimes Q^{\otimes^{n}}(r)$ also vanishes for $n\geq r$, and the Norm map $(P(n) \otimes Q^{\otimes^{n}}(r))_{\Sigma_{n}} \rightarrow (P(n) \otimes Q^{\otimes^{n}}(r))^{\Sigma_{n}} $ is an equivalence, as desired.
 \end{proof}

 \begin{defn}
 A \textbf{derived (pd) operad} is an algebra object in the monoidal $\infty$-category $\SSeq^{\gen}$ (respectively, in the monoidal $\infty$-category $\SSeq^{\gen,\pd}$).
 \end{defn}
 
 \begin{rem}\label{how_to_obtain_pd_operads}
 Proposition \ref{Norm_arity_1} implies that given a derived operad $O$ concentrated in arities $\geq 1$ gives rise to a pd operad. As an example, consider derived symmetric sequence $\mathcal{O}:=(0,\mathbb{Z},\mathbb{Z},...)$. We compute that
 
 $$
 \mathcal{O}\circ \mathcal{O}(k)\simeq \bigoplus_{n\geq 1}\bigl( \bigoplus_{ i_{1}+...+i_{k}=n} \mathbb{Z} \bigr).
 $$
The map $\mathcal{O}\circ \mathcal{O} \rightarrow \mathcal{O}$ given by the natural map $\bigoplus_{n\geq 1}\bigl( \bigoplus_{ i_{1}+...+i_{k}=n} \mathbb{Z} \bigr) \rightarrow \mathbb{Z}$ in every arity $k$, endows $\mathcal{O}$ with the structure of a derived operad. This derived operad is the \textbf{non-unital derived commutative operad}. As it vanishes in arity $0$, there is a corresponding deried pd operad which is called \textbf{the non-unital pd commutative operad}. 
 \end{rem}

Below we will construct an action of the monoidal $\infty$-category $\SSeq^{\gen,\pd}$ on $\Mod_{R}$ for a discrete commutative ring $R$.

\begin{construction}\label{pd_derived_action}

Let $R$ be a discrete commutative ring, and consider the $\infty$-category $\Mod_{R}$ and the full subcategory of finitely generated free modules inside all connective $R$-modules $\Mod_{R,0} \subset \Mod_{R,\geq 0}$.

The subcategory $\Mod_{R,0}$ has the following properties.

- For any object $X\in \Mod_{R,0}$, the dual $X^{\vee}$ is in $\Mod_{R,0}$ again;

- For any $X\in \fcat \Mod_{R,0}$ and any finitely generated free $\Sigma_{n}$-module $V$, the object $(V\otimes X^{\otimes{n}})_{\Sigma_{n}} \in \Mod_{R,\heartsuit}$ is in $\Mod_{R,0}$ again.

Together, these two conditions imply that for any finitely generated free $\Sigma_{n}$-module $V$, the functor $X \longmapsto (V \otimes X^{\otimes^{n}})^{\Sigma_{n}}$ on $\Mod_{R,\heartsuit}$ preserves the full subcategory $\Mod_{R,0} \subset \Mod_{R,\heartsuit}$. 

Therefore, there is a functor 

$$ \mathbb{Z}[\mathcal{O}_{\Sigma_{n}}] \xymatrix{\ar[r]&} \End(\Mod_{R,0})$$
sending $$V \longmapsto \bigl( X \longmapsto  (V(n) \otimes X^{\otimes^{n}})^{\Sigma_{n}} \bigr).   $$
Note that we switched invariants and coinvariants in the formula above by using the isomorphism

$$
 (V(n) \otimes X^{\otimes^{n}})^{\Sigma_{n}} \simeq \bigl( (V^{\vee}(n) \otimes (X^{\vee})^{\otimes^{n}})_{\Sigma_{n}}\bigr)^{\vee}
$$
together with two conditions listed above. Moreover, for each $n$, the functor $X \longmapsto  (V \otimes X^{\otimes^{n}})_{\Sigma_{n}}$ is additively polynomial of degree $n$, and it follows by duality that the functor $X \longmapsto (V \otimes X^{\otimes^{n}})^{\Sigma_{n}} $ is additively polynomial of degree $n$ again. Since each object $\{P(i)\}$ in $\mathbb{Z}[\mathcal{O}_{\Sigma}]$ is non-zero for only finitely many $i$, it follows $\SSeq^{\gen,\pd}$ acts on $\Mod_{R,0}$ by additively polynomial functors. To show that this action extends to an action of $\SSeq^{\gen,\pd}$ on $\Mod_{R}$, we shall use \cite[Proposition 3.48]{BCN21} which states that the functor $\Mod: \cat^{\add} \rightarrow \cat^{\St,\Sigma}$ from the $\infty$-category of small additive $\infty$-categories and additive functors $\cat^{\add}$ to the $\infty$-category $\cat^{\St,\Sigma}$ of presentable stable $\infty$-categories and sifted colimit preserving functors, refines to a symmetric monoidal functor $\Mod: \cat^{\add,\poly} \rightarrow \cat^{\St,\Sigma} $, where $\cat^{\add,\poly}$ is the $\infty$-category of small additive $\infty$-categories and additively polynomial functors between them.

\end{construction}

\begin{construction}\label{pd_derived_action_on_deriveds}
The preceeding discussion of the action of $\SSeq^{\gen,\pd}$ on $\Mod_{R}$ for any discrete commutative ring $R$ generalizes to any connective derived commutative ring. To this end, we have constructed a functor $$\xymatrix{\SSeq^{\gen,\pd}  \curvearrowright \Mod_{(-)} : \CAlg_{\mathbb{Z}}^{ \poly} \ar[r]& \Mod_{\SSeq^{\gen,\pd}} (\cat^{\St,\Sigma})}$$ which sends a polynomial ring $R$ to the $\infty$-category $\Mod_{R}$ endowed with the sifted colimit preserving action of $\SSeq^{\gen,\pd}$ on $\Mod_{R}$ constructed in Construction \ref{pd_derived_action}. Left Kan extending along the inclusion $\CAlg^{\poly}_{\mathbb{Z}} \subset \DAlg_{\mathbb{Z},\geq 0}$, we obtain a sifted colimit preserving functor $$\xymatrix{\SSeq^{\gen,\pd}  \curvearrowright \Mod_{(-)} : \DAlg_{\mathbb{Z},\geq 0} \ar[r]& \Mod_{\SSeq^{\gen,\pd}} (\cat^{\St,\Sigma})}$$ such that the composition with the forgetful functor $\Mod_{\SSeq^{\gen,\pd}}(\cat^{\St,\Sigma}) \rightarrow \cat^{\St,\Sigma} $ is equivalent to the functor of left modules $R \longmapsto \Mod_{R}$. The latter follows from the fact that the forgetful functor preserves sifted colimits and the functor $\Mod_{(-)}: \DAlg_{\mathbb{Z},\geq 0} \rightarrow \cat^{\St,\Sigma}$ is itself left Kan extended from $\CAlg^{\poly}_{\mathbb{Z}}$.
\end{construction}

\begin{defn}\label{LGamma}
Let $R$ be a connective derived ring. Using Remark \ref{how_to_obtain_pd_operads}, Construction \ref{pd_derived_action_on_deriveds} gives rise to a monad $\LGamma^{+}_{\R}\in \End^{\Sigma}(\Mod_{R}) $ on $\Mod_{R}$ whose underlying functor satisfies $\LGamma^{+}_{R} = \bigoplus_{n \geq 1} (X^{\otimes^{n}})^{\Sigma_{n}} $ for any $X\in \Mod_{R,0}$. We refer to this as \textbf{the free non-unital divided power (pd) algebra monad}, and algebras over this monad are called \textbf{non-unital derived divided power (pd) algebras}. We denote $$\DAlg_{R}^{\nonu,\pd}:=\Alg_{\LGamma_{R}^{+}}(\Mod_{R}). $$
\end{defn}

A map $R \rightarrow S$ of connective derived commutative rings gives a base change functor $-\otimes_{R} S: \DAlg^{\nonu,\pd}_{R} \rightarrow \DAlg^{\nonu,\pd}_{S}$, as well as restriction $\DAlg_{S}^{\nonu,\pd} \rightarrow \DAlg^{\nonu,\pd}_{R}.$

\begin{construction}\label{derived_syms_from_geometry}
Let $X$ be a derived prestack, i.e. a functor $X: \DAlg_{\mathbb{Z},\geq 0} \rightarrow \Spc$. Construction \ref{pd_derived_action_on_deriveds} implies that the $\infty$-category $\QCoh(X)$ of quasi-coherent sheaves on $X$ has an action of $\SSeq^{\gen,\pd}$. Indeed, this follows by writing the $\infty$-category $\QCoh(X)$ as the limit

$$
\QCoh(X) = \underset{\Spec R \rightarrow X}{\lim} \Mod_{R}
$$
along pull-back functors for all affines $\Spec A \rightarrow X$ mapping to $X$, and using the fact that the forgetful functor $\Mod_{\SSeq^{\gen,\pd}}(\cat^{\St,\Sigma}) \rightarrow \cat^{\St,\Sigma}$ preserves limits. Consequently, we can simply \emph{define} the $\infty$-category of derived commutative (respectively, derived non-unital divided power) algebra objects in $\QCoh(X)$ as the limit $$\DAlg^{(\nonu,\pd)}(\QCoh(X)) \simeq \underset{\Spec R \rightarrow X}{\lim} \DAlg^{(\nonu,\pd)}_{R}. $$
\end{construction}

In some cases we have to work with a symmetric monoidal $\infty$-category $\C$ which is not (or at least, not known to be) of the form $\QCoh(X)$ for a derived prestack $X$. In such a case, one approach to constructing an action of $\SSeq^{\gen,\pd}$ is to specify a $t$-structure $(\C_{\geq 0},\C_{\leq 0})$ and a choice of a subcategory $\C^{0}$ of compact generators in $\C$ satisfying certain properties which will be listed below (a derived algebraic context in the sense of \cite{R}). For such a symmetric monoidal $\infty$-category $\C$, we will be able to construct a sifted colimit preserving action of $\SSeq^{\gen,\pd}$ on $\C$. This should be compared with the construction of derived algebras in a derived algebraic context in \cite{R}, albeit the types of $\infty$-categories we consider below are slightly different from derived algebraic contexts of Raksit.

\begin{construction}\label{derived_alg_contexts}
Let $\C$ be a symmetric monoidal $\infty$-category equipped with a compatible $t$-structure $(\C_{\geq 0}, \C_{\leq 0})$ and a full subcategory $\C_{0}\subset \C_{\heartsuit} \subset \C_{\geq 0}$ consisting of compact projective generators of $\C_{\geq 0}$. We require that the $t$-structure is left complete, and the full subcategory $\C_{0}$ is closed in $\C_{\heartsuit}$ under the action of $\SSeq^{\gen,\pd}_{0}$, \emph{and} the action of $\SSeq^{\gen}$. Using the strategy of Construction \ref{pd_derived_action}, there is a sifted colimit preserving action of the monoidal $\infty$-category $\SSeq^{\gen, \pd}$ on $\C$. Moreover, this construction is functorial in $\C$ in the following sense. Notice that for $\C$ as above, there is an equivalence $\C\simeq \Mod(\C_{0})$. Consequently, using \cite[Proposition 3.48]{BCN21} again as in Construction \ref{pd_derived_action}, we obtain a functor 

$$\Mod_{\SSeq^{\gen,\pd}_{0} }(\cat^{\add,\poly}) \xymatrix{\ar[r]&}  \Mod_{\SSeq^{\gen,\pd}}(\cat^{\St,\Sigma}).   $$

Therefore, the action of $\SSeq^{\gen,\pd}$ on $\C$ is functorial in symmetric monoidal right $t$-exact functors $F: \C \rightarrow \D$ such that $F(\C_{0}) \subset \D_{0}$. The same construction applies to the action of $\SSeq^{\gen}$. Note that if we drop the condition that $\C_{0}$ is closed under the action of $\SSeq^{\gen}_{0}$, we still get a $\SSeq^{\gen}$-action on $\C$ by restricting along $\Nm^{*}: \SSeq^{\gen} \rightarrow \SSeq^{\gen,\pd}$, however, the naturality of this action could be more subtle without the aforementioned condition.
\end{construction}

\begin{observation}\label{right_left_extended_monads}
In studying the $\LGamma^{+}$-monad defined in Definition \ref{LGamma} in such a context as in Construction \ref{derived_alg_contexts}, restriction $\LGamma^{+}|_{\C_{0}}$ to $\C_{0} \subset \C$ does not yield a monad. However, as each $\LGamma^{n}$ preserves $\C_{0}$, it follows that $\LGamma^{+}$ preserves the full subcategory $\Ind(\C_{0}) \subset \C$, and hence $\LGamma^{+}|_{\Ind(\C_{0})}: \Ind(\C_{0}) \rightarrow \Ind(\C_{0})$ yields a monad. Moreover, the original monad $\LGamma^{+}$ on $\C$ is uniquely determined by $\LGamma^{+}|_{\Ind(\C_{0})} $ in the following sense. Let $\End_{\Sigma, RL}^{\Ind(\C_{0})}(\C) \subset \End_{\Sigma}(\C) $ be the subcategory of endofunctors preserving the full subcategory $\Ind(\C_{0}) \subset \C$ which are right-left extended from $\C_{0} \subset \C$. Then the restriction functor

$$
\xymatrix{\End^{\Ind(\C_{0})}_{\Sigma, RL}(\C)      \ar[r]^-{\sim}& \End_{\omega}(\Ind(\C_{0}))  }
$$
is an equivalence. We will use this observation for constructing maps of monads right-left extended from $\C_{0}$. 

The same applies to the monad $\LSym: \C \rightarrow \C$, as each $\LSym^{n}: \C \rightarrow \C$ preserves the full subcategory $\C_{0}$.
\end{observation}

Let us finish this section with some observations on how lax symmetric monoidal functors are incorporated in this setting.

\begin{construction}
Let $\A_{0}$ be a small additive $\infty$-category, consider the $\infty$-category $$\Mod_{\A^{0}}(\cat^{\add,\poly})_{\lax}:=\Fun(\B\A_{0}, \cat^{\add,\poly})_{\lax}.$$ Composing with the monoidal functor $\Mod: \cat^{\add,\poly} \rightarrow \cat^{\St,\Sigma}$, we obtain a functor $$\Fun(\B\A_{0}, \cat^{\add,\poly})_{\lax} \rightarrow \Fun(\B\A_{0}, \cat^{\St,\Sigma})_{\lax}. $$ Notice that for any $\C \in \cat^{\St,\Sigma} $, the endomorphism object $\Fun_{\Sigma}(\C,\C)$ is a presentable stable monoidal $\infty$-category. Therefore, any monoidal functor $\A_{0} \rightarrow \Fun_{\Sigma}(\C,\C) $ extends to a continuous monoidal functor $\A:=\Mod(\A_{0}) \rightarrow \Fun_{\Sigma}(\C,\C) $, and we obtain a functor $$\Mod_{\A_{0}}(\cat^{\add,\poly})_{\lax} \rightarrow \Mod_{\A}(\cat^{\St,\Sigma})_{\lax}, $$

where $\Mod_{\A_{0}}(\cat^{\add,\poly})_{\lax}:= \Fun(\textbf{BA},\cat^{\St,\Sigma})_{\lax}$.
\end{construction}

A consequence of this construction is the following example.

\begin{ex}\label{obtaining_diagrams}
Assume $$I \rightarrow \Mod_{\A_{0}}(\cat^{\add,\poly})_{\lax}$$ is a diagram. Then applying $\Mod: \cat^{\add,\poly} \rightarrow \cat^{\St,\Sigma}$, we obtain a diagram $$I \rightarrow \Mod_{\A}(\cat^{\St,\Sigma})_{\lax}.$$
\end{ex}

We will sometimes apply this example in the setting $\A_{0}:=\SSeq^{\gen,\pd}_{0}$ and $\A:=\SSeq^{\gen,\pd}$.

\subsection{Recollection on gradings and filtrations.}

\begin{defn}

We let $\mathbb{Z}^{\ds}$ be the discrete category whose objects are integers $n\in \mathbb{Z}$ and all morphisms are identities, and $\mathbb{Z}_{\leq *}$ be the nerve of the partially ordered set of integers. Note that both $\mathbb{Z}^{\ds}$ and $\mathbb{Z}_{\leq *}$ are symmetric monoidal categories with respect to the operation of addition. In what follows, we will work over a connective derived commutative ring $R$.

\begin{enumerate}

\item We define the $\infty$-category of \textbf{graded objects} in $\Mod_{R}$, denoted $\Gr\Mod_{R}$ to be the $\infty$-category of functors $$\Gr\Mod_{R}:=\Fun((\mathbb{Z}^{\ds})^{\op},\Mod_{R}).$$ An object $X \in \Gr\Mod_{R}$ is the same data as a collection of objects $\{X^{n}\}_{n\in \mathbb{Z}}$ in $\Mod_{R}$ parametrized by all integers $n$. We will often denote a graded object $X$ as $X^{\bullet}$ to emphasize the grading. The $\infty$-category $\Gr\Mod_{R}$ carries a \emph{Day convolution symmetric monoidal structure} with tensor product given by the formula

\begin{equation}\label{graded_tensor_product}
(X^{\bullet}\otimes Y^{\bullet})^{n} := \bigoplus_{p+q=n} X^{p}\otimes Y^{q}.
\end{equation}

\item W let the $\infty$-category of \textbf{filtered objects}\footnote{To be more precise, we are dealing with \emph{decreasing filtrations} here.} $\Fil\Mod_{R}$ to be the $\infty$-category of functors $$ \Fil\Mod_{R}:=\Fun((\mathbb{Z}_{\leq *})^{\op},\Mod_{R}).$$ An object $X\in \Fil\Mod_{R}$ is the data of a collection of objects $\{X^{\geq n}\}_{n \in \mathbb{Z}}$ in $\Mod_{R}$, and maps $X^{\geq m} \rightarrow X^{\geq n}$ for any pair of integers $m$ and $n$ such that $m\geq n$. We can depict an object $X \in \Fil\Mod_{R}$ as a diagram

$$
\{ ... \rightarrow X^{\geq 2} \rightarrow X^{\geq 1} \rightarrow X^{\geq 0} \rightarrow X^{\geq -1} \rightarrow X^{\geq -2} \rightarrow ...\}
$$
in $\Mod_{R}$. We will often denote the data of a filtered object by the symbol $X^{\geq \star}$ to emphasize the filtration. The $\infty$-category $\Fil\Mod_{R}$ has a \emph{Day convolution symmetric monoidal structure} with tensor product given by the formula

\begin{equation}\label{filtered_tensor_product}
(X^{\geq \star}\otimes Y^{\geq \star})^{n} :=\colim_{p+q\geq n} X^{\geq p} \otimes Y^{\geq q}.
\end{equation}

\item The \textbf{associated graded object} functor $\gr: \Fil\Mod_{R} \rightarrow \Gr\Mod_{R}$ is the functor sending a filtered object $X^{\geq \star} \in \Fil\Mod_{R}$ to a graded object $\gr^{\bullet}(X)$ with the $n$-th graded piece defined by the formula $$\gr(X)^{n} := X^{n}/X^{n+1}=\cofib(X^{n} \rightarrow X^{n+1}).$$ By \cite[Proposition 3.2.1]{L}, this functor is symmetric monoidal (alternatively, see Remark \ref{functors_geometrically}).

\item We define a \textbf{forgetful functor} $(\Fil \rightarrow \Gr): \Fil\Mod_{R} \rightarrow \Mod_{R}$ which sends a filtered object $X:=X^{\geq \star} $ to the underlying graded object whose weighted pieces are $(\Fil \rightarrow \Gr)(X)^{n}:= X^{\geq n}. $ This functor is right lax symmetric monoidal (see Remark \ref{functors_geometrically} for explanation)

\item  We let $\ins^{n}: \Mod_{R} \rightarrow \Gr\Mod_{R}$ be the \textbf{insertion functor} sending an object $X \in \Mod_{R}$ to the graded object $\ins^{n}(X)^{\bullet}$ which has $\ins^{n}(X)^{i}=X$ when $i=n$, and $\ins^{n}(X)^{i}=0$ if $i\neq 0$. It has a right adjoint \textbf{evaluation functor} $\ev^{n}: \Gr\Mod_{R} \rightarrow \Mod_{R}$ sending a graded object $X^{\bullet} $ to its $n$-th graded piece $X^{n} \in \Mod_{R}$. For $n=0$, the functor $\ins^{0}$ is symmatric monoidal, and the functor $\ev^{0}$ is \emph{right-lax symmetric monoidal}.

\item Similarly, we let $\ins^{n} : \Mod_{R} \rightarrow \Fil\Mod_{R}$ be the insertion functor sending an $R$-module $M$ to the filtered object which has $M$ in degree $n$, and is $0$ otherwise. It has a left adjoint $\gr^{n}: \Fil\Mod_{R} \rightarrow \Mod_{R}$ and a right adjoint $\ev^{n}$. The functor $\ins^{0}: \Mod_{R} \rightarrow \Fil\Mod_{R}$ is symmetric monoidal, and its right adjoint $\ev^{0}: \Fil\Mod_{R} \rightarrow \Mod_{R}$ is lax symmetric monoidal.

\item Let $X=X^{\bullet} \in \Gr\Mod_{R}$. We define a new graded object $X(i)$, called the \textbf{$i$-th twist} of $X$, by the formula $X(i)^{n} = X^{n-i}$. For $X\in \Mod_{R}$, we will also often denote $X(i):=\ins^{i}(X)$.

\item There exists a \textbf{constant filtration} functor $\const: \Mod_{R} \rightarrow \Fil\Mod_{R} $ which sends any $R$-module $M$ to the filtered object 

$$
\xymatrix{ \const(M):= \{... \ar[r]^-{\Id} &M \ar[r]^-{\Id}  &  M \ar[r]^-{\Id} & M \ar[r]^-{\Id} & ... \} .}
$$
It has a left adjoint functor sending a filtered object $X^{\geq \star} $ to the colimit $\underset{\mathbb{Z}^{\op}_{\leq *}}{\colim} \:X^{\geq \star} \in \Mod_{R}$. We will often use the notation $|X|= \underset{\mathbb{Z}^{\op}_{\leq *}}{\colim}\: X^{\geq \star}$.

\item The $\infty$-category $\Gr\Mod_{R}$ contains a full symmetric monoidal subcategory $\Gr^{\geq 0} \Mod_{R}\subset \Gr\Mod_{R}$ consisting of \textbf{non-negatively graded objects}, i.e. graded objects $X^{\bullet}$ with $X^{n}=0$ for $n<0$.

\item A filtered object $X^{\geq \star}$ is \textbf{non-negatively filtered} if all maps $X^{n} \rightarrow X^{n-1}$ are equivalences for $n \leq 0$. There is a symmetric monoidal subcategory $\Fil^{\geq 0} \Mod_{R} \subset \Fil \Mod_{R}$ spanned by \textbf{non-negatively filtered objects} in all filtered objects.

\end{enumerate}

In the next sections, we will mainly work with non-negative filtrations/gradings. Therefore, by a slight abuse of notation, we will often speak of non-negatively filtered/graded objects as simply filtered/graded objects.

\end{defn}

 \begin{construction}\label{filtrations_geometrically}
 Let us work in the case $R=\mathbb{Z}$. Let $\mathbb{Z}[t]$ be the free commutative ring on one generator, and assign the element $t$ degree $-1$. By \cite[Proposition 3.1.6]{L}, there is a symmetric monoidal \textbf{Rees equivalence} (or \textbf{Rees construction}),
 
 \begin{equation}\label{Fil_as_t}
\Rees: \Fil \Mod_{\mathbb{Z}} \xymatrix{\ar[r]^-{\sim}&} \Mod_{\mathbb{Z}[t]}(\Gr \Mod_{\mathbb{Z}}).
\end{equation}
This equivalence sends a filtered object $X^{\geq \star}$ to the underlying graded object with a graded action of the graded algebra $\mathbb{Z}[t]$ induced by the filtration. 

The perspective of the Rees construction is closely related with the stacky approach to filtrations. Let $\mathbb{G}_{m}$ be the multiplicative group scheme, and $\B\mathbb{G}_{m}$ the classifying stack of $\mathbb{G}_{m}$. Then there is an equivalence of symmetric monoidal $\infty$-categories (see \cite[Theorem 4.1]{Maul21})

$$
\Gr\Mod_{\mathbb{Z}} \simeq \QCoh(\B\mathbb{G}_{m}).
$$

Also, let $\mathbb{A}^{1}/\mathbb{G}_{m}$ be the quotient of affine line by the standard $\mathbb{G}_{m}$-action. There is an equivalence of symmetric monoidal categories:
 
 \begin{equation}\label{A/G_as_R[t]}
 \QCoh(\mathbb{A}^{1}/\mathbb{G}_{m}) \simeq \Mod_{\mathbb{Z}[t]} (\Gr\Mod_{\mathbb{Z}}).
 \end{equation}
 
Combining equivalences \ref{Fil_as_t} and \ref{A/G_as_R[t]}, we obtain an equivalence
 $$
\xymatrix{\Fil\Mod_{\mathbb{Z}} \simeq  \QCoh(\mathbb{A}^{1}/\mathbb{G}_{m}) } 
 $$
 giving a description of the $\infty$-category of filtered objects as quasi-coherent sheaves on a stack.
 
For a connective derived commutative algebra $R$, we have $$\Fil\Mod_{R} \simeq \Fil\Mod_{\mathbb{Z}} \underset{\Mod_{\mathbb{Z}}}{\otimes} \Mod_{R} \simeq \QCoh(\mathbb{A}^{1}/\mathbb{G}_{m} \times \Spec R). $$
 
 It follows from Construction \ref{derived_syms_from_geometry} that $\Fil\Mod_{R}$ supports the notion of derived commutative (pd) algebras therein.
 
 \end{construction}
 
 \begin{rem}\label{functors_geometrically}
From the point of view of Construction \ref{filtrations_geometrically}, the functor $\gr: \Fil\Mod_{R} \rightarrow \Gr\Mod_{R}$ is the pullback of quasi-coherent sheaves along the inclusion of the $0$-section $\B\mathbb{G}_{m} \rightarrow \mathbb{A}^{1}/\mathbb{G}_{m}$. The functor $(\Fil \rightarrow \Gr): \Fil\Mod_{R} \rightarrow \Gr\Mod_{R}$ can be interpreted as the push-forward of quasi-coherent sheaves along the map $\mathbb{A}^{1}/\mathbb{G}_{m} \rightarrow \B\mathbb{G}_{m}. $ It follows from this interpretation that $\gr$ is symmetric monoidal, and $(\Fil \rightarrow \Gr)$ is right lax symmetric monoidal.
 \end{rem}
 
 \begin{construction}\label{t_structures}
For a general connective derived commutative ring $R$, we can construct several different $t$-structures on $\Gr\Mod_{R}$. We will record the construction of two of them.

\begin{enumerate}

\item The \textbf{neutral} $t$-structure on $\Gr\Mod_{R}$ has connective part  $\Gr\Mod_{R,\geq 0}$ spanned by graded objects $X^{\bullet}$ for which $X^{n} \in \Mod_{R,\geq 0}$  for all $n$. The coconnective part of the neutral $t$-structure consists of objects $X^{\bullet}$ for which $X^{n} \in \Mod_{R,\leq 0}$ for all $n$. The neutral $t$-structure on $\Gr\Mod_{R}$ has the heart $\Gr\Mod_{R,\heartsuit}$ consisting of graded objects in the abelian category of $t$-discrete objects in $\Mod_{R}$.

\item Another important $t$-structure on $\Gr\Mod_{R}$ is the \textbf{negative} $t$-structure. It has the connective part $\Gr\Mod_{R,\geq 0^{-}}$ spanned by graded objects $X^{\bullet}$ which have $X^{n} \in \Mod_{R,\geq -n}$ for any $n$, and the coconnective part spanned by objects with $X^{n} \in \Mod_{R,\leq -n}$ for any $n$. The heart of this $t$-structure is equivalent to the $\infty$-category of graded objects in $\Mod_{R,\heartsuit}$ again, albeit in a non-tautological way: a typical $t$-discrete object $X^{\bullet}$ in the negative $t$-structure has each $X^{n} \in \Mod_{R},[-n,-n],$ which is a discrete $R$-module concentrated in degree $-n$. 

\end{enumerate}

\end{construction}

\begin{rem}\label{Koszul_sign_rule}

The negative $t$-structure is compatible with the symmetric monoidal structure. Indeed, given two objects $X^{\bullet} , Y^{\bullet} \in \fcat \Gr\Mod_{R,\geq 0^{-}}$, we have $X^{p} \otimes Y^{q} \in \Mod_{R,\leq -p-q} $ since $X^{p} \in \Mod_{R,\leq -p}$ and $Y^{q} \in \Mod_{R,\leq -q} $, and the formula \ref{graded_tensor_product} implies that $X^{\bullet}\otimes Y^{\bullet}  $ is connective in the negative $t$-structure. In particular, the heart $\Gr\Mod_{R,\heartsuit^{-}}$ becomes a symmetric monoidal $\infty$-category. 
\end{rem}

Recall that to define a discrete additive symmetric monoidal category $\C$, we need to give a bilinear functor $-\otimes -: \C \times \C \rightarrow \C$, and a natural $\Sigma_{2}$-action on the tensor product $X \otimes Y$ of any pair of elements (a "braiding"), which is subject to the constraints listed, for instance, in \cite[Chapter XI, 1]{ML}. The derived $\infty$-category $\D(\C):=\Fun^{\oplus}(\C^{\op}, \Spt )$ then acquires a symmetric monoidal structure via right-left extension.

\begin{construction}
In this construction, we will define the so-called \textbf{Koszul sign rule} symmetric monoidal structure on $\Gr\Mod_{\mathbb{Z}}$. First, we define a tensor product on the abelian category $\Gr\Mod_{\mathbb{Z},\heartsuit}$ such that for two objects $X, Y \in \Gr\Mod_{\mathbb{Z}}$, the tensor product of $X$ and $Y$ coincides with the usual tensor product $X\otimes Y$ of graded modules, but the action of $\Sigma_{2}$ is twisted by the sign permutation, i.e. for an element $x\otimes y \in X\otimes Y$, the generator $\sigma \in \Sigma_{2}$ acts by $\sigma(x \otimes y) = -y\otimes x$. This defines a symmetric monoidal structure on $\Gr\Mod_{\mathbb{Z},\heartsuit}$ which restricts to a symmetric monoidal structure on the additive category $\Proj^{\omega}(\Gr\Mod_{\mathbb{Z},\heartsuit})$ of compact projective objects in $\Gr\Mod_{\mathbb{Z},\heartsuit}$, and right-left extends to a symmetric monoidal structure on $\Gr\Mod_{\mathbb{Z}}$.
\end{construction}

It turns out that the heart of the negative $t$-structure is equivalent to $\Gr\Mod_{R,\heartsuit}$ with Koszul sign rule symmetric monoidal structure via the following construction.

\begin{construction}\label{shear}
The $\infty$-category $\Gr\Mod_{R}$ carries a \textbf{shear} auto-equivalence $\xymatrix{ [\bullet] :\Gr\Mod_{R} \ar[r]^-{\sim} & \Gr\Mod_{R} }$ defined by the formula $(X[\bullet])^{i} = X^{i}[i] $ for any graded object $X:=X^{\bullet}$. 

\end{construction}

Below will establish the main properties of the shear equivalence. These statements are well-known. For instance, Proposition \ref{décalage_LSym} is essentially an equivalent way of formulating property 2) of Proposition \ref{shear_1} below, and it was already established by L.Illusie in the text \cite{I1}.

\begin{prop}\label{shear_1}

The shear construction has the following properties:

\begin{enumerate}

\item It is left and right $t$-exact with respect to the standard $t$-structure on the target, and the \emph{negative} $t$-structure on the source.  

\item It intertwines the  \emph{Koszul sign rule} symmetric monoidal structure on $\Gr\Mod_{\mathbb{Z}}$ with the standard symmetric monoidal structure. In particular, it induces an equivalence of abelian symmetric monoidal categories $\Gr\Mod_{\mathbb{Z},\heartsuit} \simeq \Gr\Mod_{\mathbb{Z},\heartsuit^{-}}$, where the left hand side is endowed with the Koszul sign rule tensor product.

\end{enumerate}
\end{prop}

\begin{proof}
Identification of $t$-structures follows immediately from definition of the negative $t$-structure and the formula for the shear equivalence. 

Since both symmetric monoidal structures are right-left extended the symmetric monoidal structures on the hearts, it suffices to show the statement in the following setting. Given two objects $X, Y \in \Proj^{\omega}(\Gr\Mod_{\mathbb{Z},\heartsuit})$, the Koszul sign rule tensor product $X\otimes Y$ carries a $\Sigma_{2}$-action twisted by sign permutation; we need to prove that the $\Sigma_{2}$-action on the graded object $(X\otimes Y)[\bullet]$ induced by functoriality, is equivalent to the usual $\Sigma_{2}$-action. In this situation, the $\Sigma_{2}$-action on $(X\otimes Y)[\bullet]$ is completely determined by the $\Sigma_{2}$-action on its homotopy groups. The statement then follows from the fact that the induced braiding on the homotopy groups of the tensor product of two objects, satisfies the Koszul sign rule.
\end{proof}

\begin{prop}\label{décalage_LSym}
Let $R$ be a connective derived commutative ring. There are equivalences 

$$
\LSym^{n}_{R}(X[1]) \simeq \LLambda^{n}_{R}(X)[n],
$$

$$
\LGamma^{n}_{R}(X[-1]) \simeq \LLambda^{n}_{R}(X)[-n].
$$
for any $n \geq 1$ and $X\in \Mod_{R}$.
\end{prop}

\begin{proof}
It is sufficient to prove either statement in the case $R$ is a polynomial ring. In this case, since all functors involved are derived from the full subcategory of finitely generated free modules, using base change from $\mathbb{Z}$, it suffices to prove that $\LSym_{\mathbb{Z}}^{n}(\mathbb{Z}^{d}[1]) \simeq \Lambda_{\mathbb{Z}}^{n}(\mathbb{Z}^{d})[n] $ for any $d$. This follows from property 2 of Proposition \ref{shear_1}. The statement for derived divided power algebra follows by duality, as exterior powers commute with taking duals.
\end{proof}

Given an $R$-module $X$, let $\LGamma_{R}(X):=R\oplus \LGamma^{+}_{R}(X)$. This is the unital augmented algebra corresponding to the non-unital algebra $\LGamma_{R}^{+}(X)$. Also, let $\LLambda_{R}(X)$ be the free derived exterior algebra on $X$. 

\begin{cor}\label{décalage_for_algebras}
Let $X\in \Mod_{R}$. There is an equivalence of derived graded $R$-algebras

$$
\LGamma_{R}(X[-1])[\bullet] \simeq \LLambda_{R}(X),
$$
where the right hand side is understood as a derived graded algebra with respect to the Koszul sign rule symmetric monoidal structure.
\end{cor}

Assume $X^{\geq \star}$ is a filtered object in $\Mod_{R}$. The the graded object $\gr^{\bullet}(X)$ can be endowed with an additional structure. Namely, for any $n$, there is a cofiber sequence

$$
X^{\geq n+1} \rightarrow X^{\geq n} \rightarrow \gr^{n}(X)
$$
which produces a boundary map $\gr^{n}(X) \rightarrow X^{\geq n+1}[1]$. Composing with the quotient map $X^{\geq n+1}[1] \rightarrow \gr^{n+1}(X)[1]$, we obtain a map $\gr^{n}(X) \rightarrow \gr^{n+1}(X)[1]$ or equivalently, a map $\gr^{n}(X)[-1] \rightarrow \gr^{n+1}(X)$. In other words, the associated graded object of a filtered object inherits a graded map $\gr(X)(1)[-1] \rightarrow \gr(X)$ encoding successive extensions in the filtration. This structure can be understood in terms of an action of a certain algebra, whose definition we will give below.
 
 \begin{defn}\label{D_-}
Let $R$ be a connective derived commutative ring. Following \cite{R}, we define a square-zero extension algebra $$\mathbb{D}_{-} : = R \oplus R[-1](1).$$

\end{defn}

\begin{rem}\label{LGamma_D_-}
Vanishing of higher exterior powers of a free module of rank $1$ together with Proposition \ref{décalage_for_algebras} imply that there is an equivalence $$\LGamma_{R}\bigl(R[-1](1)\bigr) \simeq \bigl(\LLambda_{R}(R) \bigr)[-\bullet] \simeq \bigl(R\oplus R(1) \bigr)[-\bullet] \simeq R \oplus R[-1](1).$$
The diagonal map $R \rightarrow R\oplus R$ induces a map $$\mathbb{D}_{-}=\LGamma_{R}\bigl(R[-1](1)\bigr) \rightarrow \LGamma_{R}\bigl(R[-1](1)\bigr) \otimes \LGamma_{R}\bigl(R[-1](1)\bigr) \simeq \mathbb{D}_{-}\otimes \mathbb{D}_{-},$$ which endows $\mathbb{D}_{-}$ with the structure of a \emph{cocommutative Hopf algebra} in $\Gr\Mod_{R}$. 
 \end{rem}
 
 \begin{defn}\label{DG_-_modules}
 For a connective derived commutative ring $R$, we define the $\infty$-category $\DG_{-}\Mod_{R}$ of \textbf{dg objects in $\Mod_{R}$}, or \textbf{dg $R$-modules},  as the $\infty$-category of graded modules over the graded algebra $\mathbb{D}_{-}$:
 
 $$
 \DG_{-}\Mod_{R} := \Mod_{\mathbb{D}_{-}}(\Gr \Mod_{R}).
 $$
 
Using the cocommutative Hopf algebra structure on $\mathbb{D}_{-}$ observed in Remark \ref{LGamma_D_-}, the $\infty$-category gets a symmetric monoidal structure making the forgetful functor $\DG_{-}\Mod_{R} \rightarrow \Gr\Mod_{R}$ symmetric monoidal. 
 \end{defn}
 
 \begin{rem}\label{gr_lifts}
 The functor $\gr: \Fil\Mod_{R} \rightarrow \Gr\Mod_{R}$ lifts to a symmetric monoidal functor $\gr: \Fil\Mod_{R} \rightarrow \DG_{-}\Mod_{R}$. To explain this, recall that in terms of the Rees equivalence, the functor $\gr: \Fil\Mod_{R} \rightarrow \Gr\Mod_{R}$ is equivalent to the base change $-\underset{R[t]}{\otimes} R: \Gr\Mod_{R[t]} \rightarrow \Gr\Mod_{R}$. Consequently, the associated graded object $\gr(X)$ of any filtered object $X$ carries a natural action of the algebra $\underline{\Hom}_{R[t]}(R,R)$, where the notation $\underline{\Hom}$ means the internal graded hom-object. The existence of a natural $\mathbb{D}_{-}$-action then follows from the equivalence of graded $R$-algebras
 
$$
 \underline{\Hom}_{R[t]}(R, R) \simeq \underline{\Hom}_{R}(R\underset{R[t]}{\otimes} R, R) \simeq \underline{\Hom}_{R}\bigl( \LSym_{R}(R(-1)[1]), R\bigr) \simeq \LGamma_{R}\bigl(R(1)[-1]\bigr) \simeq \mathbb{D}_{-},
$$
 \end{rem}

We can think of dg modules as representations of a derived algebraic group scheme whose algebra of distributions is $\mathbb{D}_{-}$. In representation theory, there are functors of \emph{homology} and \emph{cohomology}, or speaking in a more higher algebraic language, \emph{homotopy orbits} and \emph{homotopy fixed points} of a representation.
 
 \begin{construction}\label{homotopy_orbits_and_fixed_points}
 Let $\triv: \Gr\Mod_{R} \rightarrow \DG_{-}\Mod_{R}$ be the functor endowing a graded module with a trivial $\mathbb{D}_{-}$-action. In other words, it is restricting along the augmentation map $\mathbb{D}_{-} \rightarrow R$. There is an adjunction
 
$$
\xymatrix{ \Gr\Mod_{R} \ar[rrr]^-{\triv } &&& \ar@/_1.1pc/[lll]_{(-)_{h \mathbb{D}_{-} }  }  \ar@/^1.1pc/[lll]^{(-)^{ h \mathbb{D}_{-} }}  \DG_{-}\Mod_{R}. }
$$

For $X \in \DG_{-}\Mod_{R}$, we call $X_{h \mathbb{D}_{-}}$ \textbf{the $\mathbb{D}_{-}$-homotopy orbits object} of $X$, and $X^{ h \mathbb{D}_{-}}$ \textbf{homotopy fixed points object}.

 \end{construction}

 \begin{rem}
 Note that as the functor $\triv$ is symmetric monoidal, the right adjoint $(-)^{h \mathbb{D}_{-}}$ is right-lax symmetric monoidal. It is given explicitly by the formula $$X^{h \mathbb{D}_{-}}:= \underline{\Hom}_{\mathbb{D}_{-}}(R, X) ,$$ where the $\underline{\Hom}$ is the graded internal hom-object. Similarly as in Remark \ref{gr_lifts}, we can compute that there is an equivalence $$ R^{h \mathbb{D}_{-}} = \underline{\Hom}_{\mathbb{D}_{-}}(R,R) = R[t]$$ of graded algebras, where $t$ has degree $-1$. In particular, the graded object $X^{h \mathbb{D}_{-}}$ carries an operation of degree $-1$, which is the same data as a filtration. We denote $$(-)^{\geq \star}: \DG_{-}\Mod_{R} \xymatrix{\ar[r]&} \Fil \Mod_{R}$$ the resulting functor lifting $(-)^{h \mathbb{D}_{-}}$.
 \end{rem}

We record the following proposition without proof, which can be found in \cite[Theorem 3.2.14]{R}. Let $\widehat{\Fil\Mod}_{R} \subset \Fil\Mod_{R}$ be the $\infty$-category of \textbf{complete filtered objects} in $\Mod_{R}$, i.e. the full subcategory of $\Fil\Mod_{R}$ spanned by the objects $X^{\geq \star}$ which satisfy $\underset{n}{\lim} \: X^{\geq n}\simeq 0$.
 
 \begin{Theor}\label{FilMod_and_DGMod}
Let $R$ be a derived commutative ring. The functors $\gr$ and $(-)^{h \mathbb{D}_{-}}$ form a mutually inverse equivalence of symmetric monoidal $\infty$-categories
 
$$
\xymatrix{ \widehat{\Fil\Mod}_{R}   \ar@/^1.1pc/[rrr]^{\gr }  &&& \ar@/^1.1pc/[lll]^{(-)^{\geq \star}}  \DG_{-}\Mod_{R}. }
$$
 \end{Theor}

\begin{rem}
If $R$ is connective, then the algebra $\mathbb{D}_{-}$ is connective in the negative $t$-structure, therefore the negative $t$-structure on $\Gr\Mod_{R}$ extends to a $t$-structure on $\DG_{-}\Mod_{R}$ which we call \textbf{negative $t$-structure} as well. The heart of this $t$-structure is equivalent to the \emph{abelian category of cochain complexes} of $\pi_{0}(R)$-modules. Note that when $R$ is discrete, the algebra $\mathbb{D}_{-}$ is discrete in the negative $t$-structure too.
\end{rem}

\begin{rem}
Let us work over $R=\mathbb{Z}$. The structure of a $\mathbb{D}_{-}$-action can be understood in stacky terms as follows. Following A.Preygel \cite{Pr}, let $\Omega\mathbb{A}^{1}$ be the based loop space of $\mathbb{A}^{1}$ defined as a pull-back square in derived affine schemes
 
 $$
\Omega \mathbb{A}^{1}:= \{0\} \underset{\mathbb{A}^{1}} {\times} \{0\},
$$
where both maps $\{0\}=\Spec(\mathbb{Z}) \rightarrow \mathbb{A}^{1}$ are inclusions of the zero section. $\Omega\mathbb{A}^{1}$ becomes a group object in derived schemes with respect to composition of loops. The zero section in $\mathbb{A}^{1}$ is fixed by the action of $\mathbb{G}_{m}$, and induces an action on $\Omega\mathbb{A}^{1}$. Therefore, we can consider the quotient stack $\Omega \mathbb{A}^{1}/\mathbb{G}_{m}$ which is a derived group stack over $\mathbb{G}_{m}$. Let $\Rep(\Omega \mathbb{A}^{1}/\mathbb{G}_{m})$ be the category of graded $\Omega \mathbb{A}^{1}$-representations (understood as, say, comodules over the algebra of functions on $\Omega\mathbb{A}^{1}$). In these terms, the functor $\xymatrix{\gr: \Fil\Mod_{R} \ar[r] & \Gr\Mod_{\mathbb{D}_{-}}  } $ is given by the pull-back $$\xymatrix{- \underset{\mathbb{A}^{1}/\mathbb{G}_{m}}{\times} \B\mathbb{G}_{m}: \QCoh(\mathbb{A}^{1}/\mathbb{G}_{m}) \ar[r]& \Rep(\Omega \mathbb{A}^{1}/\mathbb{G}_{m})} $$ along the inclusion of zero section $i: \B\mathbb{G}_{m}\rightarrow \mathbb{A}^{1}/\mathbb{G}_{m}. $ 
 
\end{rem}

\begin{defn}\label{derived_filtered_algebras}
Let $R$ be a connective derived ring. Using Remark \ref{derived_syms_from_geometry} and Construction \ref{filtrations_geometrically}, there is a notion of derived commutative algebras in $\Fil\Mod_{R}$, $\Gr\Mod_{R}$, $\Fil^{\geq 0}\Mod_{R}$ and $\Gr^{\geq 0}\Mod_{R}$.

We let $$\DAlg(\Gr \Mod_{R}), \DAlg(\Fil \Mod_{R}), \DAlg(\Gr^{\geq 0} \Mod_{R}), \DAlg(\Fil^{\geq 0} \Mod_{R})$$ to be the respective $\infty$-categories of derived commutative algebras. We will often refer to them as \textbf{graded, filtered, non-negatively graded and non-negatively filtered}\footnote{Note that there is a potential abuse of terminology here as according to standard conventions, a "graded/filtered derived commutative algebra" could mean a graded/filtered object in the $\infty$-category of derived commutative algebras. However, in the case of algebras, this seems to be a standard terminology, so there is no risk of confusion. If we did need to speak about filtered/graded objects in derived commutative rings, we would call them by their longer name. } \textbf{derived commutative algebras} in $\Mod_{R}$ respectively, and we use respective notations $$\Gr\DAlg_{R}, \Fil\DAlg_{R}, \Gr^{\geq 0}\DAlg_{R}, \Fil^{\geq 0}\DAlg_{R}.$$ 

\end{defn}

Below we will give an equivalent, more explicit, description of the $\infty$-categories of derived algebras given in Definition \ref{derived_filtered_algebras} in the case when $R$ is discrete.

\begin{construction}\label{LSym_on_fils_from_t_structure}
Assume $R$ is a discrete ring. Let us unwind the natural $t$-structure on $\Fil\Mod_{R}$ obtained from the equivalence $\Fil\Mod_{R}\simeq \QCoh(\mathbb{A}^{1}/\mathbb{G}_{m})$. Using the affine map $p:\mathbb{A}^{1}/\mathbb{G}_{m} \rightarrow \B\mathbb{G}_{m}, $ an object $F\in \QCoh(\mathbb{A}^{1}/\mathbb{G}_{m})$ is connective iff its push-forward $p_{*}F$ is connective as a quasi-coherent sheaf on $\B\mathbb{G}_{m}$. Denote the graded components of $p_{*}F$ as $\{F^{i}\}_{i\in \mathbb{Z}}$. Pulling-back along the map $i: \Spec(R) \rightarrow \B \mathbb{G}_{m}$, we obtain that $i^{*}p_{*}F$ is connective as an $R$-module. But we can identify $i^{*}p_{*}F \simeq F^{0}$, and $F^{0}$ is a connective $R$-module. Similarly, by shifting the grading, we obtain that all $F^{i}$'s must be connective $R$-modules. Therefore, the natural $t$-structure on $\Fil\Mod_{R} \simeq \QCoh(\mathbb{A}^{1}/\mathbb{G}_{m})$ has the connective part consisting of filtered objects $F^{\geq \star}$ which are connective in all degrees. This $t$-structure is compatible with the symmetric monoidal structure, and is left and right complete. 

Let $\Fil\Mod_{R,0} \subset \Fil\Mod_{R} $ be the subcategory spanned by direct sums of objects of the form $R(i)$ for all $i \in \mathbb{Z}$. This is an additive subcategory of $\Fil\Mod_{R,\heartsuit} $ which is closed under the action of $\SSeq^{\gen}_{0}$. Consequently, using Construction \ref{derived_alg_contexts}, we obtain an action of $\SSeq^{\gen}$ on $\Fil\Mod_{R}$. This action has the property that for any map $f:\Spec A \rightarrow \mathbb{A}^{1}/\mathbb{G}_{m}$, the pull-back $f^{*}:\Fil\Mod_{R}\simeq  \QCoh(\mathbb{A}^{1}/\mathbb{G}_{m}) \rightarrow \Mod_{A}$ is $\SSeq^{\gen}$-linear. In the same way, we construct an action of $\SSeq^{\gen,\pd}$ on $\Fil\Mod_{R}$. Therefore, we obtain two monads on $\Fil\Mod_{R}$:

\begin{enumerate}
\item The \textbf{free filtered derived algebra monad} $\LSym^{\geq \star}_{R}: \Fil\Mod_{R} \rightarrow \Fil\Mod_{R} $;

\item The \textbf{free non-unital derived divided power algebra monad} $\LGamma^{+,\geq \star}_{R}: \Fil\Mod_{R} \rightarrow \Fil\Mod_{R} $.

\end{enumerate}

\end{construction}

\begin{rem}\label{right_left_extended_LSym_on_fils}
The monad $\LSym^{\geq \star}_{R}: \Fil\Mod_{R} \rightarrow \Fil\Mod_{R}$ restricts to a filtered-colimit preserving endofunctor $\Sym^{\geq \star}_{R,\heartsuit}: \Ind(\Fil\Mod_{R,0}) \rightarrow \Ind(\Fil\Mod_{R,0})$ and is uniquely determined by it according to Observation \ref{right_left_extended_monads}. The same applies to the monad $\LGamma^{+,\geq \star}_{R}$.
\end{rem}

Let us mention a concrete description of the category $\Ind(\Fil^{\geq 0}\Mod_{R,0})$.

\begin{prop}
The category $\Ind(\Fil^{\geq 0}\Mod_{R,0})$ is spanned by objects $F^{\geq \star} \in \Fil^{\geq 0}\Mod_{R} $ where $F^{0}$ is a flat $R$-module, and each map $F^{\geq i+1} \hookrightarrow F^{i}$ is injective with flat cokernel.
\end{prop}

The notion of filtered/graded derived $R$-algebra can be extended to the non-connective case as follows.

\begin{rem}
Let $R$ be a derived ring, and consider $R$ as a filtered (graded) derived ring inserted in degree $0$, extended by $0$ in positive degrees, and in a constant way in negative degrees. Then we can define $$\Fil\DAlg_{R}:= (\Fil\DAlg_{\mathbb{Z}})_{R/}, \;\; \Gr\DAlg_{R}:= (\Gr\DAlg_{\mathbb{Z}})_{R/}. $$

This recovers the previous definition when $R$ is connective, and extends it immediately to the not necessarily connective case.
\end{rem}

\section{Filtrations and divided power algebras.}

\subsection{Smith ideals.}

\begin{construction}\label{Day_arrows}

Below we will introduce two constructions which feature prominently in the text.

\begin{enumerate}

\item Let $\Delta^{1}$ be the $1$-simplex considered as a category with two objects $0,1$ and a single non-trivial arrow $0\rightarrow 1$. There is a symmetric monoidal structure $\ast: \Delta^{1}\times \Delta^{1}\rightarrow \Delta^{1} $ on $\Delta^{1}$ defined as:

\begin{equation}\label{join_product}
0\ast 0 = 0, 0\ast 1 = 1, 1\ast 0 = 1, 1\ast 1 = 1
\end{equation}

 Fix a (connective, derived) commutative ring $R$. Then the $\infty$-category $\Fun((\Delta^{1})^{\op},\Mod_{R})$ is a symmetric monoidal $\infty$-category with respect to the \emph{Day convolution symmetric monoidal structure}. The formula for the tensor product of two objects $X^{1}\rightarrow X^{0}$ and $Y^{1}\rightarrow Y^{0}$ is as follows:

$$
(X^{1}\rightarrow X^{0}) \otimes (Y^{1}\rightarrow Y^{0}) \simeq X^{0}\otimes Y^{1} \bigsqcup_{X^{1}\otimes X^{1}} X^{1}\otimes Y^{0} \rightarrow X^{0}\otimes Y^{0}.
$$

We use the notation $$\Mod^{\Delta^{1}_{\vee}}_{R} :=\Fun((\Delta^{1})^{\op},\Mod_{R})$$ considered with Day convolution symmetric monoidal structure satisfying the formula above. 

\item Let $\{0,1\}$ be the discrete 2 point category endowed with the symmetric monoidal structure defined on objects by the same formulas as \ref{join_product}. The $\infty$-category $\Fun(\{0,1\}^{\op}, \Mod_{R})$ gets a Day convolution symmetric monoidal structure with tensor product

$$
(X^{1}, X^{0}) \otimes (Y^{1}, Y^{0}) \simeq (X^{0} \otimes Y^{1} \oplus X^{1} \otimes Y^{1} \oplus X^{1} \otimes Y^{0}, X^{0}\otimes Y^{0}).
$$

We denote $$\Mod^{\{0,1\},\vee}_{R}:=\Fun(\{0,1\}^{\op}, \Mod_{R}) $$ considered with Day convolution symmetric monoidal structure.

\end{enumerate}

We will refer to both symmetric monoidal structures defined in the construction above as \textbf{"push-out symmetric monoidal structures"}.
\end{construction}

\begin{defn} Let $R$ be a (connective, derived) commutative ring.
\begin{enumerate}
\item Let $$\CAlg_{R}^{\Delta^{1}_{\vee}}:= \CAlg(\Mod^{\Delta^{1}_{\vee}}_{R})$$ be the $\infty$-category of $\mathbb{E}_{\infty}$-algebras in $\Mod^{\Delta^{1}_{\vee}}_{R}$.

\item Similarly, let $$\CAlg_{R}^{\{0,1\},\vee}:= \CAlg(\Mod_{R}^{\{0,1\},\vee}) $$ be the $\infty$-category of commutative algebra objects in $\Mod_{R}^{\{0,1\},\vee}$.

\end{enumerate}
\end{defn}

In what follows we will clarify the meanings of these two $\infty$-categories. As a warm-up, let us give a description of discrete objects in $ \CAlg_{R}^{\Delta^{1}_{\vee}}$ and $ \CAlg_{R}^{\{0,1\},\vee}$ respectively.

\begin{defn}
Recall the classical notion of a Smith ideal\footnote{This notion was introduced by Jess Smith and was reviewed and used recently in the texts \cite{BI22}, \cite{Mao21}, \cite{Dr20} and others.}. Let $A$ be a discrete commutative ring. A \textbf{Smith ideal} in $\Mod_{A}$ is an $A$-module $M$ endowed with an $A$-linear map $i: I\rightarrow A$ such that $i(x)y=xi(y)$. The map $I\otimes_{A} I \rightarrow I$ given by sending $(x,y) \longmapsto i(x)y$ defines a \textbf{non-unital $A$-algebra} structure on $I$, and the quotient $A/I : = \cofib(I\rightarrow A)$ has a structure of a simplicial commutative ring, which is in general non-discrete unless $I \rightarrow A$ is injective.

\end{defn}

\begin{prop}\label{pairs}
Let $R$ be a discrete commutative ring. Let $\CAlg_{R,\heartsuit}^{\Delta^{1}_{\vee}}$ and $\CAlg_{R,\heartsuit}^{\{0,1\},\vee}$ be the subcategories of $\CAlg_{R}^{\Delta^{1}_{\vee}}$ and $\CAlg_{R}^{\{0,1\},\vee}$ respectively, consisting of discrete objects with respect to the canonical $t$-structure on $\Mod_{R}$.

\begin{enumerate}

\item The category $\CAlg_{R,\heartsuit}^{\Delta^{1}_{\vee}}$ is equivalent to the category of pairs $(I \rightarrow A)$ where $A$ is a commutative $R$-algebra, and $i: I\rightarrow A$ is a Smith ideal in $A$.

\item The category $\CAlg_{R,\heartsuit}^{\{0,1\},\vee}$ is equivalent to the category of pairs $(A,B)$ where $A$ is a commutative $R$-algebra, and $B$ is a non-unital $A$-algebra.

\end{enumerate}
\end{prop}

\begin{proof}
\begin{enumerate}
\item 
A unital commutative algebra structure on an object $i:A^{1}\rightarrow A^{0}$ amounts to giving a unit map $\epsilon: (0\rightarrow R) \rightarrow (A^{1}\rightarrow A^{0})$, and a multiplication map

$$
\xymatrix{A^{1} \otimes A^{0}  \underset{A^{1}\otimes A^{1}}{\bigsqcup} A^{0}\otimes A^{1} \ar[d]_-{m^{1}} \ar[r]& A^{0} \otimes A^{0} \ar[d]^-{m^{0}} \\
A^{1} \ar[r]& A^{0}  }
$$
which satisfy associativity and commutativity conditions. The maps $m^{0}$ and $\epsilon^{0}$ give a unital, commutative and associative multiplication on $A^{0}$. The map $$m^{1}: A^{1}\otimes A^{0}  \bigsqcup_{A^{1}\otimes A^{1}} A^{0}\otimes A^{1} \rightarrow A^{1}$$ is equivalent to a pair of maps $A^{0} \otimes A^{1} \rightarrow A^{1}$ and $A^{1}\otimes A^{0}\rightarrow A^{1}$ which agree on $A^{1}\otimes A^{1}$. This defines a left and right actions of $A^{0}$ on $A^{1}$ satisfying the condition $i(x)y=xi(y)$. Commutativity and associativity constraints say that $A^{1}$ is symmetric $A^{0}$-bimodule, and $i: A^{1}\rightarrow A^{0}$ is an $A^{0}$-module map. This is the same as a Smith ideal.

\item The proof of this part goes the same way as in part 1 by unwinding the definition of a commutative algebra object in tensor product defined in Construction \ref{Day_arrows}, 2.

\end{enumerate}

\end{proof}

\begin{construction}
Let $R$ be a derived commutative ring, and we will consider the $\infty$-category of arrows $\Mod_{R}^{\Delta^{1}}$ witht the \emph{pointwise} symmetric monoidal structure. The functor $\cofib: \Mod^{\Delta^{1}_{\vee}}_{R} \rightarrow \Mod_{R}^{\Delta^{1}}$ sending $ (X \rightarrow Y) \longmapsto (Y \rightarrow \cofib(X\rightarrow Y) ) $ admits an oplax symmetric monoidal structure. To see this, we first note that the functor $\ev^{0}:  \Mod^{\Delta^{1}_{\vee}}_{R} \rightarrow \Mod_{R}, (X\rightarrow Y) \longmapsto Y$ is symmetric monoidal. Then it remains to construct an oplax symmetric monoidal structure on the functor $\Mod^{\Delta^{1}_{\vee}}_{R}\rightarrow \Mod_{R}, (X\rightarrow Y) \longmapsto \cofib(X\rightarrow Y). $ But this functor is the left adjoint of the symmetric monoidal inclusion $\Mod_{R} \rightarrow  \Mod^{\Delta^{1}_{\vee}}_{R}$ which sends $X\in \Mod_{R}$ to $(0\rightarrow X) \in \Mod^{\Delta^{1}_{\vee}}_{R}$, and therefore admits an oplax symmetric monoidal structure.
\end{construction}

\begin{prop}
The oplax symmetric monoidal structure on the functor $\cofib: \Mod^{\Delta^{1}_{\vee}}_{R} \rightarrow \Mod_{R}^{\Delta^{1}}$ is strict.
\end{prop}

\begin{proof}
We need to show the following. For two arrows $(X\rightarrow Y)$ and $(X'\rightarrow Y')$ the natural map 

$$
\xymatrix{ \cofib(X\otimes Y' \underset{X\otimes X'}{\bigsqcup} Y\otimes X' \ar[r]& Y\otimes Y') \rightarrow Y/X \otimes Y'/X'    }
$$
is an equivalence. This is equivalent to showing that the diagram

$$
\xymatrix{(Y\otimes Y') /(X\otimes X')     \ar[d] \ar[r]& Y/X \otimes Y' \ar[d]\\
Y \otimes Y'/X' \ar[r]& Y/X \otimes Y'/X'         }
$$
is a push-out square. This is in turn equivalent to showing that the map induced on horizontal cofibers is an equivalence. The cofiber of the lower horizontal arrow is equivalent to $X\otimes Y'/X' [1]$. The cofiber of the upper hozitontal arrow can be read of from the diagram

$$
\xymatrix{  X\otimes X' \ar[d]_-{f\otimes f'} \ar[r]^-{\Id\otimes f'} & X\otimes Y' \ar[d]^-{f\otimes \Id} \ar[r]& X\otimes Y'/X'\ar[d]\\
Y \otimes Y' \ar[d] \ar[r]_-{=} & Y \otimes Y'  \ar[d] \ar[r]& 0 \ar[d]\\
(Y\otimes Y')/(X\otimes X') \ar[r]& Y/X \otimes Y' \ar[r] &X\otimes Y'/X' [1].     }
$$
Therefore, the natural map $\cofib ( (Y\otimes Y')/(X\otimes X') \rightarrow Y/X \otimes Y' ) \rightarrow \cofib(Y\otimes Y'/X' \rightarrow Y/X \otimes Y'/X' )  $ is an equivalence, and this finishes the proof.
\end{proof}

\begin{prop}\label{push_out_pointwise_eq}
The functor $\cofib: \Mod^{\Delta^{1}_{\vee}}_{R} \rightarrow \Mod_{R}^{\Delta^{1}}$ is an equivalence of symmetric monoidal $\infty$-categories where the source is endowed with the push-out symmetric monoidal structure, and the target with the pointwise symmetric monoidal structure.
\end{prop}

\begin{proof}
The inverse functor sends a map $X\rightarrow Y$ to $\fib(X\rightarrow Y) \rightarrow X$.
\end{proof}

The following corollary is immediate from Proposition \ref{push_out_pointwise_eq}:

\begin{cor}
There is an equivalence of $\infty$-categories:

$$
\xymatrix{\cofib: \CAlg_{R}^{\Delta^{1}_{\vee}} \ar[r]^-{\sim} & \CAlg_{R}^{\Delta^{1}} .}
$$
\end{cor}

Let us now clarify the meaning of the $\infty$-category $\CAlg_{R}^{\{0,1\},\vee}$. As a preparation for the next construction, observe that given an object of $(A^{1},A^{0}) \in \CAlg^{\{0,1\},\vee}_{R,\heartsuit}$, we can define a Smith ideal by the formula $(A^{1} \rightarrow A^{0} \oplus A^{1})$. On the other hand, any Smith ideal $(I \rightarrow A)$ gives rise to a pair $(A,I)\in \CAlg^{\{0,1\},\vee}_{R,\heartsuit}$ where $I$ is considered merely as a non-unital $A$-algebra. Below we will construct an adjunction of symmetric monoidal $\infty$-categories which gives rise to these constructions on the $\infty$-categories of commutative algebra objects.

\begin{construction}\label{ideals_nonunitals}
The tautological inclusion functor $\{0,1\} \rightarrow \Delta^{1}$ is strictly symmetric monoidal with respect to the symmetric monoidal structure described in Construction \ref{Day_arrows}, and consequently induces an adjunction

$$
\xymatrix{  \Mod^{\{0,1\},\vee}_{R}  \ar@/^1.1pc/[rrr]^-{\aug    } &&& \ar@/^1.1pc/[lll]^-{(\Fil \rightarrow \Gr)}  \Mod^{\Delta^{1}_{\vee}}_{R} ,}
$$
where the left adjoint $\aug$ is strictly symmetric monoidal, and the right adjoint $(\Fil \rightarrow \Gr)$ is lax symmetric monoidal.

Unwinding the definition, the functor $(\Fil \rightarrow \Gr)$ sends an arrow $(Y \rightarrow X)$ to the pair $(Y, X)$. The left adjoint $\aug$ sends a pair $(X^{1}, X^{0})$ to the arrow $(X^{1} \rightarrow X^{1} \oplus X^{0})$. 
\end{construction}

\begin{rem}\label{i_!_monadic}
The functor $\aug$ preserves all colimits being a left adjoint. It also preserves all limits, as follows from the formula. Moreover, it is conservative: if a map $(X^{1},X^{0}) \rightarrow (Y^{1},Y^{0})$ induces an equivalence $(X^{1} \rightarrow X^{1}\oplus X^{0}) \simeq (Y^{1} \rightarrow Y^{1}\oplus Y^{0})$, then we immediately get that $X^{1} \simeq Y^{1}$ is an equivalence, and taking cofibers that $X^{0} \simeq Y^{0}$ too. It follows that $\aug$ is monadic and comonadic. 
\end{rem}

\begin{rem}\label{adjunction_refines}

It follows from proposition above that the adjunction of Construction \ref{ideals_nonunitals} refines to an adjunction

$$
\xymatrix{  \CAlg^{\{0,1\},\vee}_{R}  \ar@/^1.1pc/[rrr]^-{\aug} &&& \ar@/^1.1pc/[lll]^-{(\Fil \rightarrow \Gr)}  \CAlg^{\Delta^{1}_{\vee}}_{R} .}
$$

Based on Remark \ref{i_!_monadic}, the $\infty$-category $ \CAlg^{\{0,1\},\vee}_{R}$ should have a description as an $\infty$-category of objects of $ \CAlg^{\Delta^{1}_{\vee}}_{R} $ endowed with some additional data. Such a description will be given below.

\end{rem}

\begin{construction}
Let $\Mod_{R}^{\{0,1\}}$ be the $\infty$-category of functors $\Fun(\{0,1\},\Mod_{R})$ endowed with the \emph{pointwise} symmetric monoidal structure, i.e. the tensor product of two objects $(X^{0},X^{1})$ and $(Y^{0},Y^{1})$ is given by the formula

$$
(X^{0},X^{1}) \otimes (Y^{0},Y^{1}) \simeq (X^{0}\otimes Y^{0}, X^{1} \otimes Y^{1}),
$$
and the unit object is the pair $(R,R)\in \Mod_{R}^{\{0,1\}}$.

The symmetric monoidal $\infty$-category $\Mod_{R}^{\{0,1\}}$ carries a symmetric monoidal equivalence $\epsilon: \Mod_{R}^{\{0,1\}} \simeq \Mod_{R}^{\{0,1\}}$ defined by the formula $\epsilon(X^{1},X^{0}) :=(X^{0},X^{1})$. We define a stable $\infty$-category of \textbf{augmented arrows} as the pull-back

\begin{equation}\label{split_arrows}
\xymatrix{  \Mod_{R}^{\Delta^{1}, \aug}     \ar[d] \ar[rr]&& \Mod_{R}^{\Delta^{1}} \ar[d]^-{(\Fil \rightarrow \Gr)} \\
\Mod_{R}^{\Delta^{1}} \ar[rr]_-{\epsilon \circ (\Fil \rightarrow \Gr)}&& \Mod_{R}^{\{0,1\}}.  }
\end{equation}

Objects of $\Mod_{R}^{\Delta^{1}, \aug}$ can be thought as arrows $p: Y \rightarrow X$ endowed with a section $e: X \rightarrow Y$ such that $p\circ e \simeq \Id_{Y}$. We can endow $\Mod_{R}^{\Delta^{1},\aug}$ with \emph{pointwise} symmetric monoidal structure, so that the diagram \ref{split_arrows} is a pull-back diagram in $\CAlg(\cat)_{\lax}$. There exists a symmetric monoidal functor 

\begin{equation}\label{functor}
\xymatrix{ \Mod_{R}^{\{0,1\},\vee} \ar[r]& \Mod_{R}^{\Delta^{1},\aug} ,\\
 \;\;\;\;\;(X^{1},X^{0}) \longmapsto ( X^{0}\oplus X^{1}   \ar@<0.5ex>[r]^-{p} & \ar@<0.5ex>[l]^-{e} X^{0} ) }
\end{equation} 
where the arrow $p:(X^{1} \oplus X^{0} \rightarrow X^{0})$ is the natural projection, endowed with the $0$-section $e: X^{0} \hookrightarrow X^{1}\oplus X^{0}$.
\end{construction}

\begin{prop}
The functor \ref{functor} induces a symmetric monoidal equivalence

$$\xymatrix{   \Mod_{R}^{\{0,1\},\vee} \ar[r]^-{\sim}& \Mod_{R}^{\Delta^{1}, \aug}.    } $$
\end{prop}

\begin{proof}
The inverse functor sends a split arrow $Z \rightarrow X$ to the pair $(Y,X)$, where $Y:=\fib(Z\rightarrow X)$. 
\end{proof}

\begin{cor}\label{0,1_nonunitals}
There is an equivalence of $\infty$-categories 

$$
\CAlg_{R}^{\{0,1\},\vee} \xymatrix{\ar[r]^-{\sim}&} \CAlg^{\Delta^{1}_{\vee},\aug}_{R} := \CAlg(\Mod^{\Delta^{1}_{\vee},\aug}_{R})
$$

\end{cor}

\begin{rem}
Given the description of $\Mod_{R}^{\Delta^{1},\aug}$ as a pull-back diagram $\ref{split_arrows}$ in $\CAlg(\cat)_{\lax}$, Corollary \ref{0,1_nonunitals} together with Proposition \ref{limits_in_Cat_lax} imply that the $\infty$-category $\CAlg^{\{0,1\},\vee}_{R}:= \CAlg(\Mod_{R}^{\{0,1\},\vee})$ is equivalent to the $\infty$-category of arrows $p: B \rightarrow A$ in $\CAlg_{R}$ endowed with a section $i: A \rightarrow B$, or equivalently, pairs $(A,A^{+})$ where $A\in \CAlg_{R}$ and $A^{+}\in \CAlg^{\nonu}_{A}$.
\end{rem}

We will now move to the study of Smith ideals in the derived world.

\begin{construction}\label{standard_arrows}

We define a derived algebra object $\Mod_{R}^{\Delta^{1}}$ as an arrow $A \rightarrow B$ in the $\infty$-category $\DAlg_{R}$. The resulting $\infty$-category $\Fun(\Delta^{1}, \DAlg_{R})$ is monadic over $\Mod_{R}^{\Delta^{1}}$, and when $R$ is discrete, we can describe the monad as the right-left extension of the monad of free commutative algebra defined on the subcategory $(\Mod_{R,0})^{\Delta^{1}} \subset \Mod_{R}^{\Delta^{1}}$ using the standard (pointwise) $t$-structure on $\Mod_{R}^{\Delta^{1}}$. We denote this monad as $(\LSym_{R})^{\Delta^{1}}: \Mod_{R}^{\Delta^{1}} \rightarrow \Mod_{R}^{\Delta^{1}}$. It is uniquely determined by its restriction to the full subcategory $\Ind(\Mod_{R,0}^{\Delta^{1}}) \subset \Mod_{R}^{\Delta^{1}}$ consisting of pairs $(f: X \rightarrow Z)$ with both $X$ and $Z$ being flat $R$-modules.

\end{construction}

Below we will give another construction of the $\infty$-category of derived algebras in $\Mod_{R}^{\Delta^{1}}$, which is based on a different $t$-structure on $\Mod_{R}^{\Delta^{1}}$ and another choice of a subcategory $\Mod_{R,0}^{\Delta^{1}} \subset \Mod_{R}^{\Delta^{1}}$.

\begin{construction}
Let $R$ be a discrete ring. The $\infty$-category $\Mod_{R}^{\Delta^{1}}$ has a \textbf{geometric $t$-structure}, whose connective part $\Mod_{R, \geq 0}^{\Delta^{1}}$ consists of arrows $(f: X\rightarrow Z)$ with both $X$ and $Z$ being connective, and the map $f$ being surjective on $\pi_{0}$. Equivalently, an arrow $(f: X\rightarrow Z)$ is in $\Mod_{R, \geq 0}^{\Delta^{1}} $ if $X$ and $\fib(f)$ are connective. Consequently, an arrow $(f: X\rightarrow Z)$ is coconnective in geometric $t$-structure if $X$ and $\fib(f)$ are coconnective. Below we will list the main observations regarding this $t$-structure.

\begin{itemize}

\item The geometric $t$-structure is compatible with the pointwise symmetric monoidal structure, and is left and right complete. 

\item The geometric $t$-structure is obtained by transferring the standard (pointwise) $t$-structure on $\Mod_{R}^{\Delta^{1}}$ by means of the self-equivalence $\cofib: \Mod_{R}^{\Delta^{1}} \simeq \Mod_{R}^{\Delta^{1}}$. 

\item The heart $\Mod_{R,\heartsuit}^{\Delta^{1}}$ consists of arrows $(f: X\twoheadrightarrow Z)$ for which $X$ and $\fib(f)$ are discrete $R$-modules. Note that $Z$ itself might not be discrete, as long as $f: X \rightarrow Z$ is surjective on $\pi_{0}$. 

\item There is an additive symmetric monoidal subcategory $\Mod_{R, \heartsuit }^{\Delta^{1}, \sur}  \subset \Mod_{R,\heartsuit}^{\Delta^{1}}$ consisting of surjective maps $(f: X \twoheadrightarrow Z)$ where both $X$ and $Z$ are discrete. 

\item Let $\Mod_{R,0}^{\Delta^{1}} \subset \Mod_{R, \geq 0}^{\Delta^{1}}$ be the subcategory spanned by direct sums of objects $R \rightarrow 0$ and $\Id: R \rightarrow R$. Then $\Mod_{R,0}^{\Delta^{1}} $ is an additive subcategory of $\Mod_{R, \heartsuit}^{\Delta^{1}} $, and it is closed under the action of $\SSeq^{\gen}_{0}$ and $\SSeq^{\gen,\pd}_{0}$. Moreover, the objects $R \rightarrow 0$ and $\Id: R \rightarrow R$ are compact and projective generators of $\Mod_{R,\geq 0}$. Consequently, using Construction \ref{derived_alg_contexts}, we obtain an action of $\SSeq^{\gen}$ on $\Mod_{R}^{\Delta^{1}}$. In particular, we get an $\LSym$ monad $\LSym_{R}^{\Delta^{1}}: \Mod^{\Delta^{1}}_{R} \rightarrow \Mod_{R}^{\Delta^{1}}$. We denote the $\infty$-category of $\LSym^{\Delta^{1}}_{R}$-algebras as $\DAlg_{R}^{\Delta^{1}}$ and refer to its objects \textbf{derived ($R$-linear) arrows}. A justification for this terminology will be given below when we show that this is indeed the same $\infty$-category as the one defined in Construction \ref{standard_arrows}.

\item The additive subcategory $\Ind(\Mod_{R,0}^{\Delta^{1}}) \subset \Mod_{R}^{\Delta^{1}}$ is spanned by objects of the form $(f: X \twoheadrightarrow Z)$ where $X$ and $Z$ are flat $R$-modules. The monad $\LSym^{\Delta^{1}}_{R}$ is uniquely determined by its restriction to $\Ind(\Mod_{R,0}^{\Delta^{1}}) $.

\item There exists a tautological forgetful functor $\DAlg_{R}^{\Delta^{1}} \rightarrow \Fun(\Delta^{1},\DAlg_{R})$. To construct it, recall that in Construction \ref{standard_arrows} we let $(\Mod_{R,0})^{\Delta^{1}}$ to be the category of all arrows of finitely generated free $R$-modules; there is an inclusion $\Mod_{R,0}^{\Delta^{1}} \subset (\Mod_{R,0})^{\Delta^{1}}$, which induces a colimit-preserving fully faithful functor $\Ind(\Mod_{R,0}^{\Delta^{1}}) \hookrightarrow \Ind((\Mod_{R,0})^{\Delta^{1}})$ making commutative the following diagram:

$$
\xymatrix{   \Ind(\Mod_{R,0}^{\Delta^{1}}) \ar[d]_-{\LSym_{R}^{\Delta^{1}} } \ar[r] & \Ind((\Mod_{R,0})^{\Delta^{1}}) \ar[d]^-{(\LSym_{R})^{\Delta^{1} }} \\
  \Ind(\Mod_{R,0}^{\Delta^{1}})  \ar[r] & \Ind((\Mod_{R,0})^{\Delta^{1}})  .   }
$$

Consequently, by right-left extension we obtain a map of monads $\LSym^{\Delta^{1}}_{R} \rightarrow (\LSym_{R})^{\Delta^{1}}$, and a functor $\DAlg^{\Delta^{1}}_{R} \rightarrow \Fun(\Delta^{1},\DAlg_{R})$.

\end{itemize}
\end{construction}

\begin{prop}\label{Z[[x]]}
The functor $\DAlg_{R}^{\Delta^{1}} \rightarrow (\DAlg_{R})^{\Delta^{1}}$ is an equivalence.
\end{prop}

\begin{proof}
The forgetful functor is a functor of monadic $\infty$-categories over $\Mod_{R}^{\Delta^{1}}$. Both monads preserve sifted colimits and are filtered by subfunctors which preserve finite totalizations, the monad $(\LSym_{R})^{\Delta^{1}}$ is determined by its values on the objects $0 \rightarrow R$ and $R\rightarrow 0$, and we have an equivalence $\LSym^{\Delta^{1}}_{R}(R \rightarrow 0) \simeq (\LSym_{R})^{\Delta^{1}}(R \rightarrow 0) $ as $R \rightarrow 0$ belongs to the connective part of the geometric $t$-structure. Therefore, it only remains to show that the natural map \begin{equation}\label{map_of_LSyms}\LSym^{\Delta^{1}}(0 \rightarrow R) \xymatrix{\ar[r]&} (\LSym_{R})^{\Delta^{1}}(0 \rightarrow R) \end{equation} is an equivalence.

The object $(0\rightarrow R) \in \Mod_{R}^{\Delta^{1}}$ is connective in the standard $t$-structure, and the free $\LSym^{\Delta^{1}}_{R}$-algebra on it is the unit map $(R \rightarrow R[x]). $ On the other hand, $(0 \rightarrow R)$ is not connective in the geometric $t$-structure. To compute $\LSym^{\Delta^{1}}_{R}(0 \rightarrow R)$, we can present $(0 \rightarrow R)$ as a totalization \begin{equation}\label{tot_formula}(0 \rightarrow R)\simeq (0 \rightarrow R[1])[-1]  \simeq \underset{n \in \Delta}{\Tot}( 0 \rightarrow R[1]^{\oplus^{n}}      ),\end{equation}  and use the colimit formula of Construction \ref{LSym_monad} \begin{equation}\label{LSym_filtration}\LSym^{\Delta^{1}}_{R} := \underset{i}{\colim}\;\LSym^{\Delta^{1}, \leq i}_{R},\end{equation}
where each filtered layer $\LSym^{\Delta^{1}, \leq i}_{R} $ commutes with finite totalizations being an $i$-excisive functor.

Since $(0 \rightarrow R[1]^{\oplus^{n}})$ is connective in geometric $t$-structure, we have by definition of the $\LSym_{R}^{\Delta^{1}}$ functor: $$\LSym_{R}^{\Delta^{1}} (0 \rightarrow R[1]^{\oplus^{n}}) \simeq (R \rightarrow \LSym_{R}( R[1]^{\oplus^{n}}) . $$ Consequently, combining formulas \ref{tot_formula} and \ref{LSym_filtration}, we obtain $$\LSym^{\Delta^{1}}_{R} (0 \rightarrow R) \simeq  \underset{i}{\colim}\; \underset{n \in \Delta}{\Tot}\Bigl(  R \rightarrow  \LSym^{\leq i}_{R}(R^{\oplus^{n}}[1])  \Bigr) .$$
As the map \ref{map_of_LSyms} is compatible with filtrations, and the same totalization formula works for the standard $(\LSym_{R})^{\Delta^{1},\leq n}$-functors, we conclude that it is an equivalence.
\end{proof}

\begin{ex}
Let $R=\mathbb{Z}$. We compute in $\DAlg_{\mathbb{Z}}^{\Delta^{1}}$: $$\LSym^{\Delta^{1}}_{\mathbb{Z}    }( \Id:  \mathbb{Z} \rightarrow \mathbb{Z}  )\simeq (\Id: \mathbb{Z}[x] \rightarrow \mathbb{Z}[x]);$$ $$\LSym^{\Delta^{1}}_{\mathbb{Z}} (\mathbb{Z} \rightarrow 0 ) \simeq (\mathbb{Z}[x] \rightarrow \mathbb{Z}, x \longmapsto 0).$$ The corresponding Smith ideals in the polynomial ring on one generator are $(0) \subset \mathbb{Z}[x]$ and $(x) \subset \mathbb{Z}[x]$ respectively. 
\end{ex}

\begin{rem}\label{R_nondiscrete}
When $R$ is a discrete commutative ring, the initial object in the $\infty$-category $\DAlg_{R}^{\Delta^{1}}$ is the arrow $(\Id: R \rightarrow R)$. Therefore, when $R$ is a not necessarily discrete derived algebra, we can define the $\infty$-category of geometric $R$-linear arrows as the under-category $$\DAlg^{\Delta^{1}}_{R}:=(\DAlg_{\mathbb{Z}}^{\Delta^{1}})_{R/} ,$$ where $R$ is considered as a tautological arrow $(\Id: R \rightarrow R)$ in $\DAlg_{\mathbb{Z}}^{\Delta^{1}}$.
\end{rem}

\begin{defn}\label{derived_ideals}
Let $R$ be a connective derived ring, and $\DAlg_{R}^{\Delta^{1}}$ be the $\infty$-category of derived $R$-linear arrows, and $\DAlg_{R}^{\Delta^{1},\aug}$ be the $\infty$-category of objects of $(f: A \rightarrow B) \in \DAlg^{\Delta^{1}}_{R}$ endowed with a splitting $(s: B \rightarrow A)$ in $\DAlg_{R}$.  

\begin{enumerate}

\item We define the $\infty$-category of \textbf{derived Smith ideals} as the pull-back of $\infty$-categories along the bottom horizontal equivalence:

$$
\xymatrix{   \DAlg_{R}^{\Delta^{1}_{\vee}} \ar[d] \ar[r]^-{\cofib}_-{\sim} & \DAlg_{R}^{\Delta^{1}}    \ar[d]\\
\Mod^{\Delta^{1}_{\vee}}_{R} \ar[r]_-{\cofib}^-{\sim} & \Mod_{R}^{\Delta^{1}}.   }
$$
By construction, the forgetful functor $\DAlg_{R}^{\Delta^{1}_{\vee}} \rightarrow \Mod^{\Delta^{1}_{\vee}}_{R}$ is monadic. We denote $$\LSym_{R}^{\Delta^{1}_{\vee}}: \Mod^{\Delta^{1}_{\vee}}_{R} \rightarrow\DAlg_{R}^{\Delta^{1}_{\vee}}$$ the left adjoint of the forgetful functor. 

\item We also define an $\infty$-category $\DAlg_{R}^{\{0,1\},\vee}$, monadic over $\Mod_{R}^{\{0,1\},\vee}$, by means of the pull-back diagram

$$
\xymatrix{   \DAlg_{R}^{\{0,1\},\vee}\ar[d] \ar[r]^-{\sim}& \DAlg_{R}^{\Delta^{1}, \aug} \ar[d]\\
\CAlg^{\{0,1\},\vee}_{R} \ar[r]_-{\sim}& \CAlg^{\Delta^{1}, \aug}_{R}   ,}
$$
where the lower horizontal equivalence is the equivalence from Proposition \ref{0,1_nonunitals}, and let $$\LSym_{R}^{\{0,1\},\vee}: \Mod_{R}^{\{0,1\},\vee} \rightarrow \DAlg^{\{0,1\},\vee}_{R} $$ be the free algebra functor.
\end{enumerate}

\end{defn}

\begin{rem}\label{from_ideals_to_nonunitals}
There is an adjunction 

$$
\xymatrix{  \DAlg^{\{0,1\},\vee}_{R}  \ar@/^1.1pc/[rrr]^-{\aug} &&& \ar@/^1.1pc/[lll]^-{(\Fil \rightarrow \Gr)}  \DAlg^{\Delta^{1}_{\vee}}_{R} }
$$
lifting the adjunction noted in Remark \ref{adjunction_refines}.
\end{rem}

\begin{rem}\label{{0,1}_as_nonunitals}
Given a map $B\rightarrow A$ endowed with a section $A \rightarrow B$ is the same as endowing $B$ with the structure of an augmented $A$-algebra. We have $B \simeq A\oplus A^{+}$ where $A^{+}:=\fib(B \rightarrow A)$, and $A^{+}$ gets the structure of a non-unital $A$-algebra. Therefore, we can think of the $\infty$-category $\DAlg^{\{0,1\},\vee}_{R}$ as the $\infty$-category of pairs $(A,A^{+})$ where $A\in \DAlg_{R}$ and $A^{+}$ is a non-unital derived $A$-algebra.
\end{rem}

For a connective derived ring $R$, we will construct a $t$-exact lax symmetric monoidal functor $\Fil^{\geq 0}\Mod_{R} \rightarrow \Mod_{R}^{\Delta^{1}_{\vee}}$ and extend it to a functor $\Fil^{\geq 0}\DAlg_{R} \rightarrow \DAlg^{\Delta^{1}_{\vee}}_{R}. $

\begin{construction}\label{ev[0,1]_for_derived_rings}
Let $ \Delta^{1}\hookrightarrow \mathbb{Z}_{\geq 0}$ be the inclusion of $0$ and $1$. Notice that it is a lax symmetric monoidal functor with respect to the symmetric monoidal structure on $\Delta^{1}$ described in Construction \ref{Day_arrows}. It induces a limit preserving lax symmetric monoidal functor on presheaf categories $ \ev^{[0,1]}: \Fil^{\geq 0}\Mod_{R} \rightarrow \Mod_{R}^{\Delta^{1}_{\vee}}$ which sends an object $F^{\geq \star}$ to the arrow $F^{\geq 1} \rightarrow F^{\geq 0}$. In fact, the lax symmetric monoidal structure is strict. Indeed, composing the functor $\ev^{[0,1]}$ with the symmetric monoidal equivalence $\cofib: \Mod_{R}^{\Delta^{1}_{\vee}} \simeq \Mod_{R}^{\Delta^{1}}$, we get the functor sending a filtered object $F^{\geq \star}$ to the arrow $F^{\geq 0} \rightarrow \gr_{F}^{0}$, which is manifestly symmetric monoidal. Therefore, so is the functor $\ev^{[0,1]}$.

\begin{itemize}
\item Assume $R$ is a discrete ring. The functor $\ev^{[0,1]}$ is continuous and sends $\Fil^{\geq 0}\Mod_{R,0} \subset \Fil^{\geq 0}\Mod_{R}$ to $\Mod^{\Delta^{1}_{\vee}}_{R,0} \subset \Mod_{R}^{\Delta^{1}_{\vee}}$. Consequently, we get a commutative square

$$
\xymatrix{  \Ind(\Fil\Mod_{R,0}) \ar[d]_-{\LSym_{R}^{\geq \star}} \ar[r]^-{\ev^{[0,1]}}& \Ind(\Mod^{\Delta^{1}_{\vee}}_{R,0}) \ar[d]^-{\LSym^{\Delta^{1}_{\vee}}_{R}}\\
\Ind(\Fil\Mod_{R,0})  \ar[r]_-{\ev^{[0,1]}}& \Ind(\Mod^{\Delta^{1}_{\vee}}_{R,0}) . }
$$             
The functor $\ev^{[0,1]}$ is $t$-exact with respect to the geometric $t$-structure $\Mod_{R}^{\Delta^{1}_{\vee}}$. Consequently, using Observation \ref{right_left_extended_monads} and the fact that both monads involved are right-left extended, we obtain a functor $$\ev^{[0,1]}: \Fil^{\geq 0}\DAlg_{R} \xymatrix{\ar[r]&} \DAlg^{\Delta^{1}_{\vee}}_{R}. $$
 
 \item If $R$ is a not necessarily discrete connective derived ring, we consider $R$ as a filtered ring with $R$ inserted in degree $0$ and zero in all other degrees, and notice that $$\Fil^{\geq 0}\DAlg_{R} \simeq (\Fil^{\geq 0}\DAlg_{\mathbb{Z}})_{R/} , \;\;\;\;\DAlg_{R}^{\Delta^{1}_{\vee}}\simeq (\DAlg^{\Delta^{1}_{\vee}}_{\mathbb{Z}})_{R/}.$$

In this case, the functor $\ev^{[0,1]}: \Fil^{\geq 0}\DAlg_{R} \rightarrow \DAlg^{\Delta^{1}_{\vee}}_{R}$ is obtained from the previous part of the construction as an induced functor on objects under $R$. 
 \end{itemize}
                                                                                                                                                                                                                                                                                                                                                                                                                                                                                                                                                                                                                                                                                                                                                                                                                                                                                                                                                                                               
\end{construction}

\begin{rem}
Using Definition \ref{derived_ideals}, let us describe limits and sifted colimits in the $\infty$-category $\DAlg_{R}^{\Delta^{1}_{\vee}}$. In the $\infty$-category $\DAlg_{R}^{\Delta^{1}} $ limits and colimits are taken pointwise. Let $S \rightarrow \DAlg_{R}^{\Delta^{1}_{\vee}}, i \longmapsto (A_{i},I_{i})$ be a diagram in $\DAlg_{R}^{\Delta^{1}_{\vee}}$. Applying the functor $\cofib: \DAlg_{R}^{\Delta^{1}_{\vee}} \rightarrow \DAlg_{R}^{\Delta^{1}}$, we get a diagram $\{A_{i}\rightarrow A_{i}/I_{i}\}$ in $\DAlg_{R}^{\Delta^{1}}$. Let $(A\rightarrow A/I):= \lim_{i} (A_{i} \rightarrow A_{i}/I_{i})$. Applying the functor $\fib$, we get that the colimit of the diagram $\{(A_{i},I_{i})\}$ is given by $(I,A)$, where $A=\lim_{i} A_{i}$ and $$I=\fib(  \lim_{i} A   \rightarrow \lim_{i} A_{i}/I_{i})  ) \simeq \lim_{i} \fib( A_{i} \rightarrow A_{i}/I_{i}) \simeq \lim_{i} I_{i},$$ where each term $I_{i}$ is endowed with the $A$-module structure obtained by restricting via the natural map $A\rightarrow A_{i}$.

Similarly, let $(A',I') = \colim_{i} (A_{i},I_{i})$. Then $A'=\colim_{i} A_{i}$, and $I'$ is computed as $$I'=\fib( \colim_{i} A_{i} \rightarrow \colim_{i} A_{i}/I_{i}).$$
Assume $S$ is a sifted diagram. Then as the forgetful functor $\DAlg_{R}^{\Delta^{1}_{\vee}} \rightarrow \Mod_{R}^{\Delta^{1}} $ commutes with sifted colimits, and in the stable $\infty$-category $\Mod_{R}$ finite limits commute with sifted colimits, we obtain that $$I'= \colim_{i} (I_{i} \otimes_{A_{i}} A') $$ as an $A'$-module.
\end{rem}

\begin{rem}
Consider the category $\CAlg_{\mathbb{Z},\heartsuit}^{\Delta^{1}_{\vee},\injec}$ of pairs $(I \rightarrow A)$ where $A$ is a discrete commutative ring, and $I\hookrightarrow A$ is a usual ideal. One can show that $\CAlg_{\mathbb{Z},\heartsuit}^{\Delta^{1}_{\vee},\injec}$ has all colimits, but the forgetful functor $\CAlg_{\mathbb{Z},\heartsuit}^{\Delta^{1}_{\vee},\injec} \rightarrow \CAlg_{\mathbb{Z}, \heartsuit}, (I \rightarrow A) \longmapsto A$ does not commute with colimits. An explicit example is worked out in \cite[Example 3.33]{Mao21}. The computation of Zhouhang Mao is motivated by \cite[Remark 23.3.5]{Stacks}, where the same example is brought to demonstrate the discrepancy between coproducts in the category of divided power algebras and tensor products of underlying algebras. In the setting of derived Smith ideals the same problem does not arise, even in the divided power case.
\end{rem}

We will finish this subsection with the following construction which will be useful in the sequel.

\begin{construction}\label{insertion_adjunction}
Let $\ins^{0}: \DAlg_{R} \rightarrow \DAlg_{R}^{\Delta^{1}_{\vee}}$ be the functor sending $A$ the trivial Smith ideal $0 \rightarrow A$ of $A$. There is an adjunction

$$
\xymatrix{ \DAlg_{R} \ar[rrr]^-{\ins^{0} } &&& \ar@/_1.1pc/[lll]_{\gr^{0} }    \ar@/^1.1pc/[lll]^{\ev^{0}}  \DAlg_{R}^{\Delta^{1}_{\vee}},}
$$
where $\gr^{0}(I \rightarrow A) = A/I$ and $\ev^{0} (I \rightarrow A) = A$.

\end{construction}

\subsection{Divided power algebras.}

Based on Definition \ref{derived_ideals}, given a map $A\rightarrow B$, we can construct a unique derived Smith ideal $I \rightarrow A$ in $A$, and formulate various properties or additional structures on the map $A\rightarrow B$ in terms of the Smith ideal $I$. In this text we will define \textbf{divided power structure} on a map of derived commutative rings as the data of a non-unutal derived divided power structure on the underlying non-unital derived commutative algebra of $I$. Below we shall explain our reasons for introducing divided power algebras.

There are two ways to think about De Rham cohomology theory, a more well-known one is in terms of differential graded commutative algebras, and another one is in terms of divided power algebras. Assume $A$ is a discrete commutative algebra. Then there is a classical universal property of the De Rham complex of $A$: for any map $A \rightarrow B^{0}$ into the $0$-th graded piece of a differential graded commutative algebra $B^{\bullet}$, there is a unique map of differential graded commutative algebras $\Omega^{\bullet}_{A} \rightarrow B^{\bullet}$. Moreover, we have an equivalence of graded algebras $\Omega^{\bullet}_{A} \simeq \Lambda^{\bullet}_{A} (\Omega^{1}_{A})$ with the differential extended from the universal derivation $d: A \rightarrow \Omega_{A}^{1}$ via the graded Leibnitz rule. A.Raksit \cite{R} generalized this universal property to the derived setting. In this text we will develop another approach in terms of derived \emph{divided power} algebras, rather than differential graded commutative algebras. It can be briefly summarized as follows. One defines an $\infty$-category $\Fil^{\geq 0} \DAlg^{\pd}_{R}$ of \textbf{non-negatively filtered derived divided power algebras}, and considers the functor $\gr^{0}: \Fil^{\geq 0} \DAlg^{\pd}_{R} \rightarrow \DAlg_{R}$. The derived De Rham complex $\LOmega^{\geq \star}_{-/R}$ is defined as the left adjoint of this functor. The main reason why this recovers the classical De Rham complex in the case of a smooth map $R\rightarrow A$, is essentially due to the décalage isomorphism

$$
\LGamma^{n}(M[-1]) \simeq \LLambda^{n}(M)[-n]
$$
demonstrated in Proposition \ref{décalage_for_algebras}. Namely, the associated graded object of $\LOmega^{\geq \star}_{A/R}$ is given by $$\LOmega^{\bullet}_{A/R} \simeq \LGamma^{\bullet}_{A}(\LL_{A}[-1]) \simeq \LLambda^{\bullet}_{A} (\LL_{A})[-n],$$ where $\LL_{A/R}$ is the cotangent complex of the map $R \rightarrow A$, and in the case when $\LL_{A/R}$ is a projective $A$-module, the divided power structure on $\LOmega_{A/R}$ is equivalent to the data of a differential graded commutative algebra. The main reason why we prefer to use think about De Rham cohomology in terms of derived divided power algebras, is because various constructions are more natural and easier to define in this setting. This particularly applies to stacky approaches to derived De Rham cohomology, relation with crystalline and prismatic cohomology. We will show that the relative derived De Rham cohomology $\LOmega_{A/B}$ of a map $A\rightarrow B$ can be understood as the \textbf{derived pd filtration} of the \textbf{derived pd envelope} of $A$ with respect to the Smith ideal $I=\fib(A\rightarrow B)$. In the case when the map $A\rightarrow B$ is a surjective map of discrete commutative algebras whose ideal $I$ is generated by a regular sequence, this recovers the classical divided power envelope. On the other hand, when the map $A\rightarrow B$ is a smooth map of discrete commutative rings, the negative $t$-structure together with décalage isomorphism enables us to think about the relative De Rham complex $\Omega_{A/B}$ in terms of the exterior algebra on the module of relative Kahler differentials $\Omega^{1}_{B/A}$.

We will finish this introduction with a historical remark. Divided power algebras featured prominently in algebraic topology, i.e. in the context of homology of loop spaces, and were brought to world of algebraic geometry by A.Grothendieck \cite{Gr} in the context of crystalline cohomology. The basic idea of divided power algebras is that outside of characteristic $0$, we can not form $\frac{x^{n}}{n!}$ of an element in an arbitrary algebra $A$, and a divided power structure on $A$ is an additional structure which axiomatizes the existence of such operations on some ideal. Initially, the definition of a divided power algebra was given in terms of the data of operations $\gamma_{n}(x)$ for all $n$ satisfying a number of conditions reflecting the idea that $\gamma_{n}(x)$ should behave like $\frac{x^{n}}{n!}$. B.Fresse in the text \cite{Fr} explained an operadic approach to divided power algebras. The manuscript of Brantner, Campos and Nuiten \cite{BCN21} developed the notion of derived pd operads in a very general setting, which we follow in this text.

\subsubsection{Non-unital divided power algebras.}

Recall the following classical definition.

\begin{defn}\label{pd_algebras_via_operations}
A \textbf{(discrete) non-unital divided power ring} is a non-unital commutative ring $A$ endowed with operations $\gamma_{n}: A \rightarrow A$ satisfying the following relations:

\begin{enumerate}
\item $\gamma_{0}(x) = 1, \gamma_{1}(x) = x $ for any $x\in A$;

\item $\gamma_{n}(x+y) = \sum_{i+j=n} \gamma_{i}(x) \gamma_{j}(y)$ for any $x, y \in A$;

\item $\gamma(ax) = a^{n} \gamma(x)$ for any $a, x \in A$;

\item $\gamma_{m}(x) \gamma_{n}(x) = \frac{(m+n)!}{m!n!} \gamma_{m+n}(x)$;

\item $\gamma_{n}(\gamma_{m}(x)) = \frac{(mn)!}{(m!)^{n} n!} \gamma_{mn}(x)$.
\end{enumerate}

We let $\CAlg^{\pd}_{\mathbb{Z},\heartsuit}$ to be the category of discrete \emph{non-unital}\footnote{There is no notion of a "unital" divided power algebra, because the unit never supports divided powers unless we are in characteristic $0$} divided power rings.

\end{defn}

\begin{rem}
It was observed by B.Fresse in \cite{Fr} that this category is monadic over $\Mod_{\mathbb{Z},\heartsuit}$ with the monad given by the formula

$$
\Gamma^{+}(M) = \bigoplus_{n\geq 1} \Gamma^{n}(M) = \bigoplus_{n\geq 1} (M^{\otimes^{n}} )^{\Sigma_{n}}.
$$

There is a \textbf{norm map} $\Sym(M) \rightarrow \Gamma^{+}(M)$ given in every degree by the map $\Sym^{n}(M) \rightarrow \Gamma^{n}(M)$ induced by the $\Sigma_{n}$-norm map $$\Nm_{\Sigma_{n}}: M^{\otimes^{n}} \rightarrow M^{\otimes^{n}}, x_{1}\otimes...\otimes x_{n} \longmapsto \sum_{\sigma \in \Sigma_{n}} x_{\sigma(1)} \otimes ... \otimes x_{\sigma(n)}  $$

One can think of the structure of a non-unital divided power algerba $A$ as a non-unital commutative algebra endowed with a factorization of the multiplication maps $m_{n}: A^{\otimes^{n}} \rightarrow A$ through the norm as follows

$$
\xymatrix{  A^{\otimes^{n}}   \ar[d]_-{\Nm_{\Sigma_{n}}} \ar[r]^-{m_{n}}& A\\
A^{\otimes^{n}} \ar[ru]_-{\gamma_{n}}  . }
$$
for any $n$, and these factorizations satisfy some additional associativity conditions. 

\end{rem}

We call $\Gamma^{+}(M)$ the \textbf{free non-unital divided power algebra} on $M$. There is a subcategory $\CAlg^{\nonu,\pd-\poly}_{\mathbb{Z}} \subset \CAlg^{\nonu,\pd}_{\mathbb{Z},\heartsuit}$ of  compact projective generators consisting of free non-unital divided power algebras on finitely many variables. One can then define an $\infty$-category $$\DAlg^{\nonu,\pd}_{\mathbb{Z},\geq 0}:=\mathcal{P}_{\Sigma}( \CAlg^{\nonu,\pd-\poly}_{\mathbb{Z}} ) $$ which is monadic over $\Mod_{\mathbb{Z},\geq 0}$, and one can show that it is also comonadic over $\DAlg_{\mathbb{Z},\geq 0}$ (see \ref{pd_nonu_lims_colims} for precise argument).  

\begin{ex}\label{free_nonunital_pd}
The free non-untal divided power algebra on one generator $\Gamma^{+}(\mathbb{Z}) \simeq  \mathbb{Z} \langle x \rangle^{+} $ can be described as the smallest non-unital subalgebra of $\mathbb{Q}[x]$ containing all elements of the form $\frac{x^{n}}{n!}. \in \mathbb{Q}[x]$ .
\end{ex}

\begin{rem}\label{pd_nonunital}

Let $(-)^{\sharp}: \DAlg^{\nonu}_{\mathbb{Z},\geq 0} \rightarrow \DAlg^{\nonu,\pd}_{\mathbb{Z},\geq 0}$ be the right adjoint of the forgetful functor. For any connective $A \in \DAlg^{\nonu}_{\mathbb{Z},\geq 0}$, there is an equivalence of spaces

$$
A^{\sharp} \simeq \Map_{\DAlg^{\nonu}_{\mathbb{Z},\geq 0}} (\mathbb{Z} \langle x \rangle^{+} , A).
$$
For a discrete $A$, the latter can be understood as the set of \textbf{divided power sequences} in $A$, i.e. the subset of $\prod_{n=1}^{\infty} A$ consisting of collections of elements $(x_{1},x_{2},...)$ satisfying $$x_{n}x_{m} = \frac{(m+n)!}{m!\:n!} x_{n+m}.$$
\end{rem}

\begin{defn}\label{pd_nonunital_definition}
Let $R$ be a connective derived commutative ring. We define the $\infty$-category of \textbf{(non-unital) derived divided power algebras} $$\DAlg^{\nonu,\pd}_{R} :=\Alg_{\LGamma^{+}_{R}}(\Mod_{R})$$ as the $\infty$-category of algebras over the non-unital derived divided power algebra monad $\LGamma^{+}_{R}$ provided by Construction \ref{pd_derived_action_on_deriveds}.
\end{defn}

\begin{prop}\label{pd_nonu_lims_colims}
Let $R$ be a connective derived commutative ring, and $\DAlg_{R}^{\nonu}$ the $\infty$-category of non-unital derived commutative $R$-algebras. The forgetful functor $\DAlg^{\nonu,\pd}_{R} \rightarrow \DAlg^{\nonu}_{R}$ commutes with limits and colimits.
\end{prop}

\begin{proof}
We argue as in the proof in \cite[ Proposition 4.2.27]{R}. Commutation with limits follows from the fact that both sides are $\infty$-categories of algebras over a monad, and the forgetful functor is restriction along a monad map. Commutation with sifted colimits follows from commutation of the monad $\LGamma^{+}_{R}: \Mod_{R} \rightarrow \Mod_{R}$ with sifted colimits. For commutation with finite coproducts, one reduces to the statement that for two objects $X,Y \in \Mod^{\free,\fg}_{R}$, the natural map $\LGamma^{+}_{R}(X) \otimes \LGamma^{+}_{R}(Y) \rightarrow \LGamma^{+}_{R} (X\oplus Y) $ is an equivalence, which follows from the case $R=\mathbb{Z}$ by base change.
\end{proof}

\begin{defn}
We let $\Env^{\pd}: \DAlg^{\nonu}_{R} \rightarrow \DAlg^{\pd}_{R}$ and $\coEnv^{\pd}: \DAlg^{\nonu}_{R} \rightarrow \DAlg^{\pd}_{R}$ be the left adjoint and the right adjoint of the forgetful functor, respectively. We call them \textbf{divided power (pd) envelope} and \textbf{divided power (pd) coenvelope}. 
\end{defn}

\begin{construction}\label{connective_cover_of_pd}
Let $\DAlg^{\pd}_{R,\geq 0}\subset \DAlg^{\pd}_{R}$ be the full subcategory of objects whose underlying $\Mod_{R}$-object is connective. As the free derived divided power algebra monad $\LGamma^{+}_{R}: \Mod_{R} \rightarrow \Mod_{R}$ preserves connective objects, it follows that the functor $\LGamma^{+}_{R}: \Mod_{R,\geq 0} \rightarrow \DAlg^{\pd}_{R,\geq 0}$, left adjoint of the forgetful functor $\DAlg^{\pd}_{R,\geq 0} \rightarrow \Mod_{R,\geq 0} $, is well-defined, and commutes with the inclusion $\DAlg^{\pd}_{R,\geq 0} \hookrightarrow \DAlg^{\pd}_{R}$. Since the inclusion commutes with all colimits, it follows that there exists a right adjoint \textbf{connective cover} functor $\tau_{\geq 0}: \DAlg^{\pd}_{R} \rightarrow \DAlg^{\pd}_{R,\geq 0}$ which fits into a commutative diagram 

$$
\xymatrix{  \DAlg^{\pd}_{R}\ar[d] \ar[r]^-{\tau_{\geq 0}}& \DAlg^{\pd}_{R,\geq 0}\ar[d]\\
\DAlg^{\nonu}_{R} \ar[r]_{\tau_{\geq 0}}& \DAlg^{\nonu}_{R},   }
$$
where the vertical functors are forgetful functors.

\end{construction}

In what follows we will also need to work with divided power algebras over not necessarily connective derived rings. To explain how to do it, consider how this is done in the context of ordinary (without divided powers) derived rings. First, we have the $\infty$-category $\DAlg_{\mathbb{Z}}$ of all derived rings. Then given a not necessarily connective $R\in \DAlg_{\mathbb{Z}}$, we can consider the $\infty$-category $\DAlg_{R}:=(\DAlg_{\mathbb{Z}})_{R/}$ of under objects. For connective $R$, this is equivalent to the $\infty$-category of algebras over the monad $\LSym_{R}$. Over a not necessarily connective $R$, this $\infty$-category is monadic over $\Mod_{R}$, and we can consider this to be the definition of the monad $\LSym_{R}$ for not necessarily connective $R$. We can work similarly in the setting of non-unital derived divided power algebras.

\begin{defn}\label{pd_over_nonconnective}
Recall that the $\infty$-category $\DAlg_{\mathbb{Z}}^{\nonu}$ of non-unital derived rings is equivalent to the $\infty$-category of augmented derived rings. Composing the forgetful functor $\DAlg_{\mathbb{Z}}^{\nonu,\pd} \rightarrow \DAlg^{\nonu}_{\mathbb{Z}}$ with the equivalence $\DAlg^{\nonu}_{\mathbb{Z}} \simeq \DAlg^{\aug}_{\mathbb{Z}} $, we obtain a forgetful functor $\DAlg_{\mathbb{Z}}^{\pd,\nonu} \rightarrow \DAlg_{\mathbb{Z}}^{\aug}$ sending $A^{+}\in \DAlg_{\mathbb{Z}}^{\pd,\nonu}$ to the pair $(\mathbb{Z}\oplus A^{+}, A^{+})$. Let $R$ be a derived ring, not necessarily connective. Consider it as an object $(R,0)$ of $\DAlg_{\mathbb{Z}}^{\aug}$ with $0$ augmentation ideal. Let $(\DAlg^{\aug}_{\mathbb{Z}})_{R/}$ the $\infty$-category of derived pairs receiving a map from $(R,0)$. Then the $\infty$-category of \textbf{derived divided power non-unital $R$-algebras} can be defined as the $\infty$-category of derived non-unital divided power $\mathbb{Z}$-algebras whose underlying augmented derived ring $(\mathbb{Z}\oplus A^{+}, A^{+})$ receives a map from $(R,0)$:

$$
\DAlg_{R}^{\pd,\nonu}:=\DAlg_{\mathbb{Z}}^{\pd,\nonu}\underset{\DAlg^{\aug}_{\mathbb{Z}}}{\times}(\DAlg^{\aug}_{\mathbb{Z}})_{R/}.
$$

There is a forgetful functor $\DAlg_{R}^{\pd,\nonu} \rightarrow \Mod_{R}$ which is monadic. We denote $\LGamma_{R}^{+}: \Mod_{R} \rightarrow \Mod_{R}$ the free algebra monad.  
\end{defn}

\begin{rem}
If $R$ is connective, then Definition \ref{pd_over_nonconnective} is equivalent to Definition \ref{pd_nonunital_definition}. To see this, we can assume that $R$ is discrete, and then it is sufficient to describe the category of discrete divided power $R$-algebras.  Then any object of the category defined as in Definition \ref{pd_over_nonconnective} gives rise to a divided power $\mathbb{Z}$-algerba $A^{+}$ together with an action of $R$ on $A^{+}$ turning $A^{+}$ commuting with divided power operations in the usual sense. The slogan is that a divided power $R$-algebra is the same as a divided power ring together with a $R$-linear structure on the underlying non-unital ring. This will be often used throughout the text.
\end{rem}

\subsubsection{Divided power Smith ideals.}

A global version of the notion of a non-unital divided power algebra is the notion of a divided power pair, or a divided power Smith ideal. To define this notion, recall that for any Smith ideal $I$ in $A$, the $A$-module $I=\fib(A \rightarrow A/I)$ gets a natural structure of a non-unital derived commutative algebra in $A$-modules. A homotopy coherent version of this construction is provided by the functor $\DAlg_{R}^{\Delta^{1}_{\vee}} \rightarrow \DAlg_{R}^{\{0,1\},\vee}$ induced by the lax symmetric monoidal functor $\Mod_{R}^{\Delta^{1}_{\vee}} \rightarrow \Mod_{R}^{\{0,1\},\vee}$ from Construction \ref{ideals_nonunitals}. As we observed in Remark \ref{{0,1}_as_nonunitals}, the $\infty$-category $\DAlg_{R}^{\{0,1\},\vee}$ is equivalent to the $\infty$-category of augmented derived $R$-algebras, i.e. pairs $(A,A^{+})$ where $A\in \DAlg_{R}$ and $A^{+}\in \DAlg_{A}^{\nonu}$. In Definition \ref{pd_over_nonconnective} we explained how to define non-unital derived divided power algebras over a not necessarily connective base ring. This enables us to give the following definition.

\begin{defn}\label{pd_nonunitals_fibered}
We define an $\infty$-category $\DAlg^{\{0,1\},\vee, \pd}_{R}$ as the $\infty$-category of pairs $(A,A^{+})$ where $A^{+}$ has the structure of a non-unital divided power $A$-algebra in the sense of Definition \ref{pd_over_nonconnective}. There is a non-conservative forgetful functor $\DAlg^{\{0,1\},\vee, \pd}_{R} \rightarrow \DAlg_{R}$ (projection on the first component), and a conservative forgetful functor $\forget_{\pd}: \DAlg_{R}^{\{0,1\},\vee,\pd} \rightarrow \DAlg_{R}^{\{0,1\},\vee}$, forgetting the divided power structure on the second component. The left adjoint of the latter $\Env^{\pd}: \DAlg_{R}^{\{0,1\},\vee} \rightarrow  \DAlg_{R}^{\{0,1\},\vee,\pd} $ sends a pair $(A,A^{+})$ to $(A, \Env^{\pd}(A^{+}))$, where $ \Env^{\pd}(A^{+}$ is the non-unital pd envelope of the non-unital $A$-algebra $A^{+}$. The forgetful functor $\forget_{\pd}$ is monadic, and so is the forgetful functor $\DAlg^{\{0,1\},\vee,\pd}_{R} \rightarrow \Mod_{R}^{\{0,1\},\vee} $. We denote $(\LSym, \LGamma^{+}): \Mod_{R}^{\{0,1\},\vee} \rightarrow \Mod_{R}^{\{0,1\},\vee} $ the corresponding monad, which satisfies the formula $$(\LSym, \LGamma^{+}) (X,Y) \simeq (\LSym_{R}(X), \LGamma^{+}_{\LSym_{R}(X)}(Y \otimes \LSym_{R}(X))) .$$

If $R$ is discrete, this monad is right-left extended from its restriction to the full subcategory $\Ind(\Mod^{\{0,1\},\vee}_{R,0}) $, where $\Mod_{R,0}^{\{0,1\},\vee}$ is the additive subcategory of $\Mod_{R}^{\{0,1\},\vee}$ consisting of objects of the form $(X,Y) $, where both $X$ and $Y$ are finitely generated free $R$-modules.
\end{defn}

\begin{defn}\label{divided_power_pairs}

Let $(I \rightarrow A) \in\DAlg_{R}^{\Delta^{1}_{\vee}}$. A \textbf{divided power (pd) structure} on $I$ is a lift of the non-unital derived commutative $A$-algebra $I$ to a non-unital derived divided power $A$-algebra. In this situation, we call $(I \rightarrow A)$ a \textbf{divided power (pd) pair}, and $I$ a \textbf{divided power Smith ideal}. They form an $\infty$-category $\DAlg_{R}^{\Delta^{1}_{\vee},\pd}$ defined as the pull-back

$$
\xymatrix{\DAlg_{R}^{\Delta^{1}_{\vee},\pd} \ar[d] \ar[r]& \DAlg_{R}^{\Delta^{1}_{\vee}}\ar[d]^-{(\Fil \rightarrow \Gr)}\\
\DAlg_{R}^{\{0,1\},\vee,\pd} \ar[r]_-{\forget_{\pd}}&  \DAlg_{R}^{\{0,1\},\vee},}
$$
where both functors involved in the diagram are forgetful functors.

We also say that an arrow $(A\rightarrow B) \in \DAlg_{R}^{\Delta^{1}}$ has the structure of a \textbf{divided power map} if the fiber $\fib(A\rightarrow B)$ is given the structure of a non-uinital divided power $A$-algebra. We let $\DAlg_{R}^{\Delta^{1},\pd}$ be the $\infty$-category of divided power maps. By construction, there is an equivalence

$$
\xymatrix{  \cofib:  \DAlg_{R}^{\Delta^{1}_{\vee},\pd} \ar[r]^-{\sim} & \DAlg_{R}^{\Delta^{1}, \pd}.  }
$$

\end{defn}

\begin{prop}\label{envelopes_coenvelopes_prop}
The $\infty$-category $\DAlg_{R}^{\Delta^{1}_{\vee},\pd}$ is presentable, and the forgetful fuctor $\DAlg_{R}^{\Delta^{1}_{\vee},\pd} \rightarrow \DAlg_{R}^{\Delta^{1}_{\vee}}$ preserves limits and sifted colimits.
\end{prop}

\begin{proof}
The $\infty$-category $\DAlg_{R}^{\Delta^{1}_{\vee},\pd}$ is presentable being a pull-back of presentable $\infty$-categories along accessible functors. Let $\mathcal{D} \rightarrow\DAlg_{R}^{\Delta^{1}_{\vee},\pd}, (i\in \mathcal{D}) \longmapsto (A_{i}, I_{i}) $ be a diagram. Let $(I \rightarrow A):= \underset{i\in \mathcal{D}}{\lim} (A_{i}, I_{i})$ be the limit taken in $\DAlg_{R}^{\Delta^{1}_{\vee}}$. The Smith ideal $I = \lim  I_{i}$ has the structure of a non-unital divided power $A$-algebra as the forgetful functor $\DAlg^{\pd}_{A} \rightarrow \DAlg_{A}^{\nonu}$ creates limits. Therefore, $(I \rightarrow A) \in \DAlg_{R}^{\Delta^{1}_{\vee},\pd}$. Let us show that the object $(I \rightarrow A)$ satisfies the universal property. Assume $(B,J) \rightarrow (A_{i}, I_{i})$ is a collection of maps into the diagram $\mathcal{D}$. We have a pull-back diagram

 $$\xymatrix{ \Map_{\DAlg_{R}^{\Delta^{1}_{\vee},\pd}} ((B,J) , (I \rightarrow A)) \ar[d]      \ar[r]     & \Map_{\DAlg_{R}^{\Delta^{1}_{\vee}},\inf}( (B,J), (I \rightarrow A))\ar[d] \\
 \Map_{\DAlg_{R}^{\{0,1\},\vee, \pd}  }( (B,J) , (I \rightarrow A)) \ar[r]    &   \Map_{\DAlg_{R}^{\{0,1\},\vee}}( (B,J), (I \rightarrow A) ) ,  }$$
 each term of which is equivalent to the limit $\underset{\mathcal{D}}{\lim} \Map((B,J), (A_{i},I_{i}))$ computed in the corresponding $\infty$-category. Therefore, we have $$\Map_{\DAlg_{R}^{\Delta^{1}_{\vee},\pd}}((B,J), (I \rightarrow A)) \simeq \underset{\mathcal{D}}{\lim}  \Map_{\DAlg_{R}^{\Delta^{1}_{\vee},\pd}} ((B,J) , (A_{i},I_{i})) $$
 as well, and $(I \rightarrow A)$ is the limit of the diagram in $\DAlg_{R}^{\Delta^{1}_{\vee},\pd}$, as desired. The same argument shows that the forgetful functor also commutes with sifted colimits. 

\end{proof}

\begin{defn}
We define the functor of \textbf{derived divided power (pd) envelope} $\Env^{\pd}: \DAlg_{R}^{\Delta^{1}_{\vee}}\rightarrow\DAlg_{R}^{\Delta^{1}_{\vee},\pd}$ as the left adjoint of the forgetful functor $\DAlg_{R}^{\Delta^{1}_{\vee},\pd} \rightarrow \DAlg_{R}^{\Delta^{1}_{\vee}}$. By commutation with sifted colimits and conservativity, it follows that the forgetful functors $\DAlg_{R}^{\Delta^{1}_{\vee}} \rightarrow \DAlg_{R}^{\Delta^{1}_{\vee}}$ and $\DAlg_{R}^{\Delta^{1}_{\vee},\pd} \rightarrow \Mod_{R}^{\Delta^{1}_{\vee}}$ are monadic.
\end{defn}

\begin{prop}
The derived divided power envelope functor $\Env^{\pd}: \DAlg_{R}^{\Delta^{1}_{\vee}}\rightarrow\DAlg_{R}^{\Delta^{1}_{\vee},\pd}$ sends the full subcategory $\DAlg_{R,\geq 0}^{\Delta^{1}_{\vee}} \subset \DAlg_{R}^{\Delta^{1}_{\vee}}$ to $ \DAlg_{R,\geq 0}^{\Delta^{1}_{\vee},\pd}\subset \DAlg_{R}^{\Delta^{1}_{\vee},\pd}$
\end{prop}

\begin{proof}
By passing to right adjoints, the statement is equivalent to the statement that for a divided power pair $(I \rightarrow A)$, the connective cover $(\tau_{\geq 0}A, \tau_{\geq 0}I)$ is again a divided power pair. This comes down to the connective cover $\tau_{\geq 0}I$ being defined as a non-unital divided power $\tau_{\geq 0} A$-algebra. This is the content of Construction \ref{connective_cover_of_pd}.
\end{proof}

We gave a rather abstract definition of the $\infty$-category of divided power Smith ideals, without using any explicit compact projective generators. As a result, it is apriori unclear what free algebras in this theory look like. The next Proposition affirms that nothing "surprising" happens.

\begin{prop}\label{pd_envelope_of_cp}
There are equivalences of divided power Smith ideals:

\begin{itemize}

\item $\Env^{\pd} \LSym_{\mathbb{Z}}^{\Delta^{1}_{\vee}} (\Id:\mathbb{Z} \rightarrow \mathbb{Z}) \simeq  \mathbb{Z}\langle x\rangle^{+} \rightarrow \mathbb{Z}\langle x\rangle ,$

\item  $\Env^{\pd}  \LSym_{\mathbb{Z}}^{\Delta^{1}_{\vee}}  (0\rightarrow \mathbb{Z} ) \simeq (0) \rightarrow \mathbb{Z}[x],$
\end{itemize}
where $\mathbb{Z}\langle x \rangle$ is as defined in Example \ref{free_nonunital_pd}.

\end{prop}

\begin{proof}
We will prove the first statement only, as the proof of the second is similar. The proof goes by observing that the objects $\Env^{\pd} \LSym_{\mathbb{Z}}^{\Delta^{1}_{\vee}} (\Id:\mathbb{Z} \rightarrow \mathbb{Z})$ lifts to the $\infty$-category $\DAlg_{R}^{\nonu,\pd}$ of non-unital divided power algebras by means of the following commutative diagram

$$
\xymatrix{   \Mod_{\mathbb{Z}}^{\Delta^{1}_{\vee}} \ar[rrr]^-{\Env^{\pd} \LSym_{\mathbb{Z}}^{\Delta^{1}_{\vee}}}&&& \DAlg_{\mathbb{Z}}^{\Delta^{1}_{\vee}.\pd} \\
\Mod_{\mathbb{Z}} \ar[u] \ar[rrr]_-{\LGamma^{+}_{\mathbb{Z}}}&&& \DAlg^{\{0,1\},\vee,\pd}_{\mathbb{Z}}  \ar[u]_-{\aug} ,}
$$
where the left vertical functor sends $X \in \Mod_{\mathbb{Z}} $ to $(\Id: X \rightarrow X) \in \Mod_{\mathbb{Z}}^{\Delta^{1}_{\vee}}$, and the right vertical functor sends a non-unital derived divided power algerba $A^{+}$ to the pd Smith ideal $A^{+} \rightarrow \mathbb{Z}\oplus A^{+}$. Commutativity of this diagram follows by passing to right adjoints, as the right adjoints of horizontal functors are forgetful functors, the right adjoint of the left vertical is the functor $\ev^{1}$ evaluating at the source of the arrow, and the right adjoint of the right vertical is the functor $(\Fil \rightarrow \Gr)$ which takes a Smith ideal $(I \rightarrow A)$ to the underlying non-unital derived pd algebra of $I$. As we know that the free non-unutal divided power algebra on one generator is $\mathbb{Z}\langle x \rangle^{+}$, and the corresponding pd Smith ideal is $\mathbb{Z}\langle x \rangle^{+} \rightarrow \mathbb{Z}\langle x \rangle $, the statement follows.
\end{proof}

\begin{prop}\label{free_pd_formula}
Let $X \rightarrow Z$ be an object of $\Mod_{R,\geq 0}^{\Delta^{1}}$, and $Y =\fib(X\rightarrow Z)$. Let $\LSym_{R}(X) \rightarrow \LSym_{R}(Z)$ be the free object in $\DAlg_{R}^{\Delta^{1}}$ on $X\rightarrow Z$. There is an equivalence

$$
\Env^{\pd} \Bigl(\LSym_{R}(X) \rightarrow \LSym_{R}(Z) \Bigr)\simeq  \LGamma_{R}(Y) \underset{\LSym_{R}( Y)}{\otimes} \LSym_{R} (X)   \rightarrow \LSym_{R}(Z) 
$$
as objects of $\DAlg_{R}^{\Delta^{1}}$.
\end{prop}

\begin{proof}
Note that there is a commutative diagram

$$
\xymatrix{   \LGamma_{R}(Y)\underset{\LSym_{R}(Y)}{\otimes} \LSym_{R}(X)  \ar[d]_-{f} \ar[r]& \LSym_{R}(Z) \ar[d]^-{\Id}\\
\LSym_{R}(X) \ar[r]& \LSym_{R}(Z),   }
$$
where the map $f$ is induced by the Norm map $\LGamma_{R}(Y) \rightarrow \LSym_{R}(Y)$. To produce a natural map from the derived divided power envelope, we need to show that the fiber of the map $$ \epsilon: \LGamma_{R}(Y)\underset{\LSym_{R}(Y)}{\otimes} \LSym_{R}(X)  \rightarrow \LSym_{R}(Z) $$ has the structure of a divided power Smith ideal. This can be done by observing that formation of the map $\epsilon$ is functorial in the fiber sequence $Y \rightarrow X \rightarrow Z$ and commutes with sifted colimits, and we ran resolve the fiber sequence $Y \rightarrow X \rightarrow Z$ as a sifted colimit of split fiber sequences of the form $Y_{i} \rightarrow Y_{i}\oplus Z_{i} \rightarrow Z_{i}$. It then suffices to show that the map $$\LGamma_{R}(Z_{i}) \underset{\LSym_{R}(Z_{i}) }{\otimes} \LSym_{R}(Y_{i} \oplus Z_{i}) \rightarrow \LSym_{R}(Y_{i}) $$ is naturally a divided power map. But we can identify it with the map $ \LGamma_{R}(Z_{i}) \otimes \LSym_{R}(Y_{i}) \rightarrow \LSym_{R}(Y_{i})$ which is the tensor product of maps $\LGamma_{R}(Z_{i}) \rightarrow R$ and $\Id: \LSym_{R}(Y_{i}) \rightarrow \LSym_{R}(Y_{i})$, so the fiber of the map $ \LGamma_{R}(Z_{i}) \otimes \LSym_{R}(Y_{i}) \rightarrow \LSym_{R}(Y_{i})$ is $\LGamma_{R}^{\geq 1}(Z_{i}) \otimes \LSym_{R}(Y_{i})$ which is a non-unital divided power algebra. Therefore, the map $\epsilon$ is a divided power map, and therefore there exists a map $$\xymatrix{\Env^{\pd} \Bigl(\LSym_{R}(X) \rightarrow \LSym_{R}(Z) \Bigr) \ar[r]& \LGamma_{R}(Y)\underset{\LSym_{R}(Y)}{\otimes} \LSym_{R}(X)  \rightarrow \LSym_{R}(Z) }$$ by the universal property of derived divided power envelope. To check that this is an equivalence, since $X$ and $Z$ are connective, we can resolve the map $X\rightarrow Z$ by finitely generated free $R$-modules, and by base change deduce the statement from Remark \ref{pd_envelope_of_cp}.
\end{proof}

\begin{prop}\label{commutation_with_coproducts}
The forgetful functor $\DAlg_{R}^{\Delta^{1}_{\vee},\pd} \rightarrow \DAlg_{R}^{\Delta^{1}_{\vee}} $ commutes with all colimits.
\end{prop}

\begin{proof}
To show that it commutes with finite coproducts. Here again, we argue as in \cite[Proposition 4.2.27]{R}. Using formulas of Proposition \ref{free_pd_formula}, we see that the maps $$\xymatrix{   \Env^{\pd} \bigl(\LSym_{R}^{\Delta^{1}_{\vee}} (  0: 0 \rightarrow X  ) \bigr) \otimes  \Env^{\pd} \bigl(\LSym^{\Delta^{1}_{\vee}}_{R} (  \Id: Z \rightarrow  Z ) \bigr) \ar[r]&   \Env^{\pd} \LSym_{ R}^{\Delta^{1}_{\vee}} \bigl(    (0: 0 \rightarrow X) \oplus (\Id: Z \rightarrow Z)  \bigr)    }$$
are equivalences for all $X, Z \in \Mod_{R}^{\fg,\free}$. From this, we conclude that $\Env^{\pd}\LSym_{R}^{\Delta^{1}_{\vee}}$ transforms directs sums of compact projective generators of $\Mod_{R,\geq 0}^{\Delta^{1}_{\vee}} $ to tensor products, and commutation of the forgetful functor with coproducts follows. 
\end{proof}

\begin{defn}\label{envelopes_coenvelopes}
It follows from Proposition \ref{commutation_with_coproducts} that there is an adjunction

$$
\xymatrix{ \DAlg_{R}^{\Delta^{1}_{\vee},\pd}  \ar[rrr]^-{\forget } &&& \ar@/_1.5pc/[lll]_{\Env^{\pd}  }    \ar@/^1.5pc/[lll]^{\coEnv^{\pd}}  \DAlg_{R}^{\Delta^{1}_{\vee}},  }
$$
where we call the left adjoint the \textbf{derived divided power (pd) envelope}, and the right adjoint \textbf{derived divided power (pd) coenvelope}. 
\end{defn}

\begin{rem}\label{commutation_with_gr_ev}
The adjunction of Construction \ref{insertion_adjunction} lifts to an adjunction

$$
\xymatrix{ \DAlg_{R} \ar[rrr]^-{\ins^{0} } &&& \ar@/_1.5pc/[lll]_{\gr^{0} }    \ar@/^1.5pc/[lll]^{\ev^{0}} \DAlg_{R}^{\Delta^{1}_{\vee},\pd} .}
$$
In particular, there is a commutative square

$$
\xymatrix{ \DAlg_{R}\ar[d]_-{\Id} \ar[r]^-{\ins^{0}}& \DAlg_{R}^{\Delta^{1}_{\vee},\pd} \ar[d]^-{\forget}\\
\DAlg_{R} \ar[r]_-{\ins^{0}}& \DAlg_{R}^{\Delta^{1}_{\vee}} .     }
$$
Passing to left and right adjoints, we obtain commutative squares

$$
\xymatrix{ \DAlg_{R}  & \ar[l]_-{\gr^{0}} \DAlg_{R}^{\Delta^{1}_{\vee},\pd}  \;\;\; \;\;\;\;\;\;\;\;\; \\                 
\DAlg_{R} \ar[u]^-{\Id}  & \ar[l]^-{\gr^{0} }  \DAlg_{R}^{\Delta^{1}_{\vee}}  \ar[u]_-{\Env^{\pd}},     }
\xymatrix{ \DAlg_{R}  & \ar[l]_-{\ev^{0}} \DAlg_{R}^{\Delta^{1}_{\vee},\pd}  \\                 
\DAlg_{R} \ar[u]^-{\Id}  & \ar[l]^-{\ev^{0} }\DAlg_{R}^{\Delta^{1}_{\vee}}  \ar[u]_-{\coEnv^{\pd}} .  }
$$

\end{rem}

\begin{rem}\label{envelope_connective}
Let $ \DAlg_{R,\geq 0}^{\Delta^{1}_{\vee},\pd} \subset \DAlg_{R}^{\Delta^{1}_{\vee},\pd}  $ be the inclusion of the full subcategory of all (divided power) Smith ideals consisting of objects $(I \rightarrow A)$ with both $A$ and $I$ being connective. Then there is a commutative diagram

$$
\xymatrix{ \DAlg_{R,\geq 0}^{\Delta^{1}_{\vee}}  \ar[d] \ar[r]^-{\Env^{\pd}}& \DAlg_{R,\geq 0}^{\Delta^{1}_{\vee},\pd}  \ar[d]\\
\DAlg_{R}^{\Delta^{1}_{\vee}}  \ar[r]_-{\Env^{\pd}}& \DAlg_{R}^{\Delta^{1}_{\vee},\pd} ,   }
$$
where the vertical functors are inclusions. In other words, divided power envelopes preserve connective Smith ideals. Commutativity of the diagram follows by passing to right adjoints as the right adjoint of inclusion is connective cover which manifestly commutes with forgetful functors. 
\end{rem}

We let $\DAlg^{\pd}_{A/} $ be the $\infty$-category of divided power maps $A \rightarrow B$ from $A$, and $\DAlg^{\pd}_{A/}$ be the $\infty$-category of divided power maps $C\rightarrow A$ to $A$. Concretely, they are defined as fiber products:

$$
\xymatrix{   \DAlg^{\pd}_{R//A,\geq 0} \ar[d] \ar[r]& \DAlg_{R,\geq 0}^{\Delta^{1},\pd} \ar[d]^-{\ev^{1}}\\
\{A\} \ar[r]& \DAlg_{R,\geq 0}  .} 
\xymatrix{ \DAlg^{\pd}_{A/,\geq 0} \ar[d] \ar[r]& \DAlg_{R,\geq 0}^{\Delta^{1},\pd} \ar[d]^-{\ev^{0}}\\
\{A\} \ar[r]& \DAlg_{R,\geq 0} }
 $$

\begin{prop} 
There exist adjunctions

$$
\xymatrix{   \DAlg_{A/}^{\pd} \ar[rrr]^-{\forget } &&& \ar@/^1.1pc/[lll]^{\coEnv_{A/}^{\pd}}  \DAlg_{A/}    \;\; ,  }\;\;\;\;
\xymatrix{   \DAlg_{/A}^{\pd} \ar[rrr]_-{\forget } &&& \ar@/_1.1pc/[lll]_{\Env_{/A}^{\pd}}  \DAlg_{/A}.       }
$$

The functor $\coEnv^{\pd}_{A/}$ preserves connective objects. If $B \rightarrow A \in \DAlg_{A/,\geq 0}$ is surjective on $\pi_{0}$, then the divided power envelope $\Env^{\pd}_{/A}(B)$ is connective again. 

\end{prop}

\begin{proof}
Both statement follow from Remark \ref{commutation_with_gr_ev}.
\end{proof}

\begin{prop}
 For a derived pair $I \rightarrow A$, the divided power coenvelope $\coEnv^{\pd}(I \rightarrow A)$ is equivalent to the pair $I^{\sharp} \rightarrow A$, where $I^{\sharp}$ is as in Remark \ref{pd_nonunital}. 
\end{prop}

\begin{defn}
In the particular case $I = A$ considered as the tautological Smith ideal $\Id: A \rightarrow A$, the divided power ideal $I^{\sharp} $ will be called \textbf{divided power (pd) radical of $R$}.
\end{defn}

We will finish this subsection with the following easy observation relating our approach to derived divided power Smith ideals with the $\infty$-category of animated pd pairs of Z.Mao \cite{Mao21}.

\begin{prop}\label{connective_and_Mao}
Let $\C^{0}$ be the full subcategory of all divided power Smith ideals consisting of finite tensor products of the objects $(0) \subset \mathbb{Z}[x]$ and $\mathbb{Z}\langle x \rangle^{+} \rightarrow \mathbb{Z}\langle x \rangle$ (see Proposition \ref{pd_envelope_of_cp}). Then there is an equivalence of $\infty$-categories

$$
\DAlg^{\Delta^{1}_{\vee},\pd}_{\mathbb{Z},\geq 0} \simeq \mathcal{P}_{\Sigma}( \C^{0}  ).
$$
\end{prop}

\begin{proof}
Indeed, by monadicity of $\DAlg^{\Delta^{1}_{\vee},\pd}_{\mathbb{Z},\geq 0}$ over $\Mod_{\mathbb{Z},\geq 0}^{\Delta^{1}_{\vee}}$ and preservation of sifted colimits by the monad, it follows that compact projective generators of $\DAlg^{\Delta^{1}_{\vee},\pd}_{\mathbb{Z},\geq 0}$ are the free objects on finite coproducts of $0 \rightarrow \mathbb{Z}$ and $\Id: \mathbb{Z} \rightarrow \mathbb{Z}$ (the latter being compact projective generators of $\Mod_{\mathbb{Z},\geq 0}^{\Delta^{1}_{\vee}}$). Since coproducts in $\DAlg^{\Delta^{1}_{\vee},\pd}_{\mathbb{Z},\geq 0}$ are the same as tensor products, the statement follows from the computation of Proposition \ref{pd_envelope_of_cp}.
\end{proof}

\subsubsection{Filtered divided power algebras.}

In this subsection we will define filtered divided power algebras. A typical derived filtered divided power algebra $F^{\geq \star}A$ looks like a filtered derived algebra together with a divided power structure on the "filtered Smith ideal" $F^{\geq 1}A:=\fib( F^{\geq \star} A \rightarrow \gr^{0}A)$. To make this precise, a necessary step is to work out the definition of filtered derived power non-unital algebras over a general derived ring. As an intermediate step for this, we will introduce an auxilary construction in the world of ordinary filtered derived rings.

\begin{construction}\label{split_filtrations}
Let $\Fil^{\geq 0}\DAlg_{\mathbb{Z}}$ be the $\infty$-category of filtered derived rings. Given any object $F^{\geq \star} A$, the derived ring $F^{\geq 0}A$ acts on all higher stages of the filtration $F^{\geq i}A$. Assume that the map $F^{\geq 0}A \rightarrow \gr^{0}A $ admits a derived algebra splitting $\gr^{0}A \rightarrow F^{\geq 0}A$. Then $F^{\geq 0}A$ becomes an augmented derived ring $F^{\geq 0}A \simeq \gr^{0} A \oplus F^{\geq 1}A$ endowed with a filtered augmentation ideal $F^{\geq 1}A$, on which $\gr^{0}A$ acts via the section $\gr^{0}A \rightarrow F^{\geq 0}A$. Let $\Fil^{\geq 0}\DAlg^{\aug}_{\mathbb{Z}}$ be the $\infty$-category of filtered derived rings $F^{\geq 0}A$ endowed with a splitting $\gr^{0}A \rightarrow F^{0}A$. Sending $F^{\geq \star}A$ to the pair $(\gr^{0},\Fil^{\geq 1}A)$, we get an equivalence of this $\infty$-category and the $\infty$-category of pairs $(B,F^{\geq 1}B)$ where $B$ is a derived ring, and $F^{\geq 1}B$ is a positively filtered non-unital derived $B$-algebra. The inverse functor associates to a pair $(B,F^{\geq 1}B)$ the augmented filtered algebra $... \rightarrow F^{\geq 1}B \rightarrow B\oplus F^{\geq 1}B$. When $R$ is any connective derived ring, the left hand side can be redefined as the $\infty$-category of under objects

$$
\Fil^{\geq 0}\DAlg_{R}^{\aug}:= (\Fil^{\geq 0}\DAlg_{\mathbb{Z}}^{\aug})_{R/},
$$
where $R$ is considered as a filtered ring inserted in degree $0$\footnote{Note that if we consider $R$ as a filtered ring with constant filtration, then its associated graded ring is $0$. Therefore, any filtered ring receiving a map from $R$ considered with constant filtration, would need to have associated graded ring $0$ too, and therefore constant filtration is not the right approach.}.
\end{construction}

\begin{construction}\label{free_filtered_divided_power_algebra}

Let $R$ be a derived commutative ring. 

\begin{itemize}

\item Assume $R=\mathbb{Z}$. Then the $\infty$-category $\Fil\Mod_{\mathbb{Z}}$ has a derived filtered non-unital divided power algebra $\LGamma^{+,\geq \star}_{\mathbb{Z}}: \Fil\Mod_{\mathbb{Z}} \rightarrow \Fil\Mod_{\mathbb{Z}}$ using Construction \ref{derived_syms_from_geometry} and Construction \ref{LSym_on_fils_from_t_structure}. We let $\Fil\DAlg^{\nonu,\pd}:=\Alg_{\LGamma_{\mathbb{Z}}^{+,\geq \star}}(\Fil\Mod_{\mathbb{Z}}) $.  The subcategory $\Fil^{\geq 1}\Mod_{\mathbb{Z}} \subset \Fil\Mod_{\mathbb{Z}}$ is closed under formation of free $\LGamma_{\mathbb{Z}}^{+,\geq \star }$-algebras. Therefore, there is an $\infty$-category $$\Fil^{\geq 1}\DAlg^{\pd}_{\mathbb{Z}}:= \Fil\DAlg^{\pd}_{\mathbb{Z}}\underset{\Fil\Mod_{\mathbb{Z}}}{\times} \Fil^{\geq 1}\Mod_{\mathbb{Z}}  $$

of \textbf{derived filtered non-unital divided power $R$-algebras}, which is monadic over $\Fil^{\geq 1}\Mod_{\mathbb{Z}}$, and the monad right-left extended from the subcategory $\Ind(\Fil^{\geq 1}\Mod_{\mathbb{Z},0})$ by Remark \ref{right_left_extended_LSym_on_fils}.

\item If $R$ is not necessarily discrete, or even connective, we can define the $\infty$-category $\Fil^{\geq 1}\DAlg^{\pd}_{A}$ using Construction \ref{split_filtrations}. There is a forgetful functor $\Fil^{\geq 1}\DAlg^{\pd}_{\mathbb{Z}} \rightarrow \Fil^{\geq 1}\DAlg_{\mathbb{Z}}$, and a forgetful functor $ \Fil^{\geq 0}\DAlg^{\aug}_{\mathbb{Z}} \rightarrow \Fil^{\geq 1}\DAlg_{\mathbb{Z}}$.  Then we define the $\infty$-category of \textbf{augmented filtered derived divided power rings} as the pull-back

$$
\xymatrix{  \Fil^{\geq 0}\DAlg_{\mathbb{Z}}^{\pd,\aug} \ar[d] \ar[r] & \Fil^{\geq 0}\DAlg_{\mathbb{Z}}^{\aug} \ar[d] \\
\Fil^{\geq 1}\DAlg_{\mathbb{Z}}^{\pd} \ar[r]& \Fil^{\geq 1}\DAlg_{\mathbb{Z}}, }
$$
i.e. an augmented filtered divided power ring is an augmented filtered derived ring endowed with a divided power structure on the filtered augmentation ideal. Let $\Fil^{\geq 0}\Mod_{\mathbb{Z}}^{\aug}$ be the $\infty$-category of \textbf{augmented filtrations}, i.e. object $F^{\geq \star} \in \Fil^{\geq 0}\Mod_{\mathbb{Z}}$ endowed with a splitting of the map $F^{\geq 0} \rightarrow \gr^{0}_{F}$. We think of objects of $\Fil^{\geq 0}\Mod_{\mathbb{Z}}^{\aug}$ as pairs $(X, F^{\geq 1}X)$. The $\infty$-category $\Fil^{\geq 0}\DAlg_{\mathbb{Z}}^{\pd,\aug}$ is monadic over $\Fil^{\geq 0}\Mod^{\aug}_{\mathbb{Z}}$. Let $\xymatrix{  (\LSym, \LGamma^{\geq \star}): \Fil^{\geq 0}\Mod_{\mathbb{Z}}^{\aug} \ar[r]& \Fil^{\geq 0}\Mod_{\mathbb{Z}}^{\aug}     } $ be the resulting monad. It satisfies the formula $$(\LSym, \LGamma^{\geq \star}) (X, F^{\geq 1}X) \simeq (\LSym_{\mathbb{Z}}(X), \LGamma^{\geq \star}_{\LSym_{\mathbb{Z}}(X)}  ( F^{\geq 1}X \otimes \LSym_{\mathbb{Z}}(X))   ) .$$

 By definition, the fiber of the forgetful functor $\ev^{0}: \Fil^{\geq 0}\DAlg_{\mathbb{Z}}^{\pd,\aug} \rightarrow \DAlg_{\mathbb{Z}} $ over some $A$ is the $\infty$-category of \textbf{non-unital filtered derived divided power $A$-algebras}\footnote{Compare this with Definition \ref{pd_over_nonconnective}.}. We denote $\LGamma_{A}^{\geq \star}: \Fil^{\geq 1}\Mod_{A} \rightarrow \Fil^{\geq 1}\Mod_{A}$ he free algebra monad. Similarly, if $R$ is a connective derived ring, we can consider $R$ as a filtered derived ring inserted in degree $0$, and let $$\Fil^{\geq 0}\DAlg_{R}^{\pd,\aug}:= \Fil^{\geq 0}\DAlg_{\mathbb{Z}}^{\pd,\aug} \underset{\Fil^{\geq 0}\DAlg_{\mathbb{Z}}^{\aug} }{\times} \Fil^{\geq 0}\DAlg_{R}^{\aug}. $$

Similarly, we have an $\infty$-category of \textbf{graded derived divided power rings} $\Gr^{\geq 0}\DAlg^{\pd}_{R}$ whose fiber over any $A\in \DAlg_{R}$ is the $\infty$-category of \textbf{non-unital graded derived divided power $A$-algebras}, and the free algebra monad is denoted as $\LGamma^{+,\bullet}_{A}: \Gr^{\geq 0}\Mod_{A} \rightarrow \Gr^{\geq 1}\Mod_{A}. $

\end{itemize}

\end{construction}

For future purposes, we specifically fix the following notation.

\begin{notation}\label{free_adjunction_for_gradeds}
Let $\DAlg_{R}\Mod$ be the $\infty$-category of pairs $(A,M)$, where $A\in \DAlg_{R}$ and $M \in \Mod_{A}$. The functor of \textbf{free graded divided power algebra on a module}  is the functor:

$$\xymatrix{ \DAlg_{R}\Mod \ar[rr]^-{\LGamma^{\bullet}}&& \Gr^{\geq 0}\DAlg_{R}^{\pd}    }$$ 
which sends a pair $(A,M) \in \DAlg_{R}\Mod$ to $(A,\LGamma^{+,\bullet}_{A}(M))$

The right adjoint of this functor is the functor $$\xymatrix{\Gr^{\geq 0}\DAlg_{R}^{\pd} \ar[rrr]^-{ (-)^{\{0,1\}}\circ \: \forget_{\pd}}&&& \DAlg_{R}\Mod.}$$

In the same vein, we define a \textbf{free graded derived algebra on a module} $\LSym^{\bullet}_{A}(M)$ for any pair $(A,M) \in \DAlg_{R}\Mod$.
\end{notation}

\begin{ex}
The free filtered non-unital divided power algebra $\Gamma^{\geq 1}_{\mathbb{Z}}(M)$ on a finitely gerated free $\mathbb{Z}$-module $M$ (put in filtered weight $1$) has filtration with stages $$\Gamma_{\mathbb{Z}}^{\geq n}(M):=\bigoplus_{i=n}^{\infty} \Gamma^{i}_{\mathbb{Z}}(M),$$ where $n\geq 1$. Concretely, the $n$-th stage of filtration consists of elements of the form: 

$$
\Gamma^{ \geq n}_{\mathbb{Z}}(M) = \{ \sum_{i_{1}+...i_{k}\geq n} \gamma^{i_{1}}(x_{1})... \gamma^{i_{k}}(x_{k}) | x_{1},...,x_{k} \in M, k \geq 1\}.
$$

\end{ex}

\begin{construction}\label{splitting_fils}
We define a functor $\Fil^{\geq 0}\DAlg_{R} \rightarrow \Fil^{\geq 0}\DAlg_{R}^{\aug}$ by sending an object $F^{\geq \star}A$ to the augmented filtered derived algebra $... \rightarrow F^{\geq 1} A \rightarrow F^{\geq 0}A \oplus F^{\geq 1}A $ so that the map $F^{\geq 0} A\oplus F^{\geq 1} A \rightarrow F^{\geq 0} A $ has a (zero) section\footnote{This is analogous to the functor $(\Fil \rightarrow \Gr):\DAlg^{\Delta^{1}_{\vee}}_{R} \rightarrow \DAlg^{\{0,1\},\vee}_{R} \simeq \DAlg^{\Delta^{1}_{\vee},\aug}_{R} $ which sends a Smith ideal $(I\rightarrow A)$ to the augmented Smith ideal $(I \rightarrow A\oplus I)$.}. 
\end{construction}

\begin{defn}\label{fil_pds}
 We define an $\infty$-category $\Fil^{\geq 0} \DAlg^{\pd}_{R}$  of \textbf{filtered derived divided power (pd) algebras}  as the fiber square

$$
\xymatrix{\Fil^{\geq 0} \DAlg^{\pd}_{R} \ar[d] \ar[rr]&& \Fil^{\geq 0}\DAlg_{R} \ar[d]\\
\Fil^{\geq 0}\DAlg_{R}^{\pd,\aug}  \ar[rr]&& \Fil^{\geq 0}\DAlg^{\aug}_{R},  }
$$
where the lower horizontal functor is the forgetful functor, and the right vertical functor is as defined in Construction \ref{splitting_fils}.
\end{defn}

\begin{rem}\label{triv_obs}
Before the next construction, we make a trivial observation. The functor $\ev^{[0,1]}: \Fil^{\geq 0}\DAlg_{R} \rightarrow \DAlg^{\Delta^{1}_{\vee}}_{R}$ constructed in Construction \ref{ev[0,1]_for_derived_rings} lifts to a functor $$\ev^{[0,1]}: \Fil^{\geq 0}\DAlg_{R}^{\aug} \xymatrix{\ar[r]&} \DAlg^{\Delta^{1}_{\vee}, \aug}_{R} \simeq \DAlg_{R}^{\{0,1\},\vee}$$ as for any object $F^{\geq \star}A \in \Fil^{\geq 0}\DAlg_{R}^{\aug}$, the Smith ideal $F^{\geq 1}A \rightarrow F^{0}A $ has a splitting.
\end{rem}

\begin{construction}\label{fils_to_nonus}
We will construct a functor $\ev^{[0,1]}: \Fil^{\geq 0}\DAlg_{R}^{\pd,\aug} \rightarrow \DAlg^{\{0,1\},\vee,\pd}_{R}$ for any connective derived commutative ring $R$, lifting the functor $\ev^{[0,1]} : \Fil^{\geq 0}\DAlg^{\aug}_{R} \rightarrow \DAlg^{\{0,1\},\vee}_{R}$ provided by Remark \ref{triv_obs}. This is sufficient to do in the case $A=\mathbb{Z}$. In this case, the derived divided power algebra monad is right-left extended. We observe that there is a commutative diagram

$$
\xymatrix{    \Ind(\Fil^{\geq 0}\Mod^{\aug}_{\mathbb{Z},0}) \ar[d]_-{(\LSym, \LGamma^{\geq \star})} \ar[r]^-{\ev^{[0,1]}} &\Ind( \Mod_{\mathbb{Z},0}^{\{0,1\},\vee}) \ar[d]^-{(\LSym, \LGamma^{+})} \\
 \Ind(\Fil^{\geq 0}\Mod^{\aug}_{\mathbb{Z},0}) \ar[r]_-{\ev^{[0,1]}}&\Ind( \Mod_{\mathbb{Z},0}^{\{0,1\},\vee})     }
$$
Therefore, by right-left extension we obtain a functor $\ev^{[0,1]}: \Fil^{\geq 0}\DAlg_{\mathbb{Z}}^{\pd,\aug} \rightarrow \DAlg^{\{0,1\},\vee,\pd}_{\mathbb{Z}}. $
\end{construction}

For a commutative algebra $A$ with a divided power ideal $I\subset A$, we can construct the filtration by ideals $... \subset I^{[3]} \subset I^{[2]} \subset I \subset A$ where $I^{[n]}$ is spanned by elements $x_{1}^{[i_{1}]} ... x_{n}^{[i_{n}]}$ with $i_{1}+...+i_{n}\geq n$. The associated graded algebra for this filtration is the free graded divided power algebra $\Gamma_{A/I}(I/I^{2})=\bigoplus_{n\geq 0} I^{[n]}/I^{[n+1]}$.

\begin{construction}\label{generalized_adic}
The functor $\ev^{[0,1]}$ lifts to a limit preserving functor $\ev^{[0,1]}: \Fil^{\geq 0}\DAlg^{\pd} \rightarrow \DAlg_{R}^{\Delta^{1}_{\vee},\pd}$ on the $\infty$-categories of derived divided power algebras. For that, using Definitions \ref{fil_pds} and \ref{divided_power_pairs}, we need to show that the functor $\ev^{1}: \Fil^{\geq 1}\DAlg^{\pd}_{A} \rightarrow \DAlg^{\nonu}_{A}$ lifts to a functor $\ev^{1}: \Fil^{\geq 1}\DAlg^{\pd}_{A} \rightarrow \DAlg^{\pd,\nonu}_{A}$ naturally for any $A\in \DAlg_{R}$, which is done in Construction \ref{fils_to_nonus}. We let the functor of \textbf{derived divided power (pd) adic filtration}
$$
\xymatrix{\adic_{\pd}: \DAlg_{R}^{\Delta^{1}_{\vee},\pd} \ar[r]& \Fil^{\geq 0}\DAlg^{\pd}_{R}  }
$$
to be left adjoint of $\ev^{[0,1]}: \Fil^{\geq 0}\DAlg^{\pd}_{R} \rightarrow \DAlg_{R}^{\Delta^{1}_{\vee},\pd}$.

Similarly, there is a functor $\ev^{[0,1]}: \Fil^{\geq 0}\DAlg_{R} \rightarrow \DAlg_{R}^{\Delta^{1}_{\vee}}$ and its left adjoint \textbf{derived adic filtration} functor $\adic: \DAlg_{R}^{\Delta^{1}_{\vee}} \rightarrow \Fil^{\geq 0}\DAlg_{R}$.

\end{construction}

In the next section we will compute associated gradeds of the (pd) adic filtration. We will also identify the first truncation of the (pd) adic filtraton with the universal square-zero extension.

\subsection{Two-step filtrations and square-zero extensions.}

In this subsection we will develop the formalism of square-zero extensions from the point of view of two-step filtrations. Our motivation for the constructions of this subsection comes from the following situation. Let $\LOmega^{\geq \star}_{A/R}$ be the derived De Rham cohomology of $A$ over $R$. We can truncate it by filtered pieces of weights $\geq 2$ to get a "derived square-zero extension" of $A$ by the shift of the cotangent complex $\LL_{A/R}[-1]$. It is expectable that this square-zero extension is controlled by the \emph{universal derivation}. The formalism developed in this subsection gives a simple proof of this fact which has the advantage of being simple to generalize to other settings.

\begin{defn}
Let $\{0,1\}$ be the subset of $\mathbb{Z}^{\ds}$ consisting of two objects $0,1 \in \mathbb{Z}$, considered as a discrete category. Similarly, let $(0<1)$ be the full subcategory of $\mathbb{Z}_{\leq *}$ spanned by these two objects. Fix a connective derived commutative ring $R$.

\begin{itemize}

\item We define the $\infty$-category of \textbf{two-step graded} objects  as follows: $\Gr^{\{0,1\}}(\Mod_{R}):= \Fun(\{0,1\}^{\op},\Mod_{R})$. An object of this $\infty$-category is the same as a pair of objects $(X^{0}, X^{1})$. This $\infty$-category is a localization of the $\infty$-category $\Gr^{\geq 0}\Mod_{R}$. We let $(-)^{\{0,1\}}: \Gr^{\geq 0}\Mod_{R} \rightarrow \Gr^{\{0,1\}}\Mod_{R}$ to be the localization functor. It sends a positively graded object $\{X^{i}\}_{i\geq 0} \in \Gr^{\geq 0} \Mod_{R}$ to the two-step graded object $(X^{0},X^{1}) \in \Gr^{\{0,1\}}\Mod_{R}$. The $\infty$-category $\Gr^{\{0,1\}}\Mod_{R}$ has a symmetric monoidal structure such that the localization functor is symmetric monoidal.

\item Similarly, we define the $\infty$-category of \textbf{two-step filtrations} $\Fil^{[0,1]}(\Mod_{R}):= \Fun((0< 1)^{\op}, \Mod_{R})$. An object of this $\infty$-category is the same as a map $X^{1}\rightarrow X^{0}$ in $\Mod_{R}$. As with two-step graded objects, this $\infty$-category is a localization of $\Fil^{\geq 0} \Mod_{R}$. The localization functor $(-)^{\leq 1}: \Fil^{\geq 0}\Mod_{R} \rightarrow \Fil^{[0,1]}\Mod_{R} $ is given by the formula

$$
(... \rightarrow X^{2} \rightarrow X^{1} \rightarrow X^{0})^{\leq 1} \simeq (X^{1}/X^{2} \rightarrow X^{0}/X^{2}).
$$

The $\infty$-category $\Fil^{[0,1]}\Mod_{R}$ has a symmetric monoidal structure making the localization functor symmetric monoidal. 

\end{itemize}
\end{defn}

Let $\DAlg_{R}\Mod$ be the $\infty$-category of pairs $(A,M)$ with $A\in \DAlg_{R}$ and $M \in \Mod_{A}$, and $\Gr^{\{0,1\}}\Mod_{R}\subset \Gr^{\geq 0}\Mod_{R}$ be the full subcategory of non-negatively filtered $R$-module spectra consisting of objects concentrated in two degrees. There is a symmetric monoidal structure on $\Gr^{\{0,1\}}\Mod_{R}$ such that the functor $(-)^{\{0,1\}}: \Gr^{\geq 0}\Mod_{R} \rightarrow \Gr^{\{0,1\}}\Mod_{R}$, left adjoint to the inclusion, is symmetric monoidal. The following observation is due to A.Raksit, see \cite{R}. 

\begin{prop}
There is an equivalence of $\infty$-categories $$\xymatrix{\Gr^{\{0,1\}}\DAlg_{R} \simeq  \DAlg_{R} \Mod.}$$
\end{prop}

\begin{rem}
As the functor $\Fil^{[0,1]}\Mod_{R} \rightarrow \Fil^{\geq 0}\Mod_{R}$ (respectively, $\Gr^{\{0,1\}}\Mod_{R} \rightarrow \Gr^{\geq 0}\Mod_{R}$) is an inclusion, we can say that a derived commutative algebra in $\Fil^{[0,1]}\Mod_{R} $ (respectively, in $\Gr^{\{0,1\}}\Mod_{R}$) is the same as a derived commutative algebra in $\Fil^{\geq 0}\Mod_{R}$ ($\Gr^{\geq 0}\Mod_{R}$) whose underlying linear object is in $\Fil^{[0,1]}\Mod_{R} $ ($\Gr^{\{0,1\}}\Mod_{R}$). We will use notations $\Fil^{[0,1]}\DAlg_{R}$ for the $\infty$-category of \textbf{two-step filtered derived algebras} and $\Gr^{\{0,1\}}\DAlg_{R}$ for the $\infty$-category of \textbf{two-step graded derived algebras} respectively.
\end{rem}

We mentioned an equivalence $\Gr^{\{0,1\}}\DAlg_{R} \simeq \DAlg_{R}\Mod $ before. Let us now compute the free two-step graded derived commutative algebra $\LSym_{R}^{\{0,1\}}(M^{0}\oplus M^{1})$ on an object $M^{0}\oplus M^{1} \in \Gr^{\{0,1\}}\Mod_{R}$. For a derived commutative ring $A$ and an $A$-module, we let $A\oplus M$ be the derived commutative algebra in $\Gr^{\{0,1\}}\Mod_{R}$ formed by taking a trivial square zero extension of $A$ by $M$ put in weight 1. 

\begin{prop}\label{killing_higher_powers}
There is an equivalence of derived commutative rings in $\Gr^{\{0,1\}}\Mod_{R} $:

$$
\LSym_{R}^{\{0,1\}}(M^{0}\oplus M^{1}) \simeq \LSym_{R}(M^{0}) \oplus M^{1}.
$$

\end{prop}

\begin{proof}

Indeed, let $\LSym_{R}^{\bullet }(M^{0} \oplus M^{1})$ be the free graded derived commutative ring on $(M^{0},M^{1})$ considered as an object of $\Gr^{\geq 0} \Mod_{R}$. There is a map of derived commutative algebras $\LSym^{\bullet}_{R}(M^{0}\oplus M^{1}) \rightarrow \LSym_{R} (M^{0}) \oplus M^{1}$ obtained by moding out all higher symmetric powers of $M^{1}$. Applying symmetric monoidal truncation functor $(-)^{\leq 1}:\Gr^{\geq 0}\Mod_{R} \rightarrow \Gr^{\{0,1\}}\Mod_{R}$, we obtain a map $\LSym_{R}^{\{0,1\} } (M^{0}\oplus M^{1}) \simeq \LSym^{\bullet}_{R}(M^{0}\oplus M^{1})^{\leq 1} \rightarrow \LSym_{R} (M^{0}) \oplus M^{1}$. By definition, truncation functor kills all pieces in weights $\geq 2$. Since $n$-th derived symmetric power of an object in weight $1$ is in weight $n$, it follows that truncation of a free derived commutative algebra on an object $M^{0}\oplus M^{1}$ kills all higher symmetric powers of $M^{1}$ of degree $n\geq 2$. Therefore, the constructed map is an equivalence.
\end{proof}

The following remark enables us to define a symmetric monoidal $\infty$-category of $\mathbb{D}_{-}$-module objects in $\Gr^{\{0,1\}}\Mod_{R}.$

\begin{construction}\label{D_-_in_0,1}
The algebra $\mathbb{D}_{-}$ has the structure of a \emph{derived cocommutative bialgebra} in $\Gr^{\{0,1\}}\Mod_{R}$. To see this, take the free graded associative algebra $\T(R(1) [-1])$ on the shift of the object $R(1)[-1] \in \Gr^{\geq 0} \Mod_{R}$. It has a structure of a derived cocommutative bialgebra in $\Gr^{\geq 0} \Mod_{R}$. Applying symmetric monoidal truncation functor $(-)^{\leq 1}$, wee see that the trivial square zero extension algebra $\T(R (1)[-1])^{\leq 1} \simeq R \oplus R(1)[-1] \simeq \mathbb{D}_{-}$ gets the structure of a derived cocommutative bialgebra in $\Gr^{\{0,1\}}\Mod_{R}$. 
\end{construction}

\begin{defn}
Using Construction \ref{D_-_in_0,1}, we define an $\infty$-category of \textbf{differential two-step graded modules} $$\DG^{\{0,1\}}_{-}\Mod_{R} := \Mod_{\mathbb{D}_{-}} (\Gr^{\{0,1\}} \Mod_{R})$$ as $\mathbb{D}_{-}$-modules in $\Gr^{\{0,1\}}\Mod_{R}$ with the symmetric monoidal structure coming from the cocommutative bialgebra structure on 
$\mathbb{D}_{-}$. In particular, the forgetful functor $\DG^{\{0,1\}}_{-} \Mod_{R} \rightarrow \Gr^{\{0,1\}} \Mod_{R}$ is symmetric monoidal. We will also consider the $\infty$-category 
$$\DG^{\{0,1\}}\DAlg_{R} :=\DAlg(\DG^{\{0,1\}}_{-} \Mod_{R})$$ of derived commutative algebra objects in this symmetric monoidal $\infty$-category, which is the same as derived differential graded algebras concentrated in degrees $0,1$. These are referred to as \textbf{derived differential two-step graded algebras in $\Mod_{R}$}.
\end{defn}

The filtered cousine of differential two-step graded algebras is the $\infty$-category $\Fil^{[0,1]}\DAlg_{R}$ of non-negatively filtered derived algebras which are concentrated in degrees $0,1$. We can also understand them as derived commutative algebras in the symmetric monoidal $\infty$-category $\Fil^{[0,1]}\Mod_{R}$ (with the symmetric monoidal structure making truncation functor $(-)^{\leq 1}: \Fil^{\geq 0}\Mod_{R} \rightarrow \Fil^{[0,1]}\Mod_{R}$ symmetric monoidal. Note that the functor $\gr:\Fil^{\geq 0}\Mod_{R} \rightarrow \DG_{-}^{\geq 0} \Mod_{R} $ restricts to a symmetric monoidal functor $\gr: \Fil^{[0,1]}\Mod_{R} \rightarrow \DG_{-}^{\{0,1\}} \Mod_{R} $.

\begin{prop}\label{two_step_fils_via_differential}
The functor $$\xymatrix{\gr: \Fil^{[0,1]}\Mod_{R} \ar[r]^-{\sim} & \DG_{-}^{\{0,1\}}  \Mod_{R}}$$  is an equivalence of symmetric monoidal $\infty$-categories.
\end{prop}

\begin{proof}
Indeed, this follows from the fact that $\gr:\widehat{\Fil^{\geq 0} \Mod } \rightarrow \DG_{-}^{\geq 0}  \Mod$ is an equivalence of symmetric monoidal $\infty$-categories, and the fact that any two-step filtration is automatically complete.
\end{proof}

\begin{cor}\label{two_step_fils_via_differential_1}
The functor $\gr$ lifts to an equivalence of $\infty$-categories

$$
\xymatrix{\gr: \Fil^{[0,1]}\DAlg_{R} \ar[r]^-{\sim} & \DG_{-}^{\{0,1\}}\DAlg_{R}.    }
$$
\end{cor}

The category of discrete two-step graded commutative algebras $\Fil^{[0,1]}\CAlg_{R,\heartsuit}$ can be seen to be equivalent to the full subcategory of the category of algebras with a Smith ideal spanned by objects $I \rightarrow A$ for which the multiplication map $I\otimes_{A} I \rightarrow I$ is zero. In other words, this is the category of "generalized square-zero extensions". It is a classical fact in derived algebraic geometry that square-zero extensions are controlled by derivations. The equivalence of Corollary \ref{two_step_fils_via_differential_1} suggests that the right hand side should have a description in terms of derivations, and the left hand-side in terms of derived square-zero Smith ideals. Below we will show that this is indeed the case.

In what follows, we fix a connective derived commutative ring $R$. Recall from Construction \ref{ideals_nonunitals} that there is a lax symmetric monoidal structure on the functor $(\Fil \rightarrow \Gr): \Mod_{R}^{\Delta^{1}_{\vee}} \rightarrow \Mod_{R}^{\{0,1\}, \vee}$.  We observe the following proposition.

\begin{construction}\label{taut_inclusion}
The tautological functor $\Gr^{\{0,1\}}\Mod_{R} \rightarrow \Mod_{R}^{\{0,1\},\vee} $ has a lax symmetric monoidal structure. We can realize it as a composition of two lax symmetric monoidal functors. First, we use lax symmetric monoidal inclusion $\Gr^{\{0,1\}}\Mod_{R} \hookrightarrow \Gr^{\geq 0}\Mod_{R}$. Next, we consider the functor $\ev^{\{0,1\}}: \Gr^{\geq 0}\Mod_{R} \xymatrix{\ar[r]&} \Mod_{R}^{\{0,1\},\vee} $ defined as a composition of lax symmetric monoidal functors 

$$
\xymatrix{\Gr^{\geq 0}\Mod_{R} \ar[r]^-{\ev^{\{0,1\}}} \ar[d]_-{\spl} & \Mod^{\{0,1\},\vee}_{R}    \\
\Fil^{\geq 0}\Mod_{R} \ar[r]_-{\ev^{[0,1]}} & \Mod_{R}^{\Delta^{1}_{\vee}}    \ar[u]_-{(\Fil \rightarrow \Gr)} . }
$$

The composition of the inclusion $\Gr^{\{0,1\}}\Mod_{R} \hookrightarrow \Gr^{\geq 0}\Mod_{R}$ with $\ev^{\{0,1\}}$ thus has a lax symmetric monoidal structure.
\end{construction}

\begin{construction}

We will construct a commutative square in the $\infty$-category $\CAlg(\Pr^{L})_{\lax}$ of symmetric monoidal $\infty$-categories and lax symmetric monoidal functors

\begin{equation}\label{commutative_square_of_laxs}
\xymatrix{  \Fil^{[0,1]}\Mod_{R} \ar[d] \ar[r]& \Mod_{R}^{\Delta^{1}_{\vee}} \ar[d]^-{  (\Fil \rightarrow \Gr) }\\
\Gr^{\{0,1\}}\Mod_{R} \ar[r]_-{\ev^{\{0,1\}}} & \Mod_{R}^{\{0,1\},\vee}.  }
\end{equation}

The upper horizontal functor $\Fil^{[0,1]}\Mod_{R} \xymatrix{\ar[r]&} \Mod_{R}^{\Delta^{1}_{\vee}}$ is defined as the lax symmetric monoidal composition $$\xymatrix{\Fil^{[0,1]}\Mod_{R} \hookrightarrow \Fil^{\geq 0}\Mod_{R} \ar[r]^-{\ev^{[0,1]}}& \Mod_{R}^{\Delta^{1}_{\vee}} }.$$ 

The left vertical functor is obtained by restricting the forgetful functor $(\Fil \rightarrow \Gr): \Fil^{\geq 0}\Mod_{R} \xymatrix{ \ar[r]&} \Gr^{\geq 0}\Mod_{R}$ to two-step filtered objects. By construction, the square \ref{commutative_square_of_laxs} fits into the larger commutative diagram 

$$
\begin{tikzcd}
  & \Fil^{\geq 0}\Mod_{R} \arrow{rd}{\ev^{[0,1]}} \arrow{ddd}{(\Fil \rightarrow \Gr)} &\\
 \Fil^{[0,1]}\Mod_{R} \arrow[hook]{ru} \ar{d}[swap]{(\Fil \rightarrow \Gr)} && \Mod_{R}^{\Delta^{1}_{\vee}}\arrow{d}{(\Fil \rightarrow \Gr)}\\
\Gr^{\{0,1\}}\Mod_{R} \arrow[hook]{rd} && \Mod_{R}^{\{0,1\},\vee}   \\
& \Gr^{\geq 0}\Mod_{R} \arrow{ru}[swap]{\ev^{\{0,1\}}} ,&      
\end{tikzcd}
$$
where both the left and the right halves are commutative diagrams of symmetric monoidal $\infty$-categories and lax symmetric monoidal functors. Therefore, the initial square diagram is a commutative diagram of lax symmetric monoidal functors.
\end{construction}

\begin{prop}\label{pull_back_for_fil_0,1}
The commutative diagram \ref{commutative_square_of_laxs} is a pull-back square in $\CAlg(\Pr^{L})_{\lax}$.
\end{prop}

\begin{proof}
Indeed, by Proposition \ref{limits_in_Cat_lax} this can be checked on the level of underlying $\infty$-categories. Since the lower horizontal functor and the upper horizontal functor are both equivalences, it follows that the diagram is a pull-back diagram. 
\end{proof}

\begin{observation}
Assume $R$ is a discrete ring. Then the $\infty$-categories appearing in the diagram \ref{commutative_square_of_laxs} have an action of $\SSeq^{\gen}$ and all the functors are lax $\SSeq^{\gen}$-linear. Indeed, in the diagram \ref{commutative_square_of_laxs} all functors are $t$-exact, and the whole diagram is derived from the diagram

$$
\xymatrix{  \Fil^{[0,1]}\Mod_{R,0} \ar[d] \ar[r]& \Mod_{R}^{\Delta^{1}_{\vee},0} \ar[d]^-{  (\Fil \rightarrow \Gr) }\\
\Gr^{\{0,1\}}\Mod_{R,0} \ar[r]_-{\ev^{\{0,1\}}} & \Mod_{R,0}^{\{0,1\},\vee} }
$$
of small additive categories. In the latter diagram all functors are lax symmetric monoidal, and are lax equivariant with respect to the $\SSeq^{\gen}_{0}$-action. Then the desired observation follows from Example \ref{obtaining_diagrams} by applying the functor $\Mod: \cat^{\add,\poly} \rightarrow \cat^{\St,\Sigma}$. To summarize, the diagram \ref{commutative_square_of_laxs} is a commutative diagram in $\Mod_{\SSeq^{\gen}}(\cat)_{\lax}$, and the same argument as in Proposition \ref{pull_back_for_fil_0,1} shows that this in fact a pull-back square in $\Mod_{\SSeq^{\gen}}(\cat)_{\lax}$. 
\end{observation}

The following proposition follows from the previous remark and Proposition \ref{main_proposition_on_monads}.

\begin{prop}\label{two_step_algebras_pullback}
Let $R$ be a discrete ring. The commutative square

$$
\xymatrix{\Fil^{[0,1]}\DAlg_{R} \ar[r] \ar[d] & \DAlg_{R}^{\Delta^{1}_{\vee}} \ar[d]\\
\Gr^{\{0,1\}}\DAlg_{R} \ar[r]&    \DAlg_{R}^{\{0,1\},\vee}  }
$$
induced by the commutative square \ref{commutative_square_of_laxs}, is a pull-back of $\infty$-categories.
\end{prop}

\begin{rem}\label{two_step_algebras_pullback_nondiscrete}
Let $R$ be a connective derived ring. Then we have $\Fil^{[0,1]}\DAlg_{R} \simeq (\Fil^{[0,1]}\DAlg_{\mathbb{Z}})_{R/}$, where $R$ is considered as a two-step filtered derived algebra by inserting it in degree $0$. Consequently, the pull-back diagram of Proposition \ref{two_step_algebras_pullback} holds true in this case as well.
\end{rem}

\begin{construction}
Let $\LSym^{\geq 1}: \Mod_{R} \rightarrow \Mod_{R}$ be the non-unital derived symmetric algebra functor. It has the structure of a monad such that the $\infty$-category of $\LSym^{\geq 1}$-algebras is equivalent to the $\infty$-category $\DAlg^{\nonu}_{R}$ of non-unital derived commutative algebras in $\Mod_{R}$. There is an augmentation map $\LSym^{\geq 1} \rightarrow \Id$ which kills all pieces of $\LSym^{\geq 1}$ in degrees greater than $1$. This map gives rise to an adjunction

$$
\xymatrix{  \DAlg^{\{0,1\},\vee}_{R}  \ar@/^1.1pc/[rrr]^-{\LL} &&& \ar@/^1.1pc/[lll]^-{\triv}  \Gr^{\{0,1\}}\DAlg_{R} \simeq \DAlg_{R}\Mod ,}
$$
where the right adjoint is the \textbf{trivial algebra functor} and the left adjoint is the functor of taking the \textbf{cotangent fiber}. For an $\LSym^{\geq 1}$-algebra $A$, the data of descending $A$ to an $\Id$-algebra along the map $\LSym^{\geq 1} \rightarrow \Id$ will be called a \textbf{trivialization data} on $A$.
\end{construction}

\begin{defn}\label{Smith_sqzero}
Let $R$ be a connective derived commutative ring. We define the $\infty$-category of \textbf{square-zero Smith ideals} as the pull-back

$$
\xymatrix{   \DAlg_{R}^{\Delta^{1}_{\vee},\sqzero} \ar[d] \ar[r]& \DAlg_{R}^{\Delta^{1}_{\vee}} \ar[d]^-{(\Fil \rightarrow \Gr)}\\
\DAlg_{R}\Mod \ar[r]_-{\triv} &\DAlg_{R}^{\{0,1\},\vee}   . }
$$

\end{defn}

\begin{prop}\label{Smith_sqzer_vs_two_step}
There is an equivalence of $\infty$-categories $$\Fil^{[0,1]}\DAlg_{R} \simeq \DAlg^{\Delta^{1}_{\vee},\sqzero}_{R}. $$
\end{prop}

\begin{proof}
Indeed, in the case of $R$ being discrete, it follows from Proposition \ref{pull_back_for_fil_0,1} and Proposition \ref{two_step_algebras_pullback}. Unwinding the definitions, we see that this recovers the pull-back square of Definition \ref{Smith_sqzero}. If $R$ is not discrete, see Remark \ref{two_step_algebras_pullback_nondiscrete}.
\end{proof}

To unravel an algebra-geometric description of the $\infty$-category $\DG_{-}^{\{0,1\}}\DAlg_{R}$, we begin with some remark concerning the underlying stable $\infty$-category $\DG_{-}^{\{0,1\}}\Mod_{R}$ for a connective derived ring $R$.

\begin{rem}\label{free_alg_=sqzero_alg}
The free associative algebra in $\Gr^{\{0,1\}}\Mod_{R}$ is given by the formula $\T^{\{0,1\}} (X) \simeq \T(X)^{\leq 1}$, we can interpret $\mathbb{D}_{-}$ as the \emph{free associative algebra} in $\Gr^{\{0,1\}}\Mod_{R}$ on the object $R[-1](1) \in \Gr^{\{0,1\}}\Mod_{R}$. Consequently, an object $M \in \DG^{\{0,1\}}\Mod_{R}$ is the same as an object $M \in \Gr^{\{0,1\}}\Mod_{R}$ together with a map $M(1)[-1] \simeq M \otimes R [-1](1) \rightarrow M$ in $\Gr^{\{0,1\}}\Mod_{R}$ (with square-zero condition being futile for degree reasons). By duality, the latter is equivalent to giving a map $M \rightarrow M\otimes R[1](-1) \simeq M(-1)[1]$. If $M$ has graded components $M^{0}, M^{1}$, then computing the tensor product in $\Gr^{\{0,1\}}\Mod_{R}$ wee see that $M(-1)[1]$ is equivalent to the object $M^{1}[1]$ put in weight $0$. This is in turn equivalent to giving a splitting $M^{0} \rightarrow M^{0} \oplus M^{1}[1]$ of the projection $M^{0} \oplus M^{1}[1] \rightarrow M^{0}$. Therefore, we have the following computation of the mapping space as the fiber product

$$
\xymatrix{  \Map_{\DG^{\{0,1\}}\Mod_{R}} (M , M(-1)[1])  \ar[d] \ar[r]& \Map_{\Mod_{R}} (M^{0}, M^{0} \oplus M^{1}[1]) \ar[d]\\
  \{\Id_{M^{0}}\} \ar[r]& \Map_{\Mod_{R}} (M^{0},M^{0}). }
$$

\end{rem}

\begin{construction}\label{DG_0,1_vs_linear_derivations}
Consider the functor $\oplus [\bullet] : \Gr^{\{0,1\}}\Mod_{R} \rightarrow \Mod_{R}$ obtained by composing the inclusion $\Gr^{\{0,1\}} \Mod_{R} \hookrightarrow \Gr^{\geq 0}\Mod_{R}$ with the shear functor $[\bullet] : \Gr^{\{0,1\}}\Mod_{R} \rightarrow \Gr^{\{0,1\}}\Mod_{R}$ followed by the direct sum totalisaion $\oplus: \Gr^{\geq 0}\Mod_{R} \rightarrow \Mod_{R}$. In other words, the functor $\oplus [\bullet]: \Gr^{\{0,1\}}\Mod_{R} \rightarrow \Mod_{R}$ sends $(X^{0},X^{1}) $ to $X^{0} \oplus X^{1}[1]$. The lemma below show that this functor is lax symmetric monoidal. The observation of Remark \ref{free_alg_=sqzero_alg} implies that $\DG_{-}^{\{0,1\}}\Mod_{R}$ sits in pull-back of $\infty$-categories

\begin{equation}\label{DG_0,1_vs_splittings}
\xymatrix{\DG^{\{0,1\}}_{-}\Mod_{R} \ar[d] \ar[r]& \LEq\bigl( \ev^{0}, \oplus [\bullet] )\ar[d] \\
\Mod_{R} \ar[r]& \Mod_{R}^{\B\mathbb{N}} ,}
\end{equation}
where $\LEq\bigl( \ev^{0}, \oplus [\bullet] )$ is the $\infty$-category of two-step graded objects $(X^{0},X^{1})$ endowed with a map $ \delta: X^{0} \rightarrow X^{0} \oplus X^{1}[1]$, the vertical functor sends $(X^{0},X^{1}, \delta)$ to $X^{0}$ endowed with the composite map $X^{0} \rightarrow X^{0}\oplus X^{1} [1] \rightarrow X^{0}$, and the horizontal functors considers any object of $\Mod_{R}$ with identity map. 

The lax equalizer in the upper right corner of diagram \ref{DG_0,1_vs_splittings} has a symmetric monoidal structure by Proposition \ref{Lax_equalizer_equivalence}, and all functors involved in the diagram are symmetric monoidal. Consequently, the pull-back diagram \ref{DG_0,1_vs_splittings} is a pull-back of symmetric monoidal $\infty$-categories.
\end{construction}

\begin{Lemma}
The functor $\oplus [\bullet]: \Gr^{\{0,1\}}\Mod_{R} \rightarrow \Gr^{\{0,1\}}\Mod_{R}$ has a lax symmetric monoidal structure.
\end{Lemma}

\begin{proof}
The functor $\oplus: \Gr^{\{0,1\}}\Mod_{R} \rightarrow \Mod_{R}$ is a composition of the lax symmetric monoidal inclusion $\Gr^{\{0,1\}}\Mod_{R} \hookrightarrow \Gr\Mod_{R}$ followed by the direct sum totalisaion functor $\oplus: \Gr^{\geq 0}\Mod_{R} \rightarrow \Mod_{R}$. The latter has a symmetric monoidal structure as we can identify it with the forgetful functor from the $\infty$-category of $\mathbb{G}_{m}$-representations in $\Mod_{R}$ to $\Mod_{R}$. Therefore, the composition has a lax symmetric monoidal structure.

It remains to check that the shear $[\bullet]: \Gr^{\{0,1\}}\Mod_{R} \rightarrow \Gr^{\{0,1\}}\Mod_{R}$ has a symmetric monoidal structure. Recall from Construction \ref{shear_1} that the shear on all non-negatively graded modules $[\bullet]: \Gr^{\geq 0}\Mod_{R} \rightarrow \Gr^{\geq 0}\Mod_{R}$ is lax symmetric monoidal with respect to Koszul sign rule symmetric monoidal structure on the target. But the symmetric monoidal structure on $\Gr^{\{0,1\}}\Mod_{R}$ induced from the Koszul sign rule symmetric monoidal structure on $\Gr^{\geq 0}\Mod_{R}$ is equivalent to the one induced by the usual symmetric monoidal structure on $\Gr^{\geq 0}\Mod_{R}$, which is seen in the initial case $R=\mathbb{Z}$ by identifying the tensor product in the heart of the negative $t$-structure.
\end{proof}

\begin{rem}
Let $R$ be a connective derived ring. Using Construction \ref{DG_0,1_vs_linear_derivations} and Proposition \ref{limits_in_Cat_lax}, and the same logic as in Proposition \ref{two_step_algebras_pullback} and Remark \ref{two_step_algebras_pullback_nondiscrete}, it follows that $\DG_{-}^{\{0,1\}}\DAlg_{R}$ is a pull-back of $\infty$-categories

\begin{equation}\label{DG_0,1_algebras_via_derivations}
\xymatrix{\DG^{\{0,1\}}_{-}\DAlg_{R} \ar[d] \ar[r]& \LEq(\ev^{0}, \oplus [\bullet]: \DAlg_{R}\Mod \rightarrow \DAlg_{R}) \ar[d] \\
\DAlg_{R} \ar[r]& \DAlg^{\B\mathbb{N}}_{R} .}    
\end{equation}

Moreover, we have that the functor $\xymatrix{\DAlg_{R} \Mod \simeq \Gr^{\{0,1\}}\DAlg_{R} \ar[r]^-{\oplus [\bullet]}& \DAlg_{R}}$ is equivalent to the functor sending $(A,M)$ to the trivial square-zero extension $A\oplus M[1]$. The implication of the pull-back diagram \ref{DG_0,1_algebras_via_derivations} is that the data of an object $A^{\bullet} \in \DG_{-}^{\{0,1\}}\DAlg_{R}$ is the same as the data of an object $(A,M) \in \DAlg_{R}\Mod$ together with a derived algebra map $A \rightarrow A \oplus M[1]$ splitting the projection.
\end{rem}

We will now define an appropriate $\infty$-category of derived algebras endowed with a derivation which gives an equivalent description of the $\infty$-category $\DG_{-}^{\{0,1\}}\Mod_{R}$.

\begin{construction}\label{DAlgDer}

Let $A$ be a derived commutative algebra in $\Mod_{R}$, and $M$ an $A$-module in $\Mod_{R}$. A \textbf{derivation of $A$ with values in $M[1]$} is the data of a map $\LL_{A/R} \rightarrow M[1]$, where $\LL_{A/R}$ is the cotangent complex of $A$ over $R$. Let $\DAlg_{R}\Mod^{\Delta^{1}}$ be the $\infty$-category of pairs $(A,f:M \rightarrow N)$, where $M,N \in \Mod_{A}$ and $f: M\rightarrow N$ is a map. Let $\ev^{0}: \DAlg_{R}\Mod^{\Delta{1}} \rightarrow \DAlg_{R}\Mod$ be the functor sending $A,f:M \rightarrow N)$ to the pair $(A,M)$, and $\LL[-1]: \DAlg_{R}\rightarrow \DAlg_{R}\Mod $ be the functor sending $A$ to $(A, \LL_{A}[-1])$. The \textbf{$\infty$-category of (shifted) derivations} is defined as the pull-back

$$
\xymatrix{  \DAlg_{R}\Der^{[1]} \ar[d] \ar[r]&   \DAlg_{R}\Mod^{\Delta^{1}}\ar[d]^-{\ev^{0}}\\
\DAlg_{R} \ar[r]_-{\LL[-1]} & \DAlg_{R} \Mod  .}
$$

\end{construction}

\begin{rem}\label{derivation_as_splittings}
Equivalently, a derivation of $A$ with values in $M[1]$ is a map of derived commutative $R$-algebras $\delta: A\rightarrow A\oplus M[1]$ which splits the projection map $A\oplus M[1] \rightarrow A $. The equivalence of this two notions of derivation follows from the natural equivalence of mapping spaces

$$
\Map_{\Mod_{A}}(\LL_{A},M[1]) \simeq \Map_{\DAlg_{R//A}}( A, A\oplus M[1]).
$$
\end{rem}

Remark \ref{derivation_as_splittings} and Construction \ref{DG_0,1_algebras_via_derivations} imply the following proposition.

\begin{prop}\label{dg_and_der}
There is an equivalence of $\infty$-categories

$$
\xymatrix{   \DG_{-}^{\{0,1\}}\DAlg_{R} \ar[r]^-{\sim} & \DAlg_{R} \Der^{[1]}.     }
$$
\end{prop}

The importance of Proposition \ref{dg_and_der} stems from the following construction.

\begin{construction}\label{cotangent_adjunction}
There is an adjunction

$$
\xymatrix{  \DAlg_{R}   \ar@/^1.1pc/[rrr]^-{\LL[-1]} &&& \ar@/^1.1pc/[lll]^-{ A  \leftarrow  (A,M,d) }  \DAlg_{R} \Der^{[1]} ,}
$$
where the left adjoint sends a derived commutative algebra $A$ to the universal derivation $(A, \LL_{A}[-1], \Id: \LL_{A} \rightarrow \LL_{A})$.

Using the equivalence $\DAlg_{R}\Der^{[1]}\simeq \DG_{-}^{\{0,1\}} \DAlg_{R}$, we obtain that the left adjoint of the functor $\ev^{0}: \DG^{\{0,1\}}\DAlg_{R} \rightarrow \DAlg_{R}$ sends a derived commutative algebra $A$ to the object $A\oplus \LL_{A/R}[-1]$ endowed with $\mathbb{D}_{-}$-action coming from the universal derivation $d: A \rightarrow \LL_{A/R}$.
\end{construction}

Summarizing all the results we obtained in this section, we arrive at the following theorem.

\begin{Theor}\label{sqzero_as_der_1}
There is an equivalence of $\infty$-categories

$$
\xymatrix{\gr: \DAlg_{R}^{\Delta^{1}_{\vee},\sqzero} \ar[r]^-{\sim} & \DAlg_{R} \Der^{[1]}       ,  }
$$
such that the underlying $\DG^{\{0,1\}}\Mod_{R}$-object of $\gr(I \rightarrow A)$ is the triple $(A/I, I, \delta)$, where $\delta: A/I \rightarrow I[1]$ is the boundary map for the fiber sequence $I \rightarrow A \rightarrow A/I$.
\end{Theor}

\begin{proof}
This follows from Propositions \ref{two_step_fils_via_differential_1}, \ref{Smith_sqzer_vs_two_step} and \ref{dg_and_der}.
\end{proof}

\begin{rem}\label{from_dg_to_derivations}
Let us give a more explicit description of the functors involved in the equivalence of Theorem \ref{sqzero_as_der_1}. Let $(B,M,\delta) \in \DAlg_{R}\Der^{[1]}$. We define a new derived commutative algebra with a map to $B$ as a pull-back

$$
\xymatrix{  A \ar[d] \ar[r] &B \ar[d]^-{(\Id, \delta)} \\
B \ar[r]_-{(\Id,0)}& B\oplus M[1]. }
$$
We have an equivalence $\fib(A \rightarrow B) \simeq M$, moreover this is an equivalence of non-unital $A$-algebras, where $M$ is endowed with trivial non-unital algebra structure. Therefore, we can consider $ \{ ... \rightarrow 0 \rightarrow M \rightarrow A\}$ as a filtered objects concentrated in degrees $[0,1]$.

The inverse functor $ \Fil^{[0,1]}\DAlg_{R} \rightarrow \DAlg_{R}\Der^{[1]}$ can be described as follows. Assume $(I \rightarrow A) \in \Fil^{[0,1]}\DAlg_{R}. $ The left $A$-module structure on $I$ canonically descends to an $A/I$-module structure as all multiplication maps $I^{\otimes^{n}} \rightarrow I$ are naturally null-homotopic for degree reasons. Consider the tensor product $A/I \underset{A}{\otimes} A/I$. The unit $A/I \rightarrow \LSym_{A/I}(I)$ induces a map 

$$
\xymatrix{  A/I \underset{A}{\otimes} A/I \ar[r]& A/I \underset{\LSym_{A/I}(I) }{\otimes} A/I  \ar[r]^-{\sim} & \LSym_{A/I}(I[1]) \ar[r] & A/I \oplus I[1],   }
$$
which defines two different $A/I$-algebra structures on $A/I \oplus I[1]$, coming from the left and the right action of $A/I$ on $A/I \underset{A}{\otimes} A/I$. The right action gives a non-trivial derivation $\delta: A/I \rightarrow A/I \oplus I[1]$ such that the map $A/I \rightarrow I[1]$ is the extension class of the two-step filtration $I \rightarrow A$. 
\end{rem}

We will now unravel some relations between the cotangent complex of the map $A\rightarrow A/I$ and the Smith ideal $I \rightarrow A$. 

\begin{construction}\label{L_and_I}
We will also consider an adjunction

$$
\xymatrix{   \DAlg_{R}^{\Delta^{1}}    \ar@/^1.1pc/[rrr]^-{\LL[-1]} &&& \ar@/^1.1pc/[lll]^-{\triv[1]} \DAlg_{R}\Mod ,    }
$$
where the right adjoint is the functor of \textbf{trivial square-zero extension} $(A,M) \longmapsto (A \rightarrow A\oplus M[1])$, and the left adjoint sends an arrow $(A\rightarrow B)$ to the pair $(B, \LL_{B/A}[-1]) \in\DAlg_{R}\Mod$. By means of the equivalence $\DAlg_{R}^{\Delta^{1}_{\vee}} \simeq \DAlg_{R}^{\Delta^{1}}$, we obtain a functor $ \xymatrix{\LL[-1]\circ \cofib: \DAlg_{R}^{\Delta^{1}_{\vee}} \simeq \DAlg_{R}^{\Delta^{1}} \ar[r]^-{\LL} & \DAlg_{R}\Mod }$. The right adjoint of $\LL[-1]\circ \cofib $ identifies with the composite functor $\xymatrix{  \DAlg_{R}\Mod \ar[r]_-{\triv[1]} & \DAlg_{R}^{\Delta^{1}} \ar[r]_-{\fib}^-{\sim} & \DAlg_{R}^{\Delta^{1}_{\vee}}  } $ which sends a pair $(A,M)$ to the Smith ideal $$\xymatrix{ \bigl( \fib( A \ar[r]^-{(\Id,0)}& A\oplus M) \rightarrow A\bigr) \simeq \bigl( M[-1] \ar[r]^-{0}& A\bigr) ,}$$
and this functor is equivalent to the functor $$\xymatrix{   \DAlg_{R}\Mod \simeq \Gr^{\{0,1\}}\DAlg_{R} \ar[rr]^-{\spl^{[0,1]}}&& \DAlg_{R}^{\Delta^{1}_{\vee}}     } .$$

 \end{construction}
 
 This construction makes transparent the folliowing proposition:

\begin{prop}
Let $(I\rightarrow A)$ be a Smith ideal corresponding to a map $A\rightarrow B \simeq A/I$. There is an equivalence of graded derived algebras

$$
\gr^{\bullet} \adic  (I\rightarrow A) \simeq \LSym^{\bullet}_{B}(\LL_{(B/A} [-1]).
$$
\end{prop}

\begin{proof}
Indeed, it follows from commutativity of the diagram 

$$
\xymatrix{  \DAlg_{R}^{\Delta^{1}_{\vee}} && \ar[ll]_-{\ev^{[0,1]}} \Fil^{\geq 0}\DAlg_{R} \\
\DAlg_{R} \Mod \ar[u]^-{\spl^{[0,1]}} && \ar[ll]^-{\ev^{[0,1]}} \Gr^{\geq 0}\DAlg_{R} \ar[u]_-{\spl}  }
$$
by passing to left adjoints, and using Construction \ref{L_and_I}.
\end{proof}

\begin{rem}

Using the same notations as before, we let Let $I^{n}:=\Fil^{n} \adic (I\rightarrow A)$ be the $n$-th stage of the divided power adic filtration corresponding to this ideal. Then we have fiber sequences
 
 $$
I^{n+1} \rightarrow I^{n} \rightarrow \LSym^{n}_{B} (\LL_{B/A}[-1]).
 $$
  In particular, for $n=1$, we get $\LL_{B/A}  \simeq \cofib(I^{2} \rightarrow I) [1]$, and we get a derivation $\LL_{B/A} \rightarrow I[1]$.

We also have an equivalence of two-step filtered objects $$\adic(I \rightarrow A)^{\leq 1} \simeq \Bigl(\frac{I}{I^{2}} \rightarrow \frac{A}{I^{2}}\Bigr),$$
which implies an equivalence of graded objects $$\gr^{\{0,1\}}\; \adic(I\rightarrow A) \simeq \Bigl(\frac{I}{I^{2} } , \frac{A}{I^{2} } / \frac{I}{I^{2} }\Bigr)\simeq \Bigl(\frac{I}{I^{2} }, \frac{A}{I} \Bigr)  ,$$
and the $\mathbb{D}_{-}$-action $A/I \rightarrow I/I^{2}[1] \simeq \LL_{B/A}$ matches with the universal derivation for the map $A\rightarrow B$.
 \end{rem}
 
 \begin{ex}
 Let us specify to the case when $I\rightarrow A$ is a regular ideal in a discrete algebra $A$. A standard argument shows that the cotangent complex $\LL_{B /A}\simeq I/I^{2}[1]$, and hence is a shift of the compact projective $A/I$-module $I/I^{2}\simeq I \otimes_{A} A/I$. It follows by induction from the fiber sequences
 
 $$
 I^{n+1} \rightarrow I^{n} \rightarrow \Sym^{n}_{A/I} (I/I^{2})
 $$
 that all $A$-modules $I^{n}$ are discrete, and therefore the object $\adic(I \rightarrow A)$ identifies with the classical adic filtration of the ideal $I$. Taking divided power envelopes, we get that $\adic^{\pd} \; \Env^{\pd}(I \rightarrow A)$ is the divided power adic filtration on the divided power envelope $\Env^{\pd}(I \rightarrow A)$.
 \end{ex}

In the same way, we have:

\begin{prop}\label{gr_of_adic_1}
For any $I\rightarrow A \in \DAlg_{R}^{\Delta^{1}_{\vee},\pd}$ with $B=A/I$, there is an equivalence of graded derived divided power algebras

$$
\gr^{\bullet} \adic_{\pd}  (I\rightarrow A) \simeq \LGamma^{\bullet}_{B}(\LL_{B/A} [-1]).
$$
\end{prop}

\section{Derived De Rham complex and derived crystalline cohomology.}

\subsection{Divided powers and derived De Rham complex.}

We consider the $\infty$-category $\Fil^{\geq 0} \DAlg^{\pd}_{R}$ introduced in Definition \ref{fil_pds}, and the $\infty$-category $\widehat{\Fil^{\geq 0}\DAlg^{\pd}}_{R} $ consisting of complete objects. We also let the $\infty$-category of \textbf{differential (non-negatively) graded derived divided power algebras} $\DG^{\geq 0}_{-}\DAlg^{\pd}$ be defined as the pull-back

$$
\xymatrix{  \DG^{\geq 0}_{-}\DAlg_{R}^{\pd} \ar[d] \ar[r]& \widehat{\Fil^{\geq 0}\DAlg^{\pd}}_{R} \ar[d]\\
\DG_{-}^{\geq 0}\Mod_{R} \ar[r]^-{\sim}_-{\gr}&  \widehat{\Fil^{\geq 0}\Mod_{R}}.}
$$
The functor $\ev^{0}: \DG_{-}^{\geq 0} \DAlg_{R}^{\pd} \rightarrow \DAlg_{R} $ preserves all limits, and therefore has a left adjoint. Consequently, so does the functor $\gr^{0}: \Fil^{\geq 0}\DAlg^{\pd}_{R} \rightarrow \DAlg_{R}$. 

\begin{defn}

We will define two notions of derived De Rham cohomology, one in terms of filtered derived divided power algebras, and another in terms of differential graded derived divided power algebras. 

\begin{enumerate}

 \item We define the \textbf{Hodge-filtered derived De Rham cohomology} $$\xymatrix{\LOmega^{\geq \star}_{-/R}: \DAlg_{R} \ar[rr] && \Fil^{\geq 0} \DAlg^{\pd}_{R}}$$  as the left adjoint of the functor $\gr^{0}$. The composition with the completion functor $\Fil^{\geq 0}\DAlg^{\pd}_{R} \rightarrow \widehat{\Fil^{\geq 0}\DAlg^{\pd}}_{R} $ is the \textbf{Hodge completed derived De Rham cohomology} $\widehat{\LOmega}^{\geq \star}_{-/R}: \DAlg_{R} \rightarrow \widehat{\Fil^{\geq 0}\DAlg^{\pd}}_{R}. $

 \item We also define the \textbf{derived De Rham complex} $$\xymatrix{\LOmega^{\bullet}_{-/R} : \DAlg_{R} \ar[rr]&& \DG^{\geq 0}_{-}\DAlg_{R}^{\pd}}  $$  as the composition of $\LOmega^{\geq \star}$ with the functor $\gr: \Fil^{\geq 0} \Mod_{R} \rightarrow \DG_{-}^{\geq 0}\Mod_{R}$. Alternatively, it is the left adjoint of the functor $\ev^{0}: \DG_{-}^{\geq 0} \DAlg_{R}^{\pd} \rightarrow \DAlg_{R}$.
 
 \item Letting $R$ vary, we obtain the functor of \textbf{relative Hodge-filtered derived De Rham cohomology} on the $\infty$-category of arrows $$\xymatrix{  \LOmega^{\geq \star }_{-/-} : \DAlg^{\Delta^{1}}_{\mathbb{Z}} \ar[rr]&& \Fil^{\geq 0} \DAlg^{\pd}_{\mathbb{Z}}   , } $$
 
 as well as its graded cousin, the \textbf{relative derived De Rham complex} $ \LOmega^{\bullet}_{-/-}$, and the complete version \textbf{relative Hodge completed derived De Rham cohomology} $ \widehat{\LOmega}_{-/-}^{\geq \star}$.

\end{enumerate}

\end{defn}

\begin{rem}\label{adic_and_deRham}
The composition 

$$
\xymatrix{  \DAlg_{\mathbb{Z}}^{\Delta^{1}} \ar[r]_-{\sim}^-{\fib} & \DAlg_{\mathbb{Z}}^{\Delta^{1}_{\vee}}  \ar[r]^-{\Env^{\pd}} & \DAlg_{\mathbb{Z}}^{\Delta^{1}_{\vee},\pd} \ar[r]^-{\adic_{\pd}} & \Fil^{\geq 0} \DAlg_{\mathbb{Z}}^{\pd}         }
$$
is equivalent to $\LOmega^{\geq \star}_{-/-}$. Indeed, it suffices to check that the right adjoints of both functors match. The right adjoint of this composition sends $F^{\geq \star}A$ to the arrow $A \rightarrow A/F^{1}A \simeq \gr^{0}A$. But the right adjoint of $\LOmega^{\geq \star}_{-/-}$ also sends $F^{\ast}A \longmapsto (A\rightarrow \gr^{0}A) $, and the claim follows by uniqueness of the left adjoint. 
\end{rem}

\begin{ex}
Let $A\rightarrow A/I$ be a surjective arrow whose kernel is an ideal generated by a \emph{regular} sequence $I=(t_{1},t_{2},...,t_{n})$. Then it follows from Remark \ref{adic_and_deRham} that the derived De Rham cohomology $\LOmega^{\geq \star}_{(A/I) /A}$ gets identified with the algebra $A\langle x_{1}, x_{2},... , x_{n} \rangle / (x_{1}-t_{1}, ..., x_{n}-t_{n})$ endowed with the divided power filtration.
\end{ex}

\begin{Theor}\label{properties_of_LOmega}

For any derived commutative $R$-algebra $A$, there are natural equivalences:
\begin{itemize}
\item $ \LOmega^{\bullet}_{A/R} \simeq \LGamma^{\bullet}_{A}(\LL_{A/R}[-1]) $ of graded derived divided power algebras;
\item $\LOmega^{\{0,1\}}_{A/R} \simeq A \oplus \LL_{A/R}[-1](1)$ of differential two-step graded derived algebras, with the $\mathbb{D}_{-}$-action on the latter given by the universal derivation $d: A \rightarrow \LL_{A/R}$.

\end{itemize}
\end{Theor}

\begin{proof}
Let $I = \fib(R \rightarrow A)$. The the first equivalence follows from Remark \ref{adic_and_deRham} and Propositions \ref{gr_of_adic_1}:

$$
\LOmega_{A/R}^{\bullet} \simeq \gr\adic_{\pd} \Env^{\pd}  (R,I) \simeq \LGamma^{\bullet}_{A}(\LL_{A/R}[-1]).
$$

For the second equivalence, consider the universal property of the object $\LOmega^{\{0,1\}}_{A/R}$. For a fixed $R$, the functor $\LOmega^{\{0,1\}}_{-/R}$ is the composition of left adjoints

$$
\xymatrix{   \DAlg_{R} \ar[rr]^-{\LOmega^{\bullet}_{-/R}} &&    \DG^{\geq 0}_{-} \DAlg^{\pd}_{R} \ar[rr]^-{(-)^{\{0,1\}}  } &&   \DG^{\{0,1\}}_{-}\DAlg^{\pd}_{R} \simeq \DAlg_{R}\Der^{[1]}.  }
$$
The right adjoint of the composition is the functor that forgets the derivation. Therefore, the composite left adjoint is given by the functor $\LL_{-/R}[-1]$ of Construction \ref{cotangent_adjunction}.
\end{proof}

\begin{cor}\label{DeRham_of_smooth}
Let $R\rightarrow A$ be a smooth map of discrete commutative algebras. Then the derived De Rham complex $\LOmega_{A/R}^{\bullet}$ is discrete in the negative $t$-structure on $\DG_{-}\Mod_{R}$.
\end{cor}

\begin{proof}
Indeed, this follows from the décalage isomorphism of Proposition \ref{décalage_LSym}: $\LGamma^{n}_{A} (\LL_{A/R}[-1]) \simeq \LLambda^{n}_{A}(\LL_{A/R})[-n]$. The cotangent complex is equivalent to the projective $A$-module of Kahler differentials $\LL_{A/R} \simeq \Omega^{1}_{A/R}$. Therefore, the derived exterior powers $\LLambda^{n}_{A}(\Omega^{1}_{A/R})$ are discrete too, and the $n$-th graded piece of the De Rham complex, $\LGamma^{n}_{A}(\LL_{A}[-1]) \simeq \Lambda^{n}_{A} (\Omega^{1}_{A/R})[-n]$ sits in homological degree $-n$.
\end{proof}

To relate our derived De Rham complex with the classical De Rham complex in the smooth case, we use the \emph{shear equivalence}.

\begin{prop}
Let $R\rightarrow A$ be a discrete map of smooth commutative rings. Then the image $\LOmega_{A/R}^{\bullet} [\bullet] $ under the shear equivalence is a discrete commutative algebra in $\Gr\Mod_{R,\heartsuit}$ (with respect to the Koszul sign rule symmetric monoidal structure). There is an equivalence $$\Omega_{A/R}^{\bullet}[\bullet] \simeq \Lambda_{A}^{\bullet}(\Omega^{1}_{A/R}),$$ and the $\mathbb{D}_{-}$-action on $\LOmega_{A/R}^{\bullet}$ translates into the De Rham differential on $\Lambda^{\bullet}_{A}(\Omega^{1}_{A/R})$.
\end{prop}

\begin{proof}
Indeed, by Corollary \ref{DeRham_of_smooth}, properties 1) and 2) of the shear equivalence from Construction \ref{shear}, the algebra $  \Omega^{\bullet}_{A/R}[\bullet] $ is discrete in the neutral $t$-structure on $\Gr\Mod_{R}$ and has the structure of a graded supercommutative algebra. In weight $1$, we have $(\Omega_{A/R}^{\bullet}[\bullet])^{1} \simeq \Omega^{1}_{A/R}$. Therefore, we have a map from the free graded supercommutative algebras $\Lambda^{\bullet}_{A}(\Omega_{A/R}^{1}) \rightarrow \LOmega_{A/R}^{\bullet}[\bullet]$. By décalage isomorphism, the map  $\Lambda^{n}(\Omega^{1}_{A/R}) \rightarrow \LGamma^{n}(\Omega_{A/R}^{1}[-1])[\bullet]$  is an equivalence for any $n$ (see the proof of Proposition \ref{décalage_LSym}), and therefore, the map $\Lambda^{\bullet}_{A}(\Omega_{A/R}^{1}) \rightarrow \LOmega_{A/R}^{\bullet}[\bullet]$ is an equivalence. The claim about De Rham differential follows from property 3) of Construction \ref{shear}, second part of Theorem \ref{properties_of_LOmega} saying that the $\mathbb{D}_{-}$-action on $\Omega^{\bullet}_{A/R}$ is the universal derivation $d: A \rightarrow \Omega^{1}_{A/R}$ in degree $1$, and the classical fact that De Rham differential on $\Lambda^{\bullet}_{A}(\Omega_{A/R}^{1})$ is the unique square-zero map which restricts to the universal derivation $d: A \rightarrow \Omega_{A/R}^{1}$ in degree $1$.
\end{proof}

\subsection{Cosimplicial techniques.}

In this subsection we will revisit some well-known cosimplicial techniques for computing derived De Rham cohomology. Similar formulas were obtained by similar methods in the paper \cite{Bh12(1)}, albeit formulated in a different language. These methods will be used for the theorem of comparison of the Hodge completed derived De Rham cohomology with the Hodge completed stacky De Rham cohomology cohomology in the paper \cite{Magidson2}.

\begin{construction}
Let $R\rightarrow A$ be a map in $\DAlg_{\mathbb{Z}}$. Then the Bar construction of $A$ over $R$ is the cosimplicial diagram:

$$
\xymatrix{ \overset{\bullet}{ \underset{R}{\bigotimes}} \;A := \bigl (  A  \ar@<-.5ex>[r] \ar@<.5ex>[r] & A\otimes_{R} A \ar@<-1.0ex>[r] \ar[r]  \ar@<1.0ex>[r] & A\otimes_{R} A\otimes_{R} A  \ar@<-1.5ex>[r] \ar@<-0.5ex>[r] \ar@<0.5ex>[r]  \ar@<1.5ex>[r] & ... \bigr)   }
$$

Let us apply this construction to the map $\LOmega_{A/R} \rightarrow A. $ We obtain a cosimplicial diagram 

$$
\xymatrix{  \overset{\bullet}{ \underset{\LOmega^{\geq \star}_{A/R}}{\bigotimes}} A := \bigl (  A  \ar@<-.5ex>[r] \ar@<.5ex>[r] & A\underset{\LOmega^{\geq \star}_{A/R}}{\bigotimes} A \ar@<-1.0ex>[r] \ar[r]  \ar@<1.0ex>[r] &A\underset{\LOmega^{\geq \star}_{A/R}}{\bigotimes} A\underset{\LOmega^{\geq \star}_{A/R}}{\bigotimes} A  \ar@<-1.5ex>[r] \ar@<-0.5ex>[r] \ar@<0.5ex>[r]  \ar@<1.5ex>[r] & ... \bigr)   }
$$
of filtered derived $R$-algebras. Taking $\gr^{0}$, we obtain for any $n$:

$$
\gr^{0} (   \underset{\LOmega^{\geq \star}_{A/R}}{\bigotimes^{n}} A) \simeq \underset{\gr^{0} \LOmega^{\star}_{A/R}}{\bigotimes^{n}} A \simeq \underset{A}{\bigotimes^{n}} \;A \simeq A,
$$
and using commutativity of $\gr^{0}$ with limits it follows that $$\gr^{0}\circ \Tot ( \underset{\LOmega^{\geq \star}_{A/R}}{\bigotimes^{\bullet}} A  ) \simeq \Tot \circ \gr^{0} ( \underset{\LOmega^{\geq \star}_{A/R}}{\bigotimes^{\bullet}} A) \simeq A.$$ In particular, it follows that the fiber of the map $\Tot ( \overset{\bullet}{ \underset{\LOmega^{\geq \star}_{A/R}}{\bigotimes}} A ) \rightarrow A $ has a natural divided power structure.  By universal property of Hodge filtered derived De Rham cohomology, we obtain a filtered divided power map

$$
 \epsilon:  \LOmega^{\geq \star}_{A/R} \rightarrow \Tot (\underset{\LOmega^{\geq \star}_{A/R}}{\bigotimes^{\bullet}} A   )   
$$

\end{construction}

\begin{Lemma}
The map $\epsilon $ is an equivalence on completions.
\end{Lemma}

\begin{proof}

Passing to associated gradeds, it suffices to show that the map

\begin{equation}\label{totalization_map}
\LGamma_{A}(\LL_{A/R}[-1])   \xymatrix{\ar[r]&} \Tot (\underset{\LGamma_{A}(\LL_{A/R}[-1])}{\bigotimes^{\bullet}} A    )
\end{equation}
arising as the Cech nerve of the map $\LGamma_{A}(\LL_{A/R}[-1]) \rightarrow A$ is an equivalence. There is an equivalence 

$$
\underset{\LGamma_{A}(\LL_{A/R}[-1])}{\bigotimes^{n}} A   \simeq \LGamma_{A}(\LL_{A/R}^{\oplus^{n}}),
$$
and we can identify the map \ref{totalization_map} with the map obtained by applying the functor $\LGamma_{A/R}(-)$ to the cosimplicial diagram:

$$
\LL_{A/R}[-1] \simeq \xymatrix{ \Tot\bigl (  0  \ar@<-.5ex>[r] \ar@<.5ex>[r] & \LL_{A/R} \ar@<-1.0ex>[r] \ar[r]  \ar@<1.0ex>[r] & \LL_{A/R}^{\oplus^{2}} \ar@<-1.5ex>[r] \ar@<-0.5ex>[r] \ar@<0.5ex>[r]  \ar@<1.5ex>[r] & ... \bigr)   .}
$$

Consequently, to prove that the map \ref{totalization_map} is a graded equivalence, it suffices to show that the map $$\LGamma^{n}_{A/R}(\LL_{A/R}[-1]) \simeq \LGamma^{n}_{A/R}\bigl( \Tot( \LL_{A/R}^{\oplus^{\bullet}}  ) \bigr) \xymatrix{  \ar[r]&  }   \Tot\bigl(  \LGamma_{A/R}^{n}(\LL_{A/R}^{\oplus^{\bullet}})    \bigr)$$
is an equivalence for any $n$. This follows by commutativity of the derived divided power functors with finite totalizations, see \cite[Proposition 3.37]{BM19}.
\end{proof}

\begin{Lemma}\label{tensor_formula_for_DeRham}
For any $n\geq 1$, there is an equivalence

$$
\underset{\LOmega^{\geq \star}_{A/R}}{\bigotimes^{n}} A \simeq  \LOmega^{\geq \star}_{A/ \bigotimes^{n} A},
$$
where $A$ is considered as algebra over the tensor product $\bigotimes^{n} A$ via the multiplication map $\bigotimes^{n} A\rightarrow A$. 
\end{Lemma}

\begin{proof}
Indeed, more generally for a map of the form $A\underset{R}{\otimes}B \rightarrow C$, there is  a natural map

$$
\xymatrix{ \LOmega^{\geq \star}_{C/A} \underset{\LOmega^{\geq \star}_{C/R}}{\otimes} \LOmega^{\star}_{C/B} \ar[r]& \LOmega^{\geq \star}_{C/A\underset{R}{\otimes} B}    }
$$
which is an equivalence. This follows from commutativity with colimits of the relative De Rham cohomology functor $\LOmega_{-/-}: \DAlg^{\Delta^{1}}_{R} \rightarrow \Fil^{\geq 0}\DAlg^{\pd}_{R}$:

$$\LOmega^{\geq \star}_{A\underset{C}{\otimes} B \rightarrow R} \simeq \LOmega^{\geq \star}_{A \rightarrow R \underset{C\rightarrow R}{\otimes} B \rightarrow R} \simeq \LOmega^{\geq \star}_{A\rightarrow R} \underset{\LOmega^{\geq \star}_{C\rightarrow R} }{\otimes} \LOmega^{\geq \star}_{B \rightarrow R} .$$ Applying this equivalence to the map $A\otimes_{R} A \rightarrow A $ and using $\LOmega_{A/A} \simeq A$, we get the claim by induction.
\end{proof}

Let $R \rightarrow A$ be a map of derived commutative rings, and denote $$I_{n}:=\fib( A\otimes_{R}A\otimes ... \otimes_{R}A \rightarrow A)$$ the fiber of the $n$-times multiplication map. The collection of all $I_{n}$ for all $n$ organize into a cosimplicial Smith ideal in the cosimplicial algerba $A^{\otimes^{\bullet}_{R}}$. Let $\widehat{\Env}^{\pd}_{\otimes^{n}_{R} A} (I_{n})$ be the completed divided power envelope of $\otimes_{R}^{n}A$ at the ideal $I_{n}$. There are equivalences $\widehat{\LOmega}_{A/\otimes_{R}^{n}A} \simeq \widehat{\Env}^{\pd}_{\otimes^{n}_{R} A} (I_{n}) $ for each $n$, commuting with cosimplicial maps. This implies the following:

\begin{prop}\label{totalisaion_formula}
There is an equivalence of complete filtered derived divided power algebras

$$
\xymatrix{\widehat{\LOmega}^{\geq \star}_{A/R}\simeq \Tot \bigl (  A  \ar@<-.5ex>[r] \ar@<.5ex>[r] & \widehat{\Env}^{\pd}_{\otimes^{2}_{R} A}(I_{2}) \ar@<-1.0ex>[r] \ar[r]  \ar@<1.0ex>[r] & \widehat{\Env}^{\pd}_{\otimes^{3}_{R} A}(I_{3})  \ar@<-1.5ex>[r] \ar@<-0.5ex>[r] \ar@<0.5ex>[r]  \ar@<1.5ex>[r] & ... \bigr)}
$$
\end{prop}

\begin{rem}
Note that when $R\rightarrow A$ is a map of connective derived commutative rings, then all the ideals $I_{n}$ are connective, and hence so are the divided power envelopes $\Env^{\pd}_{\otimes^{2}_{R} A}(I_{n})$. The formula of Proposition \ref{totalisaion_formula} tells that the (non-connective general) object $\widehat{\LOmega}_{A/R}^{\geq \star}$ can be computed as a totalisaion of connective objects.
\end{rem}

\subsection{Derived crystalline cohomology.}

As we have seen in the previous subsection, derived De Rham cohomology of $A$ can be thought as the universal filtered divided power thickening of $A$. Crystalline cohomology is also a universal divided power thickening of $A$, but with one crucial difference: in crystalline cohomology theory we require compatibility of divided powers with divided powers on $\mathbb{Z}_{p}$. More generally, crystalline cohomology can be defined relatively to any divided power algebra $I \rightarrow A$: given a derived $A/I$-algebra $R$, there exists a universal derived divided power $(I \rightarrow A)$-algebra $\R\Gamma_{\crys}(R/A)$ which has $R$ as the zeroth graded piece. While the idea is classical, we learnt from B.Antieau a formulation of the universal property of the crystalline cohomology theory in the spirit of the paper of A.Raksit \cite{R}.

\begin{defn}\label{pd_crys}
Let $(I \rightarrow A)$ be a connective derived divided power algebra, and denote $\overline{A}:=A/I$. Consider the $\infty$-category $\DAlg^{\Delta^{1},\pd}_{A}$ of derived divided power maps $(R\rightarrow \overline{R})$ under $(A \rightarrow \overline{A})$. There is a functor $\ev^{0}: \DAlg^{\Delta^{1},\pd}_{A} \rightarrow \DAlg_{\overline{A}}$ sending an object $(R \rightarrow \overline{R})$ to the derived $\overline{A}$-algebra $\overline{R}\in \DAlg_{\overline{A}}$. This functor preserves all limits, and hence admits a right adjoint. We define \textbf{derived crystalline cohomology (relative to $(A\rightarrow \overline{A})$}:

$$
\R\Gamma_{\crys}(-/(A\rightarrow \overline{A})): \xymatrix{ \DAlg_{\overline{A}} \ar[r]& \DAlg_{A}^{\Delta^{1},\pd}    }
$$
as the left adjoint of $\ev^{0}$. 

This definition can be given in families, i.e. letting the base $(I \rightarrow A)$ vary. We define the \textbf{absolute crystalline site} $\mathbb{Z}^{\crys}$ as the $\infty$-category of pairs $$\mathbb{Z}^{\crys}:=\{\:\:(A \rightarrow \overline{A}, R) \:\: | \: (A \rightarrow \overline{A}) \in \DAlg^{\Delta^{1},\pd}_{\mathbb{Z}}\:\:, \:\:R \in \DAlg_{\overline{A}}\}. $$ The functor $\ev^{0}: \DAlg^{\Delta^{1},\pd}_{\mathbb{Z}} \rightarrow \mathbb{Z}^{\crys}$ sending a derived divided power map $(A \rightarrow \overline{A})$ to the pair $(A\rightarrow \overline{A}, \overline{A})$ admits a left adjoint

$$
\R\Gamma_{\crys}: \xymatrix{ \mathbb{Z}^{\crys} \ar[r]& \DAlg_{\mathbb{Z}}^{\Delta^{1},\pd}    }
$$
whose value on the pair $(A\rightarrow \overline{A}, R)$ is $\R\Gamma_{\crys}(R/(A \rightarrow \overline{A}))$. When the base divided power ideal is clear from the definition, we simply denote it as $\R\Gamma_{\crys}(R/A)$.
\end{defn}

\begin{ex}\label{identity}
Let $A$ be any derived ring, and consider the identity map $(\Id: A \rightarrow A)$ as the base derived divided power algebra. Then the $\infty$-category $\DAlg_{(A\rightarrow A)/}^{\Delta^{1},\pd} $ is equivalent to the $\infty$-category $\DAlg_{A}^{\Delta^{1},\pd}$ of derived divided power maps over $A$, and the functor $\R\Gamma_{\crys}(-/(\Id: A \rightarrow A))$ identifies with the derived De Rham cohomology functor $\LOmega_{-/A}: \DAlg_{A} \rightarrow \DAlg^{\Delta^{1},\pd}_{A}$.
\end{ex}

\begin{ex}
Let $A=\mathbb{Z}/(p^{n})$ endowed with the pd-ideal $(p) \subset \mathbb{Z}_{p}/(p^{n})$ so that $\overline{\mathbb{Z}/(p^{n})}=\mathbb{F}_{p}$. In this case, for a discrete $\mathbb{F}_{p}$-algebra $R$, the derived divided power algebra $\R\Gamma_{\crys}(R/\mathbb{Z}/(p^{n}))$ is related to the (mod $p^{n}$) classical crystalline cohomology theory of $\mathbb{F}_{p}$-algebras (the precise statement will be explained below). In particular, taking the limit as $n$ goes to infinity, we obtain a $p$-complete derived divided power algebra

$$
\R\Gamma^{\wedge}_{\crys}(R/\mathbb{Z}_{p})=\underset{\leftarrow}{\lim} \:\R\Gamma_{\crys}(R/\mathbb{Z}/(p^{n})) \in \DAlg_{\mathbb{Z}_{p}}^{\Delta^{1},\pd,\wedge}.
$$
This should recover the classical crystalline cohomology of $R$ relative to the ring of $p$-adic integers.
\end{ex}

Let us now discuss what properties we expect derived crystalline cohomology theory to satisfy. First, given a pd base $(A \rightarrow \overline{A})$ and a derived $\overline{A}$-algebra $R$, the base change $\R\Gamma_{\crys}(R/(A \rightarrow \overline{A}))\otimes_{A} \overline{A}$ should be equivalent to the derived De Rham cohomology $\LOmega_{R/\overline{A}}$.

Another property is that for a derived $\overline{A}$-algebra $\overline{R}$, whenever we have a lift $R$ of $\overline{R}$ to a derived $A$-algebra, we can compute the crystalline cohomology of $\overline{R}$ relative to $A$ in terms as the derived De Rham cohomology of $R$. 

To formulate the last property, let us introduce some notation. Let $(A \rightarrow \overline{A})$ be a base derived divided power map, and $\R\Gamma_{\crys}(-/(A\rightarrow \overline{A}))$ be as in Definition \ref{pd_crys}. Let $$(\mathcal{A} \rightarrow \overline{A}):=\Env^{\pd}(A \rightarrow \overline{A}) \simeq \LOmega_{\overline{A}/A}$$ be the derived divided power envelope of $(A\rightarrow \overline{A})$. By adjunction, the pd structure on $(A \rightarrow \overline{A})$ gives a natural pd map $( \mathcal{A} \rightarrow \overline{A} ) \rightarrow (A \rightarrow \overline{A})$. Denote $\LOmega_{-/A}: \DAlg_{\overline{A}} \rightarrow \DAlg_{\mathcal{A}}^{\Delta^{1},\pd}$ be the derived De Rham cohomology functor of derived $\overline{A}$-algebras, taken relatively to $A$ and endowed with a natural $\mathcal{A}$-linear structure. Then the base change $\LOmega_{R/A} \otimes_{\mathcal{A}} A$ should compute derived crystalline cohomology $\R\Gamma_{\crys}(R/(A\rightarrow \overline{A}))$ for any derived $\overline{A}$-algebra $R$.

All these properties are almost immediate in our setting, as we will see now. The main ingredient is the following base-change property.

\begin{prop}\label{base_change}
Assume $(A \rightarrow \overline{A}) \rightarrow (B \rightarrow \overline{B})$ is a map of derived divided power bases. Then the following diagram commutes

$$
\xymatrix{   \DAlg_{\overline{A}} \ar[d]_-{-\otimes_{\overline{A}} \overline{B}   }  \ar[rrr]^-{\R\Gamma_{\crys}(-/(A \rightarrow \overline{A})) }&&&  \DAlg_{(A\rightarrow \overline{A})/}^{\Delta^{1},\pd} \ar[d]^-{-\otimes_{(A\rightarrow \overline{A})} (B \rightarrow \overline{B})    }\\
\DAlg_{\overline{B}} \ar[rrr]_-{\R\Gamma_{\crys}(-/(B \rightarrow \overline{B})) }&&& \DAlg_{(B \rightarrow \overline{B})/}^{\Delta^{1},\pd}.      }
$$
\end{prop}

\begin{proof}
Follows by passing to right adjoints.
\end{proof}

The following statement follows from the base-change property and summarizes the main computational properties of derived crystalline cohomology.

\begin{cor}\label{DeRham_crys_1} 

Assume $(A \rightarrow \overline{A})$ is a base derived divided power algebra, $I=\fib(A\rightarrow \overline{A})$ and $(\mathcal{I} \rightarrow \mathcal{A})$ as defined above.

\begin{enumerate}

\item The composite functor $$\xymatrix{ \DAlg_{\overline{A}} \ar[rr]^-{\R \Gamma_{\crys}(-/(A\rightarrow \overline{A}))} && \DAlg^{\Delta^{1},\pd}_{(A\rightarrow \overline{A})/} \ar[rr]^-{-\otimes_{A} \overline{A}} && \DAlg^{\Delta^{1},\pd}_{\overline{A}}   } $$

is equivalent to derived De Rham cohomology $\LOmega_{-/\overline{A}}$ relative to $\overline{A}$.

\item There is a commutative diagram of functors

$$
\xymatrix{ \DAlg_{\overline{A}} \ar[dd]_-{\LOmega_{-/A}} \ar[rr]^-{\R \Gamma_{\crys}(-/A)} && \DAlg^{\Delta^{1},\pd}_{A}\\\\
 \DAlg^{\Delta^{1},\pd}_{\mathcal{A}} \ar[rruu]_-{-\otimes_{\mathcal{A}} A}   .} 
$$

\item Consider the map $(0 \rightarrow A) \rightarrow (I \rightarrow A)$ of derived divided power algebras. There is a commutative triangle of functors

$$
\xymatrix{ \DAlg_{A} \ar[d]_-{\LOmega_{-/A}} \ar[rr]^-{-\underset{A}{\otimes} \overline{A}} && \DAlg_{\overline{A}} \ar[d]^-{\R \Gamma_{\crys}(-/(A\rightarrow \overline{A}))}\\
 \DAlg^{\Delta^{1},\pd}_{(\Id: A \rightarrow A)/} \ar[rr]_-{-\underset{(0\rightarrow A)}{\otimes} ( A \rightarrow \overline{A})} &&  \DAlg^{\Delta^{1},\pd}_{(A \rightarrow \overline{A})/}   .} 
$$ 

\end{enumerate}
\end{cor}

\begin{proof}
For the first, apply Proposition \ref{base_change} and Example \ref{identity} for the map $(A \rightarrow \overline{A}) \rightarrow (\Id: \overline{A} \rightarrow \overline{A})$. For the second, use the map $( \mathcal{A} \rightarrow \overline{A} ) \rightarrow (A \rightarrow \overline{A})$. For the third, the map $(0 \rightarrow A) \rightarrow (I \rightarrow A)$.
\end{proof}

\begin{rem}

Let us unwind the third statement of Corollary. Let $R$ be a derived $A$-algebra and $\overline{R}:=R\otimes_{A} \overline{A}$. Consider the  derived De Rham complex of $R$ taken relatively to $A$, $\LOmega_{R/A}$. This has the structure of a derived divided power algebra whose zeroth graded piece is $R$. Now endow $\LOmega^{\wedge}_{R/A}$ with a different derived pd ideal defined as the fiber of the composite map $$\LOmega_{R/A} \rightarrow R \rightarrow \overline{R} . $$ This gives a derived divided power $(A \rightarrow \overline{A})$-algebra structure on $\LOmega_{R/A}$, and the assignment $R \longmapsto \LOmega_{R/A} $ determines a functor $\LOmega_{-/A}: \DAlg_{A} \rightarrow \DAlg^{\Delta^{1},\pd}_{(A\rightarrow \overline{A})}$. Then the statement is that there is an equivalence of derived divided power $(A\rightarrow \overline{A})$-algebras

$$
\Omega_{R/A} \simeq  \R\Gamma_{\crys}(\overline{R}/(A\rightarrow \overline{A})).
$$

\end{rem}

\begin{ex}\label{DeRham_crys_2}
Let $\overline{A}[x]$ be the polynomial algebra on one generator over $\overline{A}$. It has a essentially unique lift $A[x]$ over $A$. It follows from  the third part of Corollary \ref{DeRham_crys_1} that the crystalline cohomology of $\overline{A}[x]$ is equivalent to the De Rham cohomology of $A[x]$:

$$
\R\Gamma_{\crys}(\overline{A}[x]/(A\rightarrow \overline{A})) \simeq \LOmega_{A[x]/A}.
$$
\end{ex}

\begin{rem}
Let $\mathbb{Z}^{\crys}_{\geq 0} \subset \mathbb{Z}^{\crys}$ be the full subcategory consisting of pairs $(A \rightarrow \overline{A},R)$ where all derived rings are connective. Using Proposition \ref{connective_and_Mao}, this $\infty$-category is equivalent to the \textbf{derived crystalline context}, the notion defined in the paper \cite{Mao21}. In particular, it is equivalent to the sifted completion of a small category $\mathbb{Z}^{\crys,\poly}$ consisting of pairs $(A \rightarrow \overline{A}, R)$ where the map $(A\rightarrow \overline{A})$ has the form $\mathbb{Z}\langle x_{1},...,x_{n}\rangle \otimes \mathbb{Z}[y_{1},...,y_{m}] \rightarrow \mathbb{Z}[y_{1},...,y_{m}]$ and $R$ is a polynomial $\overline{A}$-algebra. Zhouhang Mao defined derived crystalline cohomology as a functor $\mathbb{Z}^{\crys}_{\geq 0} \rightarrow \DAlg_{\mathbb{Z}} $\footnote{to be more precise, he defines a functor with values in $\mathbb{E}_{\infty}$-rings, but it is clear from the construction that it lifts to a functor with values in derived rings} as the left Kan extension of the classical crystalline cohomology functor $\mathbb{Z}^{\crys,\poly} \rightarrow \DAlg_{\mathbb{Z}}$. The resulting functor commutes with all colimits. As the classical crystalline cohomology of polynomial algebras can be computed as the De Rham cohomology of a essentially unique lift, it follows from Example \ref{DeRham_crys_2} that our approach to derived crystalline cohomology recovers the approach of Zhouhang Mao.
\end{rem}

In the rest of this subsection, we treat in more details the particular case of the pd base $(p) \subset \mathbb{Z}_{p}$. In particular, we will explain how to compute the crystalline cohomology using a lift to $\mathbb{Z}_{p}$, and relate derived crystalline cohomology relative to $\mathbb{Z}_{p}$ with the formula for the De Rham stack in characteristic $p$.

By abuse of notation, let us denote $\R\Gamma^{\wedge}_{\crys}(-/\mathbb{Z}_{p}): \DAlg_{\mathbb{F}_{p}} \rightarrow \DAlg_{\mathbb{Z}_{p}}$ the functor sending a derived commutative $\mathbb{F}_{p}$-algebra $A$ to its $p$-completed derived crystalline cohomology understood merely as a derived commutative $\mathbb{Z}_{p}$-algebra. In other words, this is the composition

$$
\xymatrix{  \DAlg_{\mathbb{F}_{p}} \ar[rr]^-{\R\Gamma_{\crys}^{\wedge}(-/\mathbb{Z}_{p})} && \DAlg_{\mathbb{Z}_{p}}^{\Delta^{1},\pd,\wedge} \ar[rr]^-{\ev^{0}}&& \DAlg^{\wedge}_{\mathbb{Z}_{p}}.    }
$$

Below we will explain how to compute the right adjoint of this functor.

\begin{defn}\label{pdred}
Recall that the right adjoint of the forgetful functor $\ev^{0}: \DAlg^{\Delta^{1}_{\vee},\pd}_{\mathbb{Z}} \rightarrow \DAlg_{\mathbb{Z}}$ is the functor sending a derived commutative ring $R$ to the derived divided power Smith ideal $(R^{\sharp}\rightarrow R)$, where $R^{\sharp}$ is the pd radical of $R$. For any derived commutative algebra $R$, we define the \textbf{pd-reduction} of $R$ by the formula $R_{\pdred}:=R/R^{\sharp}$.
\end{defn}

\begin{prop}
There is an adjunction 

$$
\xymatrix{ \DAlg_{\mathbb{F}_{p}} \ar@/^1.1pc/[rrr]^-{\R\Gamma^{\wedge}_{\crys}} &&&   \ar@/^1.1pc/[lll]^{(-)_{\pdred}}  \DAlg^{\wedge}_{\mathbb{Z}_{p}}    }
$$
\end{prop}

\begin{rem}
Let $\Spec(\R\Gamma^{\wedge}_{\crys}(A/\mathbb{Z}_{p})): \DAlg^{\wedge}_{\mathbb{Z}_{p},\geq 0} \rightarrow \Spc$ be the spectrum of the non-connective derived ring $\R\Gamma^{\wedge}_{\crys}(A/\mathbb{Z}_{p}))$. Then the functor of points formula for $\Spec(\R\Gamma^{\wedge}_{\crys}(A/\mathbb{Z}_{p}))$ is

$$
\Spec(\R\Gamma^{\wedge}_{\crys}(A/\mathbb{Z}_{p})) (R) \simeq  \Map_{\DAlg^{\wedge}_{\mathbb{Z}_{p}}}(\R\Gamma_{\crys}^{\wedge}(A/\mathbb{Z}_{p}),R)  \simeq  \Map_{\DAlg_{\mathbb{F}_{p}}}(A,R_{\pdred}).
$$
It follows that $\Spec(\R\Gamma^{\wedge}_{\crys}(A/\mathbb{Z}_{p})) \simeq X_{\crys},$ where crystallization of $X=\Spec A$ is as defined in \cite{BL22(2)} and \cite{Dr20}.
\end{rem}


\begin{bibdiv}
\addcontentsline{toc}{section}{\protect\numberline{}References}

\begin{biblist}


\bib{Bh12(1)}{article}{ 
author={Bhargav Bhatt},
title={Completions and Derived De Rham cohomology},
eprint={https://arxiv.org/abs/1207.6193 },
label = {Bh12(1)}
}

\bib{Bh12(2)}{article}{ 
author={Bhargav Bhatt},
title={p-adic derived De Rham cohomology},
eprint={https://arxiv.org/pdf/1204.6560.pdf},
label = {Bh12(2)}
}

\bib{Bh23}{article}{
title={Prismatic $F$-gauges},
author={Bhargav Bhatt},
eprint={https://www.math.ias.edu/~bhatt/teaching/mat549f22/lectures.pdf }
label = {Bh23}
}

\bib{BCN21}{article}{
title = {PD Operads and Explicit Partition Lie Algebras},
author = {Lukas Brantner},
author = {Ricardo Campos}
author = {Joost Nuiten}
eprint = {https://arxiv.org/abs/2104.03870}
label = {BCN21} }

\bib{BI22}{article}{
      author={Robert R. Bruner},
      author={Daniel C. Isacsen},
       title={Jeff Smith's theory of ideals},
      eprint={https://arxiv.org/abs/2208.07941},
      label ={BI22}
}



\bib{BL22(1)}{article}{
title={Absolute prismatic cohomology},
author={Bhargav Bhatt},
author={Jacob Lurie},
eprint={ https://arxiv.org/abs/2201.06120}
label = {BL22(1)}
}

\bib{BL22(2)}{article}{
title={Prismatization of $p$-adic formal schemes},
author={Bhargav Bhatt},
author={Jacob Lurie},
eprint={https://arxiv.org/abs/2201.06124}
label = {BL22(2)}
}

\bib{BM19}{article}{
title = {Deformation theory and partition Lie algebras},
author = {Lukas Brantner},
author = {Akhil Mathew},
eprint = { https://arxiv.org/abs/1904.07352}
label = {BM19}
}




\bib{BS22}{article}{ 
author={Bhargav Bhatt},
author={Peter Scholze}
title={Prisms and prismatic cohomology},
journal={Annals of Mathematics, Volume 196 (2022), Issue 3},
series={Princeton University&The Institute of Advanced Studies}
label = {BS22}}



\bib{Dr21}{article}{
      author={Vladimir Drinfeld},
       title={On the notion of ring groupoid},
      eprint={https://arxiv.org/abs/2104.07090},
      label = {Dr21}
}







\bib{Dr20}{article}{
title={Prismatization},
author={Vladimir Drinfeld},
eprint={https://arxiv.org/abs/2005.04746} 
label = {Dr20}   }




\bib{Fr}{article}{
title =  {On homotopy of simplicial algebras over an operad}, 
author = {Benoit Fresse},
journal = {Transactions of the American Mathematical Society, Vol. 352, No. 9},
year = {2010}
label = {Fr10}
}

\bib{FrG12}{article}{
title =  {Chiral Koszul duality}, 
author = {John Francis},
author = {Dennis Gaitsgory}
journal = {Selecta Mathematica, Vol.18},
year = {2012}
label = {FG12}
}



\bib{GaitsRozI}{book}{
   title =     {A Study in Derived Algebraic Geometry, Volume I: Correspondences and Duality},
   author =    {Dennis Gaitsgory},
   author =    {Nick Rozenblyum},
   publisher = {American Mathematical Society},
   isbn =      {1470435691,9781470435691},
   year =      {2017},
   series =    {Mathematical Surveys and Monographs},
   eprint =    {http://www.math.harvard.edu/~gaitsgde/GL/}
   label = {GR17(I)}
}

\bib{GR2}{book}{
   title =     {A Study in Derived Algebraic Geometry, Volume II: Deformations, Lie Theory and Formal Geometry},
   author =    {Dennis Gaitsgory},
   author =    {Nick Rozenblyum},
   publisher = {American Mathematical Society},
   isbn =      {1470435705,9781470435707},
   year =      {2017},
   series =    {Mathematical Surveys and Monographs},
   eprint =    {http://www.math.harvard.edu/~gaitsgde/GL/}
   label = {GR17(II)}}

\bib{Gr}{article}{
title =  {Crystals and De Rham cohomology of schemes}, 
author = {Alexander Grothendieck},
journal = {Advanced studies in pure mathematics, vol. 3},
year = {1968}
label = {Gr68}
}

\bib{H2}{article}{
title = {Derived $\delta$-Rings and Relative Prismatic Cohomology},
author = {Adam Holeman},
eprint = {https://arxiv.org/abs/2303.17447}
label = {H22}
}




\bib{H}{article}{
title = {On Lax transformations, adjunctions and monads in $(\infty,2)$-categories},
author = {Rune Haugseng},
eprint = {https://arxiv.org/abs/2303.17447}
label = {H20}
}



\bib{HA}{article}{
      author={Jacob Lurie},
       title={Higher Algebra},
       date={2017},
      eprint={http://www.math.harvard.edu/~lurie/papers/HA.pdf},
      label = {HA}
}





\bib{HTT}{book}{
   title =     {Higher Topos Theory},
   author =    {Jacob Lurie},
   publisher = {Princeton University Press},
   isbn =      {9780691140490,9781400830558},
   year =      {2009},
   eprint =    {http://www.math.harvard.edu/~lurie/papers/HTT.pdf},
   label = {HTT}
}



\bib{I1}{book}{
   title =     {Complexe Cotangent et Deformations, Vol.1},
   author =    {Luc Illusie},
   publisher = {Springer},
   isbn =      {3540056866},
   year =      {1971}
   label = {Ill71(1)}}
 

\bib{I2}{book}{
   title =     {Complexe Cotangent et Deformations, Vol.2},
   author =    {Luc Illusie},
   publisher = {Springer},
   isbn =      {3540056866},
   year =      {1971}
   label = {Ill71(2)}}
   
   
   \bib{Magidson}{article}{
title={Differential graded Cartier modules and De Rham-Witt complex (in preparation)},
author={Magidson, Kirill},
eprint={}   
label = {dRW}
 }
 
    \bib{Magidson2}{article}{
title={Divided powers and De Rham cohomology II: Geometric part},
author={Magidson, Kirill},
eprint={}   
label = {pdII}
 }
 
 \bib{Mao21}{article}{
      author={Zhouhang Mao},
       title={Revisiting derived crystalline cohomology},
      eprint={https://arxiv.org/abs/2107.02921},
      label = {Mao21}
}

\bib{Maul21}{article}{
title={The geometry of filtrations},
author={Tasos Maulinos},
journals = {Bulletin of the London Mathematical Society, 2021},
eprint={},   
label = {Maul21}
 }
 
 \bib{ML}{book}{
   title =     {Categories for the working mathematician},
   author =    {Sanders Mac Lane}
   publisher = {Springer},
   year =      {1978}
   series = {Graduate Texts in Mathematics (GTM, volume 5)}
   label = {MacL78}}
 
 \bib{L}{article}{
title={Rotation invariance in Algebraic K-theory},
author={Jacob Lurie},
eprint={https://www.math.ias.edu/~lurie/papers/Waldhaus.pdf},   
label = {L14}
 }
 

 
 
 
 \bib{NS}{article}{
title = {On topological cyclic homology},
author = {Thomas Nikolaus},
author = {Peter Scholze},
journal = {Acta Math. 221(2): 203-409},
year = {2018},
label = {NS18}
 }
 
 
 \bib{Pr}{article}{
title = {Thom-Sebastiani and duality for matrix factorizations},
author = {Anatoly Preygel}
eprint = {https://arxiv.org/pdf/1101.5834.pdf}
label = {Pr11}
}


\bib{R}{article}{
title = {Hochshild homology and the derived de Rham complex revisited},
author = {Arpon Raksit}
eprint = {https://arxiv.org/abs/2007.02576 }
label = {R20}
}






 
 
\bib{Stacks}{article}{
      author={Authors},
       title={Stacks Project},
      eprint={https://stacks.math.columbia.edu},
      label = {SP}
}




\end{biblist}
\end{bibdiv}
\end{document}